\documentclass[12pt]{amsart}
\usepackage{pb-diagram}

\usepackage{amscd}
\usepackage{amsfonts}
\usepackage{amsmath}
\usepackage{amssymb}
\usepackage{color}

\newcommand{\tr}{\textnormal{tr}}
\newcommand{\ric}{\textnormal{Ric}}

\newcommand{\Rm}{\textnormal{Rm}}

\newcommand{\ep}{\varepsilon}

\newcommand{\var}{\textnormal{Var}}
\newcommand{\fs}{\textnormal{FS}}

\newcommand{\dbar}{\overline{\partial}}

\newcommand{\ddt}[1]{\frac{\partial #1}{\partial t}}
\newcommand{\dds}[1]{\frac{\partial #1}{\partial s}}

\newcommand{\ddbar}{\sqrt{-1}\partial\dbar}

\newcommand{\cL}{\mathcal{L}}
\newcommand{\cS}{\mathcal{S}}
\newcommand{\cN}{\mathcal{N}}
\newcommand{\cW}{\mathcal{W}}
\newcommand{\cX}{\mathcal{X}}
\newcommand{\cR}{\mathcal{R}}
\newcommand{\vol}{\textnormal{Vol}}
\newcommand{\osc}{\textnormal{osc}}

\newcommand{\CCC}{\mathfrak{C}}
\newcommand{\bF}{\mathbb{F}}

\newtheorem{theorem}{Theorem}[section]
\newtheorem{lemma}[theorem]{Lemma}
\newtheorem{corollary}[theorem]{Corollary}
\newtheorem{proposition}[theorem]{Proposition}

\newtheorem{claim}[theorem]{Claim}

\numberwithin{equation}{section}

\theoremstyle{definition}
\newtheorem{remark}[theorem]{Remark}

\theoremstyle{definition}
\newtheorem{definition}[theorem]{Definition}

\begin{document}

\title{Geometric regularity of  blow-up limits of the K\"ahler-Ricci flow}

\author{Max Hallgren$^*$, Wangjian Jian$^{**}$, Jian Song$^\dagger$ and Gang Tian$^\ddagger$}

\thanks{Max Hallgren is supported in part by National Science Foundation grant DMS-2202980. Wangjian Jian is supported in part by NSFC No.12201610, NSFC No.12288201, MOST No.2021YFA1003100. Jian Song is supported in part by National Science Foundation grant DMS-2203607. Gang Tian is supported in part by NSFC No.11890660, MOST No.2020YFA0712800.}

\address{$^*$ Department of Mathematics, Rutgers University, Piscataway, NJ 08854}

\email{mh1564@math.rutgers.edu}

\address{$^{**}$ Institute of Mathematics, Academy of Mathematics and Systems Science, Chinese Academy of Sciences, Beijing, 100190, China}

\email{wangjian@amss.ac.cn}

\address{$^\dagger$ Department of Mathematics, Rutgers University, Piscataway, NJ 08854}

\email{jiansong@math.rutgers.edu}

\address{$^\ddagger$ BICMR and SMS, Peking University, Beijing 100871, China}

\email{gtian@math.pku.edu.cn}

\maketitle
\thispagestyle{empty}
\markboth{Geometric regularity of the blow-up limits of the K\"ahler-Ricci flow}{Max Hallgren, Wangjian Jian, Jian Song and Gang Tian}

\begin{abstract}  
We establish geometric regularity for Type I blow-up limits of the K\"ahler-Ricci flow based at any sequence of Ricci vertices. As a consequence, the limiting flow is continuous in time in both Gromov-Hausdorff and Gromov-$W_1$ distance. In particular, the singular sets of each time slice and its tangent cones are close and of codimension no less than $4$.
\end{abstract}

\maketitle
{\footnotesize \tableofcontents}


\bigskip
\section{Introduction}

Blow-up analysis is fundamental in the study of the formation of singularities in Ricci flow. Hamilton (cf. \cite{Ham, Ham2}) laid out various approaches to classify both finite time and long singularities of Ricci flow. Type I scaling is the most natural parabolic blow-up that one wishes to extract a geometric limit satisfying the Ricci soliton equations.   In \cite{SeT}, Perelman showed that Type I rescalings of a Fano K\"ahler-Ricci flow have uniformly bounded scalar curvature. Based on this fact, it was shown in \cite{Bam18,CW3,TiZZL} that limits of normalized Fano K\"ahler-Ricci flow have partial regularity akin to noncollapsed limits of Einstein manifolds.  Perelman's estimate was also applied in \cite{Chen} to give a Ricci flow proof of the Yau-Tian-Donaldson conjecture, and in \cite{DerSz,HanLi} to establish the uniqueness of tangent flows in the Fano setting.

However, it is still unknown whether the scalar curvature is bounded for Type I rescalings of general finite time solutions of the K\"ahler-Ricci flow. On the other hand, it was shown in \cite{JST23a} that Type I rescalings of any projective K\"ahler-Ricci flow have locally bounded scalar curvature near certain distinguished points, called Ricci vertices. Moreover, it was shown that the scalar curvature has at most quadratic growth as a function of distance from the Ricci vertex, and that a partial $C^0$ estimate also holds at bounded distance from a Ricci vertex. The position of such Ricci vertices depends on a choice of background $(1,1)$-form associated to the limiting cohomology class at the singular time, which offers substantial flexibility in the singularity analysis of projective K\"ahler-Ricci flows. Nonetheless, it is so far unclear when such blow-up limits would coincide with  tangent flows, hence they could fail to be self-similar. Because of this, the strongest form of Bamler's partial regularity theory \cite{Bam20c} does not apply. In particular, Bamler's theory does not guarantee that each time slice of the limiting metric flow has singularities of codimension four, that the metric flow is continuous, or that the tangent cones of the time slices are metric cones. The primary goal of this paper is to establish these facts.

The continuity of the limiting metric flow is established based on locally uniform estimates for distance distortion. Such estimates for Ricci flows with globally bounded scalar curvature were established in \cite{BZ17,BZ19}, and (with additional assumptions) in \cite{CW3}. Some of the techniques employed in these proofs may be successfully localized, giving a fairly straightforward proof that distances cannot decrease too quickly in time. The reverse estimate, however, requires new ideas as the corresponding proofs in \cite{BZ19,CW3} rely heavily on a global scalar curvature bound. 

We localize the proof of \cite[Section 6]{Bam18} to show that each time slice of the limiting metric flow has singularities of codimension four and is indeed a singular space in the sense of \cite{Bam17}. The partial regularity theory we rely on is from \cite{Bam20c} rather than from \cite{Bam17}, since we choose to consider the $\bF$-limit. For this reason, we have to work on the flow instead of on a single time-slice. The derivative estimate of the Ricci potential in the time direction (which is the Laplacian estimate in the space direction) plays a crucial role in working on the space-time. Such Li-Yau type estimate is established in \cite{JST23a}, and it fits well with Perelman's reduced geometry as the boundedness of the Ricci potential propagated along the reduced geodesic.

Our results are closely related to the analytic minimal model program with Ricci flow proposed in \cite{ST1, ST2, ST4} in the hope that the blow-up limits will reveal both geometric and algebraic structures of underlying algebraic singularities and the associated birational surgeries. New proofs of Perelman's estimates were discovered for the Fano K\"ahler-Ricci flow \cite{JST23b} and were extended to general case in \cite{JST23a}. They provide a refined geometric picture of the analytic minimal model program. This paper is a continuation of \cite{JST23a} and \cite{JST23b} and lays the groundwork for the future study on the formation of singularities of the K\"ahler-Ricci flow.


\medskip
\subsection{Statement of the main results}\label{STofthemain}

We consider the unnormalized K\"ahler-Ricci flow
\begin{equation}\label{unkrflow1}
\left\{
\begin{array}{l}
{ \displaystyle \ddt{g(t)} = -\ric(g(t)) ,}\\
\\
g(0)=g_0, 
\end{array} \right. 
\end{equation}
on a K\"ahler manifold $X$ of complex dimension $n\geq 2$ for some initial K\"ahler metric $g_0 \in H^{1,1}(X, \mathbb{R})\cap H^2(X, \mathbb{Q})$.  

Suppose the flow develops finite time singularities at $T>0$. Kawamata's rationality and base point free theorem imply that $T\in \mathbb{Q}$ and the limiting cohomology class $\vartheta= [\omega_0]+ T[K_X] \in H^{1,1}(X, \mathbb{R})\cap H^2(X, \mathbb{Q})$ is a semi-ample $\mathbb{Q}$-line bundle. In particular,  the semi-ample line bundle $\vartheta $ induces a unique surjective holomorphic map
\begin{equation}
\Phi: X \rightarrow Y \subset \mathbb{CP}^N, 
\end{equation}
where $Y$ is a normal projective variety and $\dim Y$ is equal to the Kodaira dimension of $\vartheta$. We will always assume that $T=1$ after replacing $\omega_0$ by $T^{-1} \omega_0$. 
\begin{enumerate}

    \item When $Y$ is a point, $X$ is a Fano manifold and the K\"ahler-Ricci flow must have finite time extinction. 
    
    \medskip

    \item When $0< m:=\dim_{\mathbb{C}} Y < \dim_{\mathbb{C}} X$, the general fibre of $\Phi: X \rightarrow Y$ is a Fano manifold and such a Fano fibration is also called a Mori fibration. 

    \medskip

    \item When $\dim Y = \dim X$, $\Phi$ is a birational morphism corresponding to a divisorial contraction or a small contraction of a flip. 
   
\end{enumerate}

We let $\theta_Y$ be a smooth closed $(1,1)$-form on $Y$ with 
 \begin{equation}\label{puly}
 \Phi^*\theta_Y\in \vartheta.
 \end{equation}
That is, $\theta_Y$ is the restriction of a local smooth closed $(1,1)$-form  through a local embedding of $Y$ into some $\mathbb{C}^M$. For example, we can choose $\theta_Y$ to be the multiple of the Fubini-Study metric $\mathbb{CP}^N$ restricted to $Y \subset \mathbb{CP}^N$. 
We abuse notation by identifying $\theta_Y$ with $\Phi^*\theta_Y$  for convenience.

For fixed $\theta_Y$, there exists $u\in C^\infty(X \times [0, 1))$ such that
$$\ric(g(t)) - (1-t)^{-1} g(t) = - (1-t)^{-1} \theta_Y  - \ddbar  u.$$
%
In \cite{JST23a}, we define the so-called Ricci vertex in the following way. A point $p$ is said to be a Ricci vertex at $t\in [0, 1)$ associated to $\theta_Y$ if
$$ u(p, t) = \inf_X u(\cdot, t). $$
The goal of this paper is to study the geometric regularity of the Type I blow-up limits of the solution of (\ref{unkrflow1}) around the Ricci vertices.




%
Consider any sequence of times $t_i\nearrow 1$. Let $(M_i, (g_{i,t})_{t\in [-T_i , 0]})$ be the flows arising from $(X, (g(t))_{t\in [0, 1) })$ by setting 
$$ M_i:=X,~~~ g_{i,t}:=(1-t_i)^{-1}g((1-t_i)t+t_i),~~~ t\in [-T_i , 0]. $$
where $T_i=t_i/(1-t_i)\to \infty$ as $i\to\infty$. By \cite{Bam20b}, let $p_i\in X$ be any base-point, by passing to a subsequence, we can obtain $\bF$-convergence on compact time-intervals
\begin{equation}\label{FcorKRF'1}
(M_i, (g_{i,t})_{t\in [-T_i , 0]}, (\nu_{p_i,0; t})_{t\in [-T_i , 0]}) \xrightarrow[i\to\infty]{\bF,\CCC}  (\cX, (\nu_{p_\infty; t})_{t\in (-\infty , 0]}),
\end{equation}
within some correspondence $\CCC$, where $\cX$ is a future continuous and $H_{2n}$-concentrated metric flow of full support over $(-\infty , 0]$.

According to \cite{Bam20c}, we can decompose $\cX$ into its regular and singular part 
\begin{equation}\label{rsd'2}
\cX= \cR \sqcup \cS,
\end{equation}
where $\cR$ is dense open subset of $\cX$. Also, $\cR$ carries the structure of a Ricci flow spacetime $(\cR, \mathfrak{t}, \partial_{\mathfrak{t}}, g)$. For any $t\in (-\infty, 0)$, writing $\cR_t=\cX_t\cap\cR$, we have that $(\cX_t, d_t)$ is the metric completion of $(\cR_t, g_t)$.

The first main result of this paper is the following theorem.

%

%
\begin{theorem}\label{mainforKRF1}
Let $\theta_Y\in \vartheta$ be any smooth closed $(1,1)$-form on $Y$, and let $p_i\in X$ be a Ricci vertex associated with $\theta_Y$ at $t_i$. Then the limiting metric flow $\cX$ is a continuous metric flow on $(-\infty, 0]$, both in the Gromov-Hausdorff sense and the Gromov-$W_1$-sense.
\end{theorem}

The main technical obstacle to proving Theorem \ref{mainforKRF1} is the lack of uniform scalar curvature control. In fact, the proof of the theorem is built on the key gradient estimate we establish for solutions of backwards heat equations, along with techniques developed in \cite{BZ19}. It seems that the validity of such an estimate relies on global information of the flow, particularly on the quadratic growth rate for scalar curvature. To prove Theorem \ref{mainforKRF1},  we first prove a Li-Yau type estimate for the forward heat kernel on the K\"ahler-Ricci flow. This estimate is similar to that of \cite{ZhZh}, but depends on the Ricci potential rather than the scalar curvature. We can further establish a gradient estimate for solutions of the backward heat equation via integral estimates for the gradient, whose weight combines the forwards heat kernel with the twisted Ricci potential. This allows us to establish locally uniform continuity of the distance function in time along the Type I rescaled flow, which can also be passed to the limit.

The following theorem analyzes the geometric regularity of the blow-up limits.

\begin{theorem}\label{mainforKRF2}
Let $\theta_Y\in \vartheta$ be any smooth closed $(1,1)$-form on $Y$, and let $p_i\in X$ be a Ricci vertex associated with $\theta_Y$ at $t_i$. Then the following statements hold for every $t\in (-\infty, 0]$.

\begin{enumerate}
%
\item $(\cX_t, d_t,\cR_t, g_t)$ is a singular space of dimension $2n$ in the sense of Definition \ref{sp}.
 \smallskip

\item The Minkowski dimension for the singular set $\cS_t$ of each time slice is given by  
$$\dim_{\mathcal{M}}\cS_{t}\leq 2n-4.$$

 \smallskip

\item Any tangent cone of $(\mathcal{X}_t,d_t)$ is a metric cone.
%
\end{enumerate}
\end{theorem}

The proof of Theorem \ref{mainforKRF2} is built by combining techniques from \cite{Bam18} with \cite{Bam20c}. An important step is to identify Gromov-Hausdorff convergence of points with convergence of points in the sense defined in \cite{Bam20b}. 

\begin{remark} It follows from the results of \cite{JST23a} that any tangent cone of $(\mathcal{X}_t,d_t)$ at a point $x\in \mathcal{X}_t$ is also a complex analytic variety. In a forthcoming note, we will show that this tangent cone is an affine algebraic variety uniquely and algebraically determined by the germ of the variety $(\mathcal{X}_t,x)$. 
\end{remark}

This paper is organized as follows.

In section \ref{conandpre}, we recall some conventions and notations, as well as some known results to be used in the later sections.

In section \ref{ApptotheKRF}, we derive analytic estimates for the finite time solutions of the K\"ahler-Ricci flow. The local good distance distortion estimates are established based on a new Harnack inequality for the heat equation coupled with the K\"ahler-Ricci flow (see Theorem \ref{HarnackiofKRF}). We further prove Theorem \ref{mainforKRF1} after obtaining Gromov-Hausdorff continuity and Gromov-$W_1$ continuity for the limiting metric flow in time.

In section \ref{prfofthmmain}, we prove Theorem \ref{mainforKRF2} for Ricci flows with scalar curvature controlled by suitable barrier functions. Such barrier functions are natural generalizations of the Ricci potentials from the K\"ahler-Ricci flow.

%



\bigskip
\section{ Conventions and preliminary results }\label{conandpre}


\medskip
\subsection{Notation and conventions}
Let $(M, g(t))_{t\in I}$ be a smooth Ricci flow on a compact $n$-dimensional manifold, where $I\subset \mathbb{R}$ is an interval. For $(x_0, t_0)\in M\times I$, $A, T^-, T^+\geq 0$, the corresponding parabolic neighborhood is defined by
\begin{equation} \label{pnbd}
P(x_0, t_0; A, -T^-, T^+) = B(x_0, t_0, A) \times \left( [t_0- T^-, t_0+ T^+]\cap I\right),
\end{equation}
where we may omit $-T^-$ or $T^+$ if it is zero. For any $r>0$, we set $P(x_0, t_0; r):=P(x_0, t_0; r, -r^2, r^2)$.

The heat operator associated to $(M, g(t))$ is given by
$$\Box = \ddt{} - \Delta$$
and the conjugate heat operator is given by
$$\Box^* = - \ddt{} - \Delta + R, $$
where $\Delta$ is the Laplacian associated to $g(t)$ and $R$ is the scalar curvature of $g(t)$.

For any $(x, t)$, $(y, s) \in M \times I$ with $s\leq t$, we denote by $K(x,t; y,s)$ the heat kernel of the Ricci flow based at $(y,s)$, satisfying
\begin{equation}
\Box K(\cdot, \cdot; y,s)=0, ~~ \lim_{t\rightarrow s^+} K(\cdot, t, ; y, s) = \delta_y,
\end{equation}
where $\delta_y$ is the Dirac measure at $y$. Then, $K(x, t; \cdot,\cdot)$ is the conjugate heat kernel based at $(x,t)$, satisfying
\begin{equation}
\Box^* K(x, t; \cdot, \cdot)=0, ~ \lim_{s\rightarrow t^-} K(x, t; \cdot, s) =\delta_x.
\end{equation}

Using the conjugate heat kernel, we can define the conjugate heat measure $\nu_{x, t; s}$ based at $(x,t)$ by
\begin{equation}
d\nu_{x,t;s} = K(x,t; \cdot, s) dg(t) = (4\pi\tau)^{-n/2} e^{-f} dg(t),
\end{equation}
where $\tau=t-s$ and $f\in C^\infty(M \times (-\infty, t))$ is called the potential of the conjugate heat measure $\nu_{x,t;s}$.

For two probability measures $\mu_1$ and $\mu_2$ on a Riemannian manifold $(M, g)$, the Wasserstein $W_1$-distance between $\mu_1$ and $\mu_2$ is defined by
\begin{equation}
d^g_{W_1}(\mu_1, \mu_2) = \sup_f \left( \int_M f d\mu_1 - \int_M f d\mu_2 \right), 
\end{equation}
where the supremum is taken over all bounded $1$-Lipschitz function on $(M, g)$. The variance between $\mu_1$ and $\mu_2$ is defined by
\begin{equation}
\var(\mu_1, \mu_2) = \int_{(x_1, x_2) \in M \times M}  d_g^2(x_1, x_2) d\mu_1(x_1) d\mu_2(x_2).
\end{equation}
Then we have the following basic relation between the Wasserstein $W_1$-distance and the variance
\begin{equation}
d^g_{W_1} (\mu_1, \mu_2) \leq \sqrt{\var(\mu_1, \mu_2)}.
\end{equation}

For any $(x_0, t_0)\in M\times I$, $A, T^-, T^+\geq 0$, we define the $P^*$-parabolic neighborhood 
\begin{equation} \label{p*nbd}
P^*(x_0, t_0; A, -T^-, T^+) \subset M\times I ,
\end{equation}
as the set of $(x,t)\in M\times I$ with $t\in [t_0-T^-, t_0+ T^+]$ and
\begin{equation} \label{p*nbd2}
d^{g_{t_0-T^-}}_{W_1}( \nu_{x_0, t_0; t_0-T^-} , \nu_{x, t; t_0-T^-} ) < A .
\end{equation}
As before, we may omit $-T^-$ or $T^+$ if it is zero. We also define the $P^*$-parabolic $r$-ball by $P^*(x_0, t_0; r):=P^*(x_0, t_0; r, -r^2, r^2)$.
Next we define the $H_n$-center at a base point along the Ricci flow.
\begin{definition}\label{hnc}
A point $(z, t)\in M\times I$ is called an $H_n$-center of a point $(x_0, t_0)\in M \times I$ if $t\leq t_0$ and
\begin{equation}
\var_t(\delta_z, \nu_{x_0, t_0; t}) \leq H_n (t_0 - t),
\end{equation}
where $\var_t$ is the variance with respect to the metric $g(t)$.

\end{definition}
Immediately,  if $(z, t)$ is an $H_n$-center of $(x_0, t_0)$, then we have
\begin{equation}\label{dvh}
d^{g_t}_{W_1}(\delta_z, \nu_{x_0, t_0; t} ) \leq \sqrt{\var(\delta_z, \nu_{x_0, t_0; t} )} \leq \sqrt{H_n(t_0-t)}.
\end{equation}
The following lemma is proved in \cite{Bam20a}, which asserts that the mass of the conjugate heat kernel measure will concentrate around the $H_n$-centers.
\begin{lemma} \label{hnc'2}
If the point $(z, t)$ is an $H_n$-center of $(x_0, t_0)$ with $t< t_0$, then for any $A>0$, we have
$$\nu_{x_0, t_0; t} \left( B\left( z, t, \sqrt{AH_n (t_0-t)} \right) \right) \geq 1- \frac{1}{A} .$$
\end{lemma}

We now define the Nash entropy introduced by Hein-Naber \cite{HN}. Let $d\nu= (4\pi \tau)^{-n/2} e^{-f} dg$ be a probability measure on a closed $n$-dimensional Riemannian manifold $(M, g)$ with $\tau>0$ and $f\in C^\infty(M)$. The Nash entropy is defined by
\begin{equation}
\cN[g, f, \tau] = \int_M f d\nu - \frac{n}{2}.
\end{equation}

Writing $d\nu_{x_0, t_0; t} = (4\pi \tau)^{-n/2} e^{-f(t)} dg(t)$, where $\tau = t_0 - t\geq 0$, 
 we define the pointed Nash entropy based at $(x_0, t_0)$ by 
\begin{equation}
\cN_{x_0, t_0}(\tau) = \cN [ g(t_0-\tau), f(t_0-\tau), \tau].
\end{equation}
We also set
$$\cN_{x_0, t_0}(0) =0,$$ 
which makes $\cN_{x_0, t_0}(\tau)$ continuous at $\tau=0$. We also define
\begin{equation}
\cN_s^*(x_0, t_0)= \cN_{x_0, t_0}(t_0-s),
\end{equation}
for $s< t_0$ and $s\in I$. The pointed Nash entropy $\cN_{x_0, t_0}(\tau)$ is non-increasing as a function of $\tau\geq 0$.

By \cite{Bam20a}, the pointed Nash entropy $\cN_s^*$ has bounded oscillation in any $P^*$-neighborhood. To be more precise, if $P=P^*(x_0, t_0; A, -T^-, T^+)$ and $T^-<t_0-s$, then we have
\begin{equation}\label{oscofNashinPstar}
\osc_{P}\cN^*_s\leq 2\left(  \frac{n}{2(t_0-s-T^-)-R_{\min}}  \right)^{1/2}A + \frac{n}{2}\ln \left(  \frac{t_0-s+T^+}{t_0-s-T^-}  \right) ,
\end{equation}
where $R_{\min}$ denotes the lower scalar curvature bound.
For any compact, n-dimensional manifold $(M, g)$, Perelman's $\cW$-functional is defined by, for any $\tau > 0$,
$$\cW[g, f, \tau] = (4\pi \tau)^{-n/2} \int_M \left(\tau ( |\nabla f|^2 + R) + f -n \right) e^{-f} dg,$$
with $f\in C^\infty(M)$ so that $\int_M (4\pi \tau)^{-n/2} e^{-f} dg =1$, and Perelman's $\mu$-functional and $\nu$-functional are defined by
$$\mu[g, \tau]=\underset{\int_M (4\pi \tau)^{-n/2} e^{-f} dg =1}{\inf} \cW[g, f, \tau],$$
and 
$$\nu[g, \tau]=\underset{0<\tau'<\tau}{\inf} \mu[g, \tau].$$
If $(M, (g_t)_{t\in [0, T)})$ is a Ricci flow, then the functions $t\to \mu[g_t, T-t]$ and $t\to \nu[g_t, T-t]$ are non-decreasing. It is proved in \cite{Bam20a} that 
\begin{equation}\label{NM}
\cN^*_t (x_0, t_0) \geq \mu[g(t), t_0-t],  
\end{equation}
for any $t<t_0$. 

Next, we define
\begin{definition}[Curvature Radius]\label{cs}
For any $(x,t)\in M\times I$, we define the curvature radius at $(x, t)$ as follows:
$$r_{\Rm}(x,t):=\sup \left\{ r>0: |\Rm|\leq r^{-2} ~~ on ~~ P( x , t ; r) \right\}.$$
\end{definition}
Then we have the following lemma.
\begin{lemma}\label{deofrrm}
For any $\alpha>0$, there exists $C(n, \alpha)<\infty$, such that the following statement holds.

Let $(M, g(t))_{t\in I}$ be a smooth Ricci flow on a compact $n$-dimensional manifold with the interval $I\subset \mathbb{R}$. Assume $[a-\alpha, b]\subset I$, then in the weak sense we have
\begin{enumerate}
    \item $|\nabla r_{\Rm}|\leq 1$ on each time-slice $M\times \left\{t\right\}$ for all $t\in [a , b]$;
    \item $|\partial_t r^2_{\Rm}|\leq C(n, \alpha)$ on $M\times [a , b]$.
\end{enumerate}
\end{lemma}
\begin{proof}
Item (1) is clear from definition. Item (2) is from \cite[Lemma 6.1]{BZ17}.
\end{proof}
%


\medskip
\subsection{Entropy and heat kernel bounds}
In \cite{Bam20a}, Bamler established systematic results on the Nash entropy and heat kernel bounds on a Ricci flow background. Let us recall some results that will be used in our theory.

The following quantitative volume estimates are established in \cite{Bam20a}.

\begin{lemma} \label{lvnc}
Let $(M, g(t))_{t\in [-r^2, 0]}$ be a solution of the Ricci flow. If 
$$R\leq r^{-2}, ~ on~ B_{g(0)}(x, r) \times [-r^2, 0],$$
then
\begin{equation}
\vol_{g(0)} (B_{g(0)}(x, r)) \geq c e^{\cN^*_{-r^2}(x, 0)} r^n.
\end{equation}

\end{lemma}

The assumption on the scalar curvature upper bound can be replaced using the $H_n$-center as proved in \cite{Bam20a}.

\begin{lemma} \label{hcv}  
Let $(M, g(t))_{t\in [-r^2, 0]}$ be a solution of the Ricci flow. Suppose $(z, -r^2)$ is an $H_n$-center of $(x_0, 0)$ and 
$$R(\cdot, -r^2) \geq R_{min},$$
for some fixed $R_{min} \in \mathbb{R}$. Then there exists $c= c(R_{min} r^2)>0$ such that 
\begin{equation}
\vol_{g(-r^2)}(B_{g(-r^2)}(z, (2H_n)^{1/2} r)) \geq c e^{\cN^*_{-r^2}(x, 0)} r^n.
\end{equation}

\end{lemma}

Next, we have the following heat kernel upper bound, which is proved in \cite{Bam20a} (Theorem 7.2).
\begin{lemma} 
Let $(M, g(t))_{t\in I}$ be a solution of the Ricci flow. Suppose that on $M \times [s, t]$,
$$[s, t]\subset I, ~ R\geq R_{min}.$$
Let $(z, s) \in M\times I$ be an $H_n$-center of $(x, t)\in M\times I$. Then there exist $C=C(R_{min}(t-s))<\infty$, such that for any $y\in M$, we have
\begin{equation}
K(x, t; y,s) \leq C (t-s)^{-n/2} e^{-\cN^*_{s}(x, t)} e^{ - \frac{ d^2_s(z, y)}{C (t-s)}}.
\end{equation}

\end{lemma}

Using this heat kernel upper bound estimate, we have the following estimate which relates the $W_1$-distance to the $\mathcal{L}$-length. This estimate was used in \cite{J} to prove an improved version of the volume non-collapsing estimate. 
\begin{lemma} \label{wdn} 
Let $(M, g(t))_{t\in (-T, 0)}$ be a solution of the Ricci flow for some $T>0$. Suppose $(s, t) \subset (-T, 0)$ with $s>-T+\epsilon>0$ for some $\epsilon>0$.  Let  $\gamma: [0, t-s] \rightarrow M \times (-T, 0)$ be a $C^1$ spacetime curve with 
$$\gamma(\tau) \in M \times \{ t-\tau\}, ~~\gamma(0)=(x,t) ~\gamma(t-s) = (y,s). $$
Then there exists $C=C(\epsilon)>0$ such that
\begin{equation}
d^{g(s)}_{W_1} (\delta_{y, s}, \nu_{x,t; s}) \leq C \left(1+ \frac{\cL(\gamma)}{2(t-s)^{1/2}} -\cN^*_s(x, t)  \right) ^{1/2} (t-s)^{1/2}.
\end{equation}
\end{lemma} 

Next, we have the following Lemma.
\begin{lemma}\label{W1balongwl}
Let $(M, g(t))_{t\in I}$ be a smooth Ricci flow on a compact $n$-dimensional manifold with the interval $I\subset \mathbb{R}$. Assume that $(x_0, t_0)\in M\times I$, $r_0\leq 1$ satisfy that $[t_0-2r_0^2, t_0]\subset I$, $R(x_0, t)\leq Yr_0^{-2}$ for all $t\in [t_0-r_0^2, t_0]$, and $\cN^*_{t_0-r_0^2}(x_0, t_0)\geq -Y$, then we have
$$d^{g(t_0-r_0^2)}_{W_1} (\nu_{x_0, t_0; t_0-r_0^2} , \delta_{x_0} )\leq Cr_0$$
for some constant $C=C(n, Y)<\infty$.
\end{lemma}
\begin{proof}
After parabolic rescaling, we may assume without loss of generality that $r_0=1$. Consider the spacetime curve defined by $\gamma(\tau)=(x_0, t_0-\tau)$ for $\tau\in [0, 1]$, then we have
$$\cL(\gamma)= \int_{0}^{1} \tau^{1/2}R(x_0, t_0-\tau) d\tau \leq 3Y,$$
Hence we can apply Lemma \ref{wdn} to obtain
\[
\begin{split}
&d^{g(t_0-1)}_{W_1} (\nu_{x_0, t_0; t_0-1} , \delta_{x_0} )  \\
\leq& C \left(1+ \frac{\cL(\gamma)}{2(t_0- (t_0-1))^{\frac{1}{2}}} - \cN^*_{t_0-1}(x_0, t_0) \right) ^{1/2} (t_0- (t_0-1))^{1/2}\\
\leq& C.   
\end{split}
\]
This completes the proof.
\end{proof}

Finally, let us recall the following result, which was proved by Perelman in \cite{Per1}.

\begin{lemma} \label{lcenter} 
Let $(M, g(t))_{t\in (-T, 0)}$ be a solution of the Ricci flow for some $T>0$. Suppose $[s, t] \subset (-T, 0)$. Then for any $x\in M$, there exists a point $y\in M$, such that
$$ \ell_{(x, t)} (y, s) \leq \frac{n}{2}.$$
\end{lemma} 

We will call the point $(y,s)$ an $\ell_n$-center of $(x,t)$ in Lemma \ref{lcenter}.


\medskip
\subsection{Metric flows and $\bF$-convergence}
Let $(X, d)$ be a complete, separable metric space and denote by $\mathcal{B}(X)$ the Borel $\sigma$-algebra generated by the open subsets of $X$. A probability measure on X is a measure $\mu$ on $\mathcal{B}(X)$ with $\mu(X)=1$. We denote by $\mathcal{P}(X)$ the set of probability measures on $X$. Denote by $\Phi:\mathbb{R}\to (0,1)$ the antiderivative satisfies that $\Phi'(x)=(4\pi)^{-1/2}e^{-x^2/4}$, $\lim_{x\to-\infty}\Phi(x)=0$, $\lim_{x\to\infty}\Phi(x)=1$.

\begin{definition}[Metric Flow Pairs, Definitions 3.2, 5.1 in \cite{Bam20b}]\label{mfpairs}
A metric flow over $I\subseteq \mathbb{R}$ is a tuple
$$(\mathcal{X},\mathfrak{t},(d_t)_{t\in I},(\nu_{x;s})_{x\in \mathcal{X},s \in I\cap (-\infty,\mathfrak{t}(x)]}),$$
where $\mathcal{X}$ is a set, $\mathfrak{t}:\mathcal{X}\to I$ is a function, $d_t$ are metrics on the level sets $\mathcal{X}_t:=\mathfrak{t}^{-1}(t)$, such that $(\mathcal{X}_t, d_t)$ is a complete and separable metric space for all $t$, and $\nu_{x;s}\in \mathcal{P}(\mathcal{X}_s)$, $s\leq \mathfrak{t}(x)$ are such that $\nu_{x;\mathfrak{t}(x)}=\delta_x$ and the following hold:
\begin{enumerate}
    \item (Gradient estimate for heat flows) For $s,t\in I$, $s<t$, $T\geq 0$, if $u_s:\mathcal{X}_s\to [0,1]$ is such that $\Phi^{-1}\circ u_s$ is $T^{-\frac{1}{2}}$-Lipschitz (or just measurable if $T=0$), then either $u_t:\mathcal{X}_t\to [0,1]$, $x\mapsto \int_{\mathcal{X}_s}u_s d\nu_{x;s}$, is constant or $\Phi^{-1}\circ u_t$ is $(T+t-s)^{-\frac{1}{2}}$-Lipschitz, 
    \item (Reproduction formula) For $t_1 \leq t_2 \leq t_3$ in $I$, $\nu_{x;t_1}(E)=\int_{\mathcal{X}_{t_2}}\nu_{y;t_1}(E)d\nu_{x;t_2}(y)$ for $x\in \mathcal{X}_{t_3}$ and all Borel sets $E\subseteq \mathcal{X}_{t_1}$.
\end{enumerate}

A conjugate heat flow on $\mathcal{X}$ is a family $\mu_t \in \mathcal{P}(\mathcal{X}_t)$, $t\in I'$, such that for $s\leq t$ in $I'$, we have $\mu_s(E)=\int_{\mathcal{X}_t} \nu_{x;s}(E) d\mu_t(x)$ for any Borel subset $E\subseteq \mathcal{X}_s$. A metric flow pair $(\mathcal{X},(\mu_t)_{t\in I'})$ consists of a metric flow $\mathcal{X}$, along with a conjugate heat flow $(\mu_t)_{t\in I'}$ such that $\text{supp}(\mu_t)=\mathcal{X}_t$ and $|I\setminus I'|=0$.
\end{definition}
Next, we have the following definitions.
\begin{definition}[Correspondences and $\mathbb{F}$-Distance, Definitions 5.4, 5.6 in \cite{Bam20b}]\label{corrspandFcon}
Given metric flows $(\mathcal{X}^i)_{i\in \mathcal{I}}$ defined over $I'^{,i}$, a correspondence over $I''\subseteq \mathbb{R}$ is a pair
$$
\mathfrak{C}=\left( (Z_t,d_t)_{t\in I''},(\varphi_t^i)_{t\in I''^{,i},i\in \mathcal{I}}\right)
$$
where $(Z_t,d_t^Z)$ are metric spaces, $I''^{,i}\subseteq I'^{,i}\cap I''$, and $\varphi_t^i:(\mathcal{X}_t^i,d_t^i)\to (Z_t,d_t^Z)$ are isometric embeddings.

The $\mathbb{F}$-distance between metric flow pairs $(\mathcal{X}^j,(\mu_t^j)_{t\in I'^{,j}})$, $j=1,2$, within $\mathfrak{C}$ is the infimum of $r>0$ such that there exists a measurable set $E\subseteq I''$ such that $I''\setminus E\subseteq I''^{,1}\cap I''^{,2}$, $|E|\leq r^2$, and there exist couplings $q_t$ of $(\mu_t^1,\mu_t^2)$, $t\in I''\setminus E$, such that for all $s,t\in I''\setminus E$ with $s\leq t$, we have
$$\int_{\mathcal{X}_t^1 \times \mathcal{X}_t^2}d_{W_1}^{Z_s}\left( (\varphi_s^1)_{\ast}\nu_{x^1;s}^1 , (\varphi_s^2)_{\ast}\nu_{x^2;s}^2 \right) dq_t(x^1,x^2) \leq r.$$
The $\mathbb{F}$-distance between metric flow pairs is the infimum of $\mathbb{F}$-distances within a correspondence $\mathfrak{C}$, where $\mathfrak{C}$ is varied among all correspondences.
\end{definition}
For the next definition, we suppose $(\mathcal{X}^i,(\mu_t^i)_{t\in I'^{,i}})$ $\mathbb{F}$-converge to $(\mathcal{X}^{\infty},(\mu_t^{\infty})_{t\in I'^{\infty}})$ within the correspondence $\mathfrak{C}$. 
\begin{definition}[Convergence within a correspondence, Definition 6.18 in \cite{Bam20b}]\label{cwithincor}
Given $\mu^i \in \mathcal{P}(\mathcal{X}_{t_i}^i)$ and $\mu^{\infty}\in \mathcal{P}(\mathcal{X}_{t_{\infty}}^{\infty})$, we write $\mu^i \xrightarrow[i\to \infty]{\mathfrak{C}} \mu^{\infty}$ if $t_i \to t_{\infty}$ and there exist $E_i \subseteq I''$ such that $|I''\setminus E_i|\to 0$, $E_i \subseteq I''$ and 
$$
\lim_{i\to \infty}\sup_{t\in I''\setminus E} d_W^{Z_t}\left( (\varphi_t^i)_{\ast} \mu_t^i , (\varphi_t^{\infty})_{\ast} \mu_t^{\infty} \right)=0,
$$
where $\mu_t^i$ is the conjugate heat flow on $\mathcal{X}^i$ with $\mu_{t_i}^i=\mu^i$, for $i\in \mathbb{N}\cup \{ \infty\}$. We write $x_i \xrightarrow[i \to \infty]{\mathfrak{C}} x_{\infty}$ if $\delta_{x_i} \xrightarrow[i\to \infty]{\mathfrak{C}} \delta_{x_{\infty}}$.
\end{definition}
We will need the following definition of singular spaces introduced by Bamler, see \cite[Definition 2.1]{Bam18}.
\begin{definition}[Singular space]\label{sp}
A tuple $(X,d,\cR,g)$ is called a singular space (of dimension $n$) if it satisfies the following properties:
\begin{enumerate}
\item $(X,d)$ is a locally compact, complete metric length space. 
\item $\cR_{X}\subset X$ is an open and dense subset, which can be equipped with a structure of a smooth Riemannian $n$-manifold $(\cR_{X}, g)$, such that the inclusion map $(\cR_{X}, d_g)\to (X,d)$ is a local isometry,
\item The length metric of $(\cR_{X}, g)$ is equal to the restriction of $d$ to $\cR$. In other words, $(X, d)$ is the completion of the length metric on $(\cR, g)$.
\item For any compact subset $K\subset X$, there are constants $0<\kappa_1(K)<\kappa_2(K)<\infty$ such that for all $x\in K$ and $0<r<1$, we have
$$\kappa_1 r^n\leq |B(x,r)\cap\cR|\leq \kappa_2 r^n.$$
Here $|\cdot|$ denotes the Riemannian volume with respect to the metric $g$ and distance
balls $B(x, r)$ are measured with respect to the metric d.
\end{enumerate}
The subset $\cR$ is called the regular part and its complement $\cS:=X\setminus\cR$ is called the singular part.
\end{definition}
%


\bigskip
\section{Continuity of blow-up limits of the K\"ahler-Ricci flow}\label{ApptotheKRF}
%


%
\medskip
\subsection{Set up and preliminary results}\label{setupandpre}
We consider the normalized version of (\ref{unkrflow1}), that is, we consider the following normalized K\"ahler-Ricci flow
\begin{equation}\label{nkrflow1}
\left\{
\begin{array}{l}
{ \displaystyle \dds{\tilde{\omega}(s)} = -\ric(\tilde{\omega}(s)) + \tilde{\omega}(s),}\\
\\
\tilde{\omega}(0) =\omega_0,
\end{array} \right.
\end{equation}
which has a long-time solution with $s\in [0, \infty)$. The relations between the unnormalized K\"ahler-Ricci flow (\ref{unkrflow1}) and normalized K\"ahler-Ricci flow (\ref{nkrflow1}) are given by
\begin{equation}\label{rb unkrf and nkrf}
s=-\ln (1-t),~~ t=1-e^{-s},~~ \tilde{\omega}(s)=(1-t)^{-1}\omega(t),~~t\in [0,1).
\end{equation}
We can always find a smooth closed $(1,1)$-form $\chi \in -[K_X]$ such that 
$$\omega_Y= \omega_0 - \chi $$
is the restriction of the Fubini-Study metric $\omega_{\fs}$ on $\mathbb{CP}^N$ to $Y$. We can also choose a smooth volume form $\Omega$ such that
$$-\ddbar \log \Omega = \chi$$
since $-\chi \in [K_X]$. 
%
%
%
%
The normalized flow (\ref{nkrflow1}) can be reduced to the complex Monge-Amp\`ere flow as below:
\begin{equation}\label{maflow1}
\left\{
\begin{array}{l}
{ \displaystyle \dds{\varphi} = \log \frac{ (\omega_0 + (e^{s}-1) \omega_Y + \ddbar \varphi )^n }{\Omega} + \varphi, ~~ s\in [0, \infty), }\\
\\
\varphi|_{ s=0} = 0 .
\end{array} \right. 
\end{equation}
We have the following well-known parabolic Schwarz lemma.
\begin{lemma}[\bf{Parabolic Schwarz Lemma}] \label{pschwarz} 
Let $\beta$ be any K\"ahler metric on $\mathbb{CP}^N$. For the solution to the unnormalized flow $\omega(t)$, we have
\begin{equation} \label{pschwarz'1}
\tr_{\omega(t)}\beta \leq C,
\end{equation}
and
\begin{equation} \label{pschwarz'2}
\left( \ddt{} - \Delta_{\omega(t)} \right) \tr_{\omega(t)}\beta \leq -C^{-1}|\nabla \tr_{\omega(t)}\beta|_{\omega(t)}^2+C,
\end{equation}
on $X\times [0, 1)$.
For the solution to the normalized flow $\tilde{\omega}(s)$, we have
\begin{equation} \label{pschwarz'3}
\tr_{\tilde{\omega}(s)}(e^s\beta) \leq C, 
\end{equation}
and
\begin{equation} \label{pschwarz'4}
\left( \dds{} - \Delta_{\tilde{\omega}(s)} \right) \tr_{\tilde{\omega}(s)}(e^s\beta) \leq -C^{-1}|\nabla \tr_{\tilde{\omega}(s)}(e^s\beta)|_{\tilde{\omega}(s)}^2+Ce^{-s},
\end{equation}
on $X\times [0, \infty)$. Here $C<\infty$ is a constant, depends on $n, \omega_0$ and the upper bound for the bisectional curvature of $\beta$.
\end{lemma}
In the normalized flow, we denote the Ricci potential by
\begin{equation}\label{Ricponu0}
u_0 = \dds{\varphi} ,
\end{equation}
then $u_0$ satisfies the following coupled equations
\begin{equation}\label{ceofu0}
\left\{
\begin{array}{l}
\dds {} u_0 = \Delta u_0 + \tr_{\tilde\omega(s)}(e^s\omega_Y) + u_0=n-R_{\tilde g}(s) + u_0 , \\
\\
\ric (\tilde\omega(s)) = \tilde\omega(s) - e^s\omega_Y - \ddbar u_0 .
\end{array} \right. 
\end{equation}
For convenience, we still denote by $u_0(t)$ the function $u_0 (s(t))$ with $s(t)=-\log(1-t)$, which is a function on the unnormalized flow $X\times [0, 1)$.

Now, let $\theta_Y$ be a smooth closed $(1,1)$-form on $Y$ with $\Phi^*\theta_Y\in \vartheta$. Then we have
\begin{equation} \label{defofrho}
\omega_Y-\theta_Y=\ddbar \rho ,
\end{equation}
where $\rho$ is a smooth function on $\mathbb{CP}^N$. We still denote by $\rho$ the pullback function $\pi\circ\rho$. Then in the normalized flow, we define
\begin{equation} \label{defofu'2}
u_1 = u_0 + e^s\rho,
\end{equation}
on $X \times [0, \infty)$. When we are in the unnormalized flow, we still denote by $u_1(t)$ the function $u_1 (s(t))$ with $s(t)=-\log(1-t)$, and we can check that
$$\ric(\omega(t)) - (1-t)^{-1} \omega(t) = - (1-t)^{-1} \theta_Y  - \ddbar  u_1,$$
for all $t\in [0, 1)$. Denote by
\begin{equation} \label{defofalp}
\alpha = \ddbar\rho.
\end{equation}
We can view $\alpha$ as a smooth form on $\mathbb{CP}^N$. Now $u_1$ is a smooth function, satisfying the following coupled equations 
\begin{equation}\label{ceofu}
\left\{
\begin{array}{l}
\dds {} u_1 = \Delta u_1  -  \tr_{\tilde\omega(s)}(e^s(\alpha-\omega_Y)) + u_1 = n- R_{\tilde g}(s) + u_1 , \\
\\
\ric (\tilde\omega(s)) = \tilde\omega(s) + e^s(\alpha-\omega_Y) - \ddbar u_1 ,
\end{array} \right.
\end{equation}
on $X \times [0, \infty)$.

In order to normalize $u_1$, denote by $a(s):=\inf_{X} u_1 ( \cdot , s )$. 
We have the following important estimates of $a(s)$.
\begin{lemma}[Lemma 4.3 in \cite{JST23a}] \label{a'3} 
For any constants $s_0, T\geq 0$, for $s\in [0, s_0+T]$ we have
\begin{equation} \label{a'4}
e^{s-s_0}a(s_0)-B\leq a(s),
\end{equation}
for some constant $B=B(n , \omega_0, \|\rho \|_{C^2(\omega_{\fs})}, T)<\infty$.
\end{lemma}
Given any sequence of times $t_i\nearrow 1$ in the normalized flow, let $s_i=-\ln (1-t_i)\to\infty$ as $i\to\infty$. Let $B_0=B(n , \omega_0, \|\rho \|_{C^2(\omega_{\fs})}, 0)$ be the constant from Lemma \ref{a'3}, then we define
\begin{equation}\label{cordob2}
b_i(s) = e^{s-s_i}a(s_i)-B_0,
\end{equation}
where $s$ is the time parameter in the normalized flow $(X, \tilde g(s))$. By Lemma \ref{a'3}, we have $b_i(s)\leq a(s)$ for all $s\in [0 , s_i]$. We then denote by
\begin{equation}\label{cordovsi}
v_{i}=u_1-b_i(s)+1,
\end{equation}
which is a smooth function on the normalized flow $(X, \tilde g(s))$, $s\in [0, \infty)$. Then we have $v_i\geq 1$ on $X\times [0 , s_i]$.

According to \cite{JST23a}, we have the following gradient and Laplacian estimates
\begin{equation} \label{gleou0}
\frac{ \left| \Delta u_1 \right| }{ u_1 - a + 1 } + \frac{ |\nabla u_1|^2}{ u_1 - a + 1 } \leq C ,
\end{equation}
on the normalized flow $(X, \tilde g(s))$, $s\in [0, \infty)$. Hence for $v_i$, by Lemma \ref{a'3} and the Schwarz lemma, we have
\begin{equation} \label{gleovi}
\frac{ \left| \partial_s v_i \right| }{ v_i } + \frac{ \left| \Delta v_i \right| }{ v_i } + \frac{ |\nabla v_i |^2}{ v_i } \leq C ,
\end{equation}
on $X\times [0 , s_i]$.

Now, recall from the unnormalized flow $(X, g(t))$, $t\in [0, 1)$, we define $M_i=X$ and $g_{i,t}:=(1-t_i)^{-1}g((1-t_i)t+t_i)$, $t\in [-T_i , 0]$ with $T_i=t_i/(1-t_i)$. Hence we can compute that
$$
\omega_{i,t} = (1-t) \tilde \omega(s(t)), ~~ s(t)= -\ln (1-t) -\ln (1-t_i), ~~t\in [-T_i , 0] .
$$
For the convenience of the notations, we still denote by $v_{i}(t)=v_{i}(s(t))$, where $s(t)=-\ln (1-t) -\ln (1-t_i)$, which makes $v_i$ a function on the Ricci flow $(M_i, (g_{i,t})_{t\in [-T_i , 0]})$, hence we have $v_i\geq 1$ on $M_i\times [-T_i, 0]$. From (\ref{ceofu}), $v_i$ satisfy the following coupled equations 
\begin{equation}\label{ceofvi}
\left\{
\begin{array}{l}
\left( \partial_t - \Delta_{\omega_{i,t}} \right) v_i = \frac{ v_i - (B_0+1) }{ 1 - t } - \frac{1}{1-t} \tr_{\omega_{i,t}} ( \frac{\alpha-\omega_Y}{1-t_i} )   , \\
\\
\ric (\omega_{i,t}) = \frac{1}{1-t}( \omega_{i,t} + \frac{\alpha-\omega_Y}{1-t_i} ) - \ddbar v_i ,
\end{array} \right.
\end{equation}
on $M_i\times [-T_i, 0]$. We should remark here that, the factor $\frac{1}{1-t}$ here is not a Type I bound, it's actually a good term on $(M_i, (g_{i,t})_{t\in [-T_i , 0]})$. From the parabolic Schwarz Lemma, say Lemma \ref{pschwarz}, we have for any K\"ahler metric $\beta$ on $\mathbb{CP}^N$, $ \tr_{\omega_{i,t}} \left( \frac{\beta}{1-t_i} \right) \leq C$ and
\begin{equation} \label{pschwarz1}
\left( \partial_t - \Delta_{\omega_{i,t}} \right) \tr_{\omega_{i,t}} \left( \frac{\beta}{1-t_i} \right) \leq -C^{-1}\left|\nabla \tr_{\omega_{i,t}} \left( \frac{\beta}{1-t_i} \right) \right|_{\omega_{i,t}}^2+C(1-t_i),
\end{equation}
on $M_i\times [-T_i, 0]$, for some constant $C$ depending on $\beta$.

From (\ref{gleovi}), on $(M_i, (g_{i,t})_{t\in [-T_i , 0]})$, we have
\begin{equation} \label{gleovi2}
\frac{ \left| \partial_t v_i \right| }{ v_i } + \frac{ \left| \Delta v_i \right| }{ v_i } + \frac{ |\nabla v_i |^2}{ v_i } \leq \frac{C}{1-t} \leq C.
\end{equation}

In conclusion, if we let $p_i$ be the Ricci vertex associated to $\theta_Y$ at $t_i=1-e^{-s_i}$, then we have the following estimates.
\begin{lemma}[] \label{viprop} 
There exists constant $C=C(n , \omega_0, \|\rho \|_{C^4(\omega_{\fs})})<\infty$, such that the following statements hold on the Ricci flow $(M_i, (g_{i,t})_{t\in [-T_i , 0]})$.
\begin{enumerate}
\item $v_i\geq 1$;
\item $\frac{|\partial_t v_i|}{v_i} + \frac{| \Delta v_i |}{v_i} + \frac{|\nabla v_i|^2}{v_i} \leq C$;
\item $R_{g_{i}} \leq Cv_i$;
\item $v_i (p_i, 0)= B_0 + 1$.
\end{enumerate}
Here all the operators are with respect to the metric $g_{i,t}$.
\end{lemma}
Passing to a subsequence, we can obtain $\bF$-convergence on compact time-intervals
\begin{equation}\label{FcorKRF'3}
(M_i, (g_{i,t})_{t\in [-T_i , 0]}, (\nu_{p_i,0; t})_{t\in [-T_i , 0]}) \xrightarrow[i\to\infty]{\bF,\CCC}  (\cX, (\nu_{p_\infty; t})_{t\in (-\infty , 0]}),
\end{equation}
within some correspondence $\CCC$, where $\cX$ is a future continuous and $H_{2n}$-concentrated metric flow of full support over $(-\infty , 0]$. 



%
Throughout this section, unless otherwise stated, all the constants will depend at most on $n , \omega_0, \|\rho \|_{C^4(\omega_{\fs})}$. We will omit this dependence in this section for convenience.


%
\medskip
\subsection{Heat kernel estimate and good distance distortion lower bound}\label{hkandgddl}
In this subsection, we will obtain the good distortion lower bound on $(M_i, (g_{i,t})_{t\in [-T_i , 0]})$. For convenience of notions, we will omit all the subscript $i$ in this subsection.

First, we have the following heat kernel estimate.

\begin{lemma}\label{ghke0}
For $-T\leq s < t \leq 0$ and $x, y\in M$, we have the following estimate for some $C<\infty$:

$(i)$ Any $H_{2n}$-center $(z,s)$ of $(x,t)$ satisfies $d_{s}^{2}(x,z)\leq C\left(1+\int_{s}^{t}R(x,\tau)d\tau\right)(t-s)$;

$(ii)$ We have
$$
K(x,t;y,s) \geq \frac{1}{C(t-s)^{\frac{n}{2}}}\exp\left(-C\int_{s}^{t}R(y,\tau)d\tau-\frac{Cd_{t}^{2}(x,y)}{(t-s)}\right),
$$
$$
K(x,t;y,s)\leq \frac{C}{(t-s)^{\frac{n}{2}}}\exp\left(C\int_{s}^{t}R(x,\tau)d\tau-\frac{d_{s}^{2}(x,y)}{C(t-s)}\right).
$$
\end{lemma}
\begin{proof}
First, we have (note here we are of real dim $2n$)
\begin{equation} \label{ghke1}
\frac{1}{C(t-s)^{n}}\exp\left(-\int_{s}^{t}R(x,\tau)d\tau\right)\leq K(x,t;x,s)\leq\frac{C}{(t-s)^{n}} .
\end{equation}
The upper bound follows from \cite{ZhQ11}, and the lower bound follows by combining $K(x,t;x,s)\geq\frac{1}{C(t-s)^{n}}e^{-\ell_{(x,t)}(x,s)}$ with 
$$
\ell_{(x,t)}(x,s)\leq\frac{1}{2\sqrt{t-s}}\int_{0}^{t-s}\sqrt{\tau}R(x,t-\tau)d\tau\leq\int_{s}^{t}R(x,\tau)d\tau.
$$
$(i)$ Let $(z,s)$ be any $H_{2n}$-center of $(x,t)$. Then from \cite[Theorem 7.2]{Bam20a} and (\ref{ghke1}), we have
$$
\frac{1}{C(t-s)^{n}}\exp\left(-\int_{s}^{t}R(x,\tau)d\tau\right)\leq K(x,t;x,s)\leq\frac{C}{(t-s)^{n}}\exp\left(-\frac{d_{s}^{2}(x,z)}{10(t-s)}\right),
$$
which implies that
$$
d_{s}^{2}(x,z)\leq C\left(1+\int_{s}^{t}R(x,\tau)d\tau\right)(t-s) .
$$

$(ii)$ Qi Zhang's gradient estimate \cite[Theorem 3.2]{ZhQ06} and the upper bound in (\ref{ghke1}) combine to give
$$
\frac{|\nabla_{x}K(x,t;y,s)|}{|K(x,t;y,s)|}\leq\sqrt{\frac{1}{t-s}\log\left(\frac{C(t-s)^{-n}}{K(x,t;y,s)}\right)},
$$
so that
$$
\left|\nabla_{x}\sqrt{\log\left(\frac{C(t-s)^{-n}}{K(x,t;y,s)}\right)}\right|\leq\frac{C}{\sqrt{t-s}},
$$
which we can integrate to get 
$$
\log\left(\frac{C(t-s)^{-n}}{K(x_{2},t;y,s)}\right)\leq2\log\left(\frac{C(t-s)^{-n}}{K(x_{1},t;y,s)}\right)+\frac{Cd_{t}^{2}(x_1, x_2)}{t-s}.
$$
Choosing $x_{1}=y$ and $x_{2}=x$, the lower bound in (\ref{ghke1}) then gives
$$
\log\left(\frac{C(t-s)^{-n}}{K(x,t;y,s)}\right)\leq C\int_{s}^{t}R(y,\tau)d\tau+\frac{Cd_{t}^{2}(x,y)}{t-s}.
$$
Rearranging terms gives the lower bound.

Combining $(i)$ with \cite[Theorem 7.2]{Bam20a} gives
$$
K(x,t;y,s)\leq\frac{C}{(t-s)^{n}}\exp\left(-\frac{d_{s}^{2}(y,z)}{9(t-s)}\right)\leq\frac{C}{(t-s)^{n}}\exp\left(C\int_{s}^{t}R(x,\tau)d\tau-\frac{d_{s}^{2}(x,y)}{10(t-s)}\right),
$$
which proves the upper bound.
\end{proof}
Now we can prove the good distance distortion lower bound.
\begin{proposition}\label{gddlb}
For any $A<T$, $D<\infty$, there exist constants $\theta=\theta(A,D)>0$, $C=C(A,D)<\infty$, such that the following statement holds.

Assume $-A\leq t_{1}< t_{2}\leq0$ satisfies $t_2-t_1\leq \theta$. Assume $x,y\in\bigcup_{ t\in[ t_{1}, t_{2} ] } B( p, t, D)$. Then we have
$$
d_{t_{2}}(x,y)\geq d_{t_{1}}(x,y)-C\sqrt{t_{2}-t_{1}}.
$$
\end{proposition}
\begin{proof}
Throughout the proof, all the constants will depend at most on $A, D$.
Let $(z,t_{1})$ be an $H_{2n}$-center of $(x,t_{2})$, then by Lemma \ref{ghke0}, we have
\begin{equation} \label{ghke2}
d_{t_1}^{2}(x,z)\leq C\left( 1+\int_{t_1}^{t_2}R(x,t)dt \right) (t_2-t_1) .
\end{equation}
Assume $t_x\in [ t_{1}, t_{2} ]$ satisfies $d_{t_x}(x, p)\leq D$, then we have
$$
R(x,t) \leq C v(x,t) \leq C v(x,t_x) \leq C ( 2v(p,t_x) + CD^2 ) \leq C ( Cv(p,0) + CD^2 ) \leq C,
$$
for all $t \in [ t_{1}, t_{2} ]$. Hence by (\ref{ghke2}), we have
\begin{equation} \label{debxz}
d_{t_1}(x,z) \leq C \sqrt{t_2-t_1} .
\end{equation}
Let $u\in C^{\infty}(M \times (t_{1},t_{2}])$ solve the heat equation, with $u(\cdot,t_{1}):=d_{t_{1}}(x,\cdot)$. Then (\ref{debxz}) and \cite[Theorem 3.14]{Bam20a} give
\[
\begin{split}
u(x,t_{2})= & \int_{M}K(x,t_{2};w,t_{1})d_{t_{1}}(w,x)dg_{t_{1}}(w) \\
\leq & C\sqrt{t_{2}-t_{1}}+\int_{M}K(x,t_{2};w,t_{1})d_{t_{1}}(w,z)dg_{t_{1}}(w) \\
\leq & C\sqrt{t_{2}-t_{1}}+\sqrt{t_{2}-t_{1}}\int_{B(z,t_{1},\sqrt{t_{2}-t_{1}})}K(x,t_{2};w,t_{1})dg_{t_{1}}(w) \\
& +\sqrt{t_{2}-t_{1}}\sum_{j=1}^{\infty}2^{j}\int_{B(z,t_{1},2^{j+1}\sqrt{t_{2}-t_{1}})\setminus B(z,t_{1},2^{j}\sqrt{t_{2}-t_{1}})}K(x,t_{2};w,t_{1})dg_{t_{1}}(w) \\
\leq & C\sqrt{t_{2}-t_{1}}\left(1+\sum_{j=1}^{\infty}2^{j}\exp\left(-\frac{(2^{j})^{2}}{10}\right)\right) \\
\leq & C\sqrt{t_{2}-t_{1}}.
\end{split}
\]
If we let $\tilde u\in C^{\infty}(M\times(t_{1},t_{2}])$ solve the heat equation with $\tilde u(\cdot,t_{1})=d_{t_{1}}(y,\cdot)$, the same computation gives $\tilde u(y,t_{2})\leq C\sqrt{t_{2}-t_{1}}$. However, $(u+\tilde u)(\cdot,t_{1})\geq d_{t_1}(x,y)$, so the maximum principle gives $(u+\tilde u)(\cdot,t_{2})\geq d_{t_1}(x,y)$, and in particular $u(y,t_{2})\geq d_{t_1}(x,y) -C\sqrt{t_{2}-t_{1}}$. Because $|\nabla u|\leq1$, we conclude
$$
d_{t_{2}}(x,y)\geq u(y,t_{2})-u(x,t_{2})\geq d_{t_1}(x,y)-C\sqrt{t_{2}-t_{1}},
$$
which completes the proof.
\end{proof}
%


%
\medskip
\subsection{Harnack inequality and global weak distance distortion upper bound}\label{Hiandgwddub}
As in subsection \ref{hkandgddl}, we will omit all the subscript $i$ of $(M_i, (g_{i,t})_{t\in [-T_i , 0]})$ in this subsection.

\begin{theorem}\label{HarnackiofKRF}
For any $t_0,t_1\in [-T, 0]$, $t_0<t_1$, there exists $C=C(t_0)<\infty$, such that the following statement holds.

For any positive solution $u\in C^{\infty}(M\times[ t_0 , t_1 ])$ of the heat equation, we have
$$
-\frac{\Delta u}{u}+\frac{1}{2}\frac{|\nabla u|^{2}}{u^{2}}\leq C\left(\frac{1}{t-t_0}+v\right),
$$
on $M\times (  t_0 , t_1  ]$.
\end{theorem}
\begin{proof}
Recall the evolution equation of the Ricci potential $v$ from (\ref{ceofvi}):
$$
\left( \partial_t - \Delta \right) v = \frac{ v - (B_0+1) }{ 1 - t } - \frac{1}{1-t} \tr_{\omega_t} \beta  ,
$$
where $\beta=\frac{\alpha-\omega_Y}{1-t_i}$. Then we can compute
\begin{equation} \label{evoeofgradv}
(\partial_{t}-\Delta)|\nabla v|^{2}=-|\nabla\nabla v|^{2}-|\nabla\overline{\nabla}v|^{2}+\frac{2}{1-t}|\nabla v|^{2}+\frac{2}{1-t}\text{Re}\left\langle \nabla\text{tr}_{\omega_{t}}\beta,\overline{\nabla}v\right\rangle .
\end{equation}
Denote by $\beta_0=\frac{1}{1-t_i}\omega_{\fs}$. Then choose a constant $A<\infty$ large enough such that $\widetilde{\omega}:=A\beta_0+\beta$ is a K\"ahler metric on $\mathbb{C}P^{N}$. By the parabolic Schwarz lemma (\ref{pschwarz1}), we have
\begin{equation} \label{pschwarz2}
(\partial_{t}-\Delta)\text{tr}_{\omega_{t}}\widetilde{\omega}\leq-C^{-1}|\nabla\text{tr}_{\omega_{t}}\widetilde{\omega}|^{2}+C ,
\end{equation}
\begin{equation} \label{pschwarz3}
\left( \partial_t - \Delta \right) \tr_{\omega_{t}} \beta_0 \leq -C^{-1}\left|\nabla \tr_{\omega_{t}} \beta_0 \right|^2+C  ,
\end{equation}
and $\text{tr}_{\omega_{t}}\widetilde{\omega}\leq C$, $\tr_{\omega_{t}} \beta_0 \leq C$. Hence by (\ref{ceofvi}), we have 
\begin{equation} \label{ricnormboundbyv}
|\nabla\overline{\nabla}v|^{2}\geq|Ric|^{2} - C  .
\end{equation}
We will estimate the following Li-Yau type Harnack quantity:
$$
F:=-\frac{\Delta u}{u}+\delta\frac{|\nabla u|^{2}}{u^{2}}+\alpha|\nabla v|^{2}+\text{tr}_{\omega_{t}}\widetilde{\omega}+\text{tr}_{\omega_{t}}\beta_0 -\gamma v ,
$$
where $\delta\in[\frac{1}{2},1)$, $\alpha,\gamma\in[1,\infty)$ are
to be determined. We compute
$$
(\partial_{t}-\Delta)\frac{\Delta u}{u}=\frac{\langle Ric,\nabla\overline{\nabla}u\rangle}{u}+2\frac{\text{Re}\left\langle \nabla\Delta u,\overline{\nabla}u\right\rangle }{u^{2}}-2\frac{|\nabla u|^{2}\Delta u}{u^{3}},
$$
\[
\begin{split}
(\partial_{t}-\Delta)\frac{|\nabla u|^{2}}{u^{2}} = & -\frac{|\nabla\nabla u|^{2}+|\nabla\overline{\nabla}u|^{2}}{u^{2}} \\
& +4\frac{\text{Re}\left(\langle\nabla\nabla u,\overline{\nabla}u\otimes\overline{\nabla}u\rangle+\langle\nabla\overline{\nabla}u,\nabla u\otimes\overline{\nabla}u\rangle\right)}{u^{3}}-6\frac{|\nabla u|^{4}}{u^{4}} .
\end{split}
\]
From (\ref{ceofvi}), we have 
$$
\frac{\langle Ric,\overline{\nabla}\nabla u\rangle}{u}=\frac{1}{1-t}\frac{\Delta u}{u} - \frac{\langle\nabla\overline{\nabla}u,\overline{\nabla}\nabla v\rangle}{u}-\frac{1}{1-t}\langle \beta,\overline{\nabla}\nabla u\rangle .
$$
Combining expressions we can compute
\begin{align*}
& (\partial_{t}-\Delta)F \\
& \leq  -\frac{1}{1-t}\frac{\Delta u}{u} + \frac{\langle\nabla\overline{\nabla}u,\overline{\nabla}\nabla v\rangle}{u}+\frac{1}{1-t}\langle \beta , \overline{\nabla}\nabla u\rangle-2\frac{\text{Re}\langle\nabla\Delta u,\nabla u\rangle}{u^{2}}+2\frac{|\nabla u|^{2}\Delta u}{u^{3}}\\
& -\delta\frac{|\nabla\nabla u|^{2}+|\nabla\overline{\nabla}u|^{2}}{u^{2}}+4\delta\frac{\text{Re}\left(\langle\nabla\nabla u,\overline{\nabla}u\otimes\overline{\nabla}u\rangle+\langle\nabla\overline{\nabla}u,\overline{\nabla}u\otimes\nabla u\rangle\right)}{u^{3}}-6\delta\frac{|\nabla u|^{4}}{u^{4}}\\
& -\alpha|\nabla\nabla v|^{2}-\alpha|\nabla\overline{\nabla}v|^{2}+\frac{2\alpha}{1-t}|\nabla v|^{2}+\frac{2\alpha}{1-t}\text{Re}\left\langle \nabla\text{tr}_{\omega_t}\beta,\overline{\nabla}v\right\rangle -C^{-1}|\nabla\text{tr}_{\omega_{t}}\beta_0|^{2} \\
& -C^{-1}|\nabla\text{tr}_{\omega_t}\widetilde{\omega}|^{2}-\gamma\frac{1}{1-t}v-\gamma\frac{1}{1-t}\text{tr}_{\omega_t}\beta+C .
\end{align*}
Combining this with 
\begin{align*}
& \text{Re}\left\langle \nabla F,\frac{\overline{\nabla}u}{u}\right\rangle \\
= & -\frac{\text{Re}\langle\nabla\Delta u,\overline{\nabla}u\rangle}{u^{2}}+\frac{|\nabla u|^{2}\Delta u}{u^{3}}+\delta\frac{\text{Re}\left(\langle\nabla\nabla u,\overline{\nabla}u\otimes\overline{\nabla}u\rangle+\langle\nabla\overline{\nabla}u,\overline{\nabla}u\otimes\nabla u\rangle\right)}{u^{3}} \\
& - 2\delta\frac{|\nabla u|^{4}}{u^{4}} +2\alpha\frac{\text{Re}\left(\langle\nabla\nabla v,\overline{\nabla}u\otimes\overline{\nabla}v\rangle+\langle\nabla\overline{\nabla}v,\overline{\nabla}u\otimes\nabla v\rangle\right)}{u} \\
& + \frac{\text{Re}\langle\nabla\text{tr}_{\omega_t}\beta_0,\overline{\nabla}u\rangle}{u} +\frac{\text{Re}\langle\nabla\text{tr}_{\omega_t}\widetilde{\omega},\overline{\nabla}u\rangle}{u}-\gamma\frac{\text{Re}\langle\nabla v,\overline{\nabla}u\rangle}{u},
\end{align*}
we obtain the following:
\begin{align*}
&(\partial_{t}-\Delta)F-2\text{Re}\left\langle \nabla F,\frac{\overline{\nabla} u}{u}\right\rangle \\
\leq & -\frac{1}{1-t}\frac{\Delta u}{u}+\frac{\langle\nabla\overline{\nabla}u,\overline{\nabla}\nabla v\rangle}{u}+\frac{1}{1-t}\langle\beta,\overline{\nabla}\nabla u\rangle-\delta\frac{|\nabla\nabla u|^{2}+|\nabla\overline{\nabla}u|^{2}}{u^{2}}-2\delta\frac{|\nabla u|^{4}}{u^{4}}\\
& -\alpha|\nabla\nabla v|^{2}-\alpha|\nabla\overline{\nabla}v|^{2}+\frac{2\alpha}{1-t}|\nabla v|^{2}+\frac{2\alpha}{1-t}\text{Re}\left\langle \nabla\text{tr}_{\omega_t} \beta,\overline{\nabla}v\right\rangle - C^{-1}|\nabla\text{tr}_{\omega_{t}}\beta_0|^{2} \\
& -C^{-1}|\nabla\text{tr}_{\omega_t}\widetilde{\omega}|^{2}-\frac{\gamma}{1-t}v-\frac{\gamma}{1-t}\text{tr}_{\omega_t}\beta-2\frac{\text{Re}\langle\nabla\text{tr}_{\omega_t}\widetilde{\omega},\overline{\nabla}u\rangle}{u}+2\gamma\frac{\text{Re}\langle\nabla v,\overline{\nabla}u\rangle}{u}\\
& -2\frac{\text{Re}\langle\nabla\text{tr}_{\omega_t}\beta_0,\overline{\nabla}u\rangle}{u}+2\delta\frac{\text{Re}\left(\langle\nabla\nabla u,\overline{\nabla}u\otimes\overline{\nabla}u\rangle+\langle\nabla\overline{\nabla}u,\overline{\nabla}u\otimes\nabla u\right)}{u^{3}}\\
& -2\alpha\frac{\text{Re}\left(\langle\nabla\nabla v,\overline{\nabla}u\otimes\overline{\nabla}v\rangle+\langle\nabla\overline{\nabla}v,\overline{\nabla}u\otimes\nabla v\rangle\right)}{u}+C.\\
& 
\end{align*}
Using Cauchy's inequality, we can also estimate
\begin{align*}
\frac{2\alpha}{1-t}\text{Re}\left\langle \nabla\text{tr}_{\omega_t}\beta,\overline{\nabla}v\right\rangle \leq & \frac{1}{20CA}|\nabla\text{tr}_{\omega_t}\beta|^{2}+\frac{20CA\alpha^{2}}{(1-t)^{2}}|\nabla v|^{2}\\
\leq & \frac{1}{10CA}\left(|\nabla\text{tr}_{\omega_t}\widetilde{\omega}|^{2}+A|\nabla\text{tr}_{\omega_t}\beta_0|^{2}\right) + \frac{20CA\alpha^{2}}{(1-t)^{2}}|\nabla v|^{2},
\end{align*}
$$
-2\frac{\text{Re}\langle\nabla\text{tr}_{\omega_t}\beta_0,\overline{\nabla}u\rangle}{u}\leq\frac{1}{10C}|\nabla\text{tr}_{\omega_t}\beta_0|^{2}+10C\frac{|\nabla u|^{2}}{u^{2}},
$$
$$
-2\frac{\text{Re}\langle\nabla\text{tr}_{\omega_t}\widetilde{\omega},\overline{\nabla}u\rangle}{u}\leq\frac{1}{10C}|\nabla\text{tr}_{\omega_t}\widetilde{\omega}|^{2}+10C\frac{|\nabla u|^{2}}{u^{2}},
$$
$$
\frac{\langle\nabla\overline{\nabla}u,\overline{\nabla}\nabla v\rangle}{u}\leq\frac{\alpha}{2}|\nabla\overline{\nabla}v|^{2}+\frac{1}{2\alpha}\frac{|\nabla\overline{\nabla}u|^{2}}{u^{2}},
$$
$$
-2\alpha\frac{\text{Re}\left(\langle\nabla\nabla v,\overline{\nabla}u\otimes\overline{\nabla}v\rangle+\langle\nabla\overline{\nabla}v,\overline{\nabla}u\otimes\nabla v\rangle\right)}{u}\leq\frac{\alpha}{2}(|\nabla\nabla v|^{2}+|\nabla\overline{\nabla}v|^{2})+4\alpha\frac{|\nabla u|^{2}}{u^{2}}|\nabla v|^{2},
$$
$$
\frac{1}{1-t}\langle\beta,\overline{\nabla}\nabla u\rangle\leq  \frac{CA\alpha}{(1-t)^{2}}(\text{tr}_{\omega_t}\widetilde{\omega}+\text{tr}_{\omega_t}\beta_0)^{2}+\frac{1}{2\alpha}\frac{|\nabla\overline{\nabla}u|^{2}}{u^{2}}\leq \frac{1}{2\alpha}\frac{|\nabla\overline{\nabla}u|^{2}}{u^{2}}+\frac{CA\alpha}{(1-t)^{2}},
$$
$$
-\frac{1}{1-t}\frac{\Delta u}{u}\leq\frac{n|\nabla\overline{\nabla}u|}{(1-t)u}\leq\frac{1}{\alpha}\frac{|\nabla\overline{\nabla}u|^{2}}{u^{2}}+\frac{n\alpha^{2}}{4(1-t)^2}.
$$
Combining expressions gives
\begin{align*}
&(\partial_{t}-\Delta)F-2\text{Re}\left\langle \nabla F,\frac{\overline{\nabla} u}{u}\right\rangle \\
\leq & \frac{2}{\alpha}\frac{|\nabla\overline{\nabla}u|^{2}}{u^{2}}-\delta\frac{|\nabla\nabla u|^{2}+|\nabla\overline{\nabla}u|^{2}}{u^{2}}-2\delta\frac{|\nabla u|^{4}}{u^{4}}+\frac{2\alpha}{1-t}|\nabla v|^{2}+4\frac{|\nabla u|^{2}}{u^{2}}|\nabla v|^{2} \\
&+C\frac{|\nabla u|^{2}}{u^{2}}+\frac{20CA\alpha^{2}}{(1-t)^{2}}|\nabla v|^{2} -\gamma\frac{1}{1-t}v-\gamma\frac{1}{1-t}\text{tr}_{\omega_t}\beta+\frac{CA\alpha^2}{(1-t)^{2}}\\
& + 2\gamma\frac{\text{Re}\langle\nabla v,\overline{\nabla}u\rangle}{u} +2\delta\frac{\text{Re}\left(\langle\nabla\nabla u,\overline{\nabla}u\otimes\overline{\nabla}u\rangle+\langle\nabla\overline{\nabla}u,\overline{\nabla}u\otimes\nabla u\rangle\right)}{u^{3}}.
\end{align*}
Next, we complete the square to obtain
$$
-\delta\frac{|\nabla\nabla u|^{2}}{u^{2}}+2\delta\frac{\text{Re}\langle\nabla\nabla u,\overline{\nabla}u\otimes\overline{\nabla}u\rangle}{u^{3}}-\delta\frac{|\nabla u|^{4}}{u^{4}}=-\delta\left|\frac{\nabla\nabla u}{u}-\frac{\nabla u\otimes\nabla u}{u^{2}}\right|^{2},
$$
\begin{align*}
-\left(\delta-\frac{2}{\alpha}\right)&\frac{|\nabla\overline{\nabla}u|^{2}}{u^{2}}+ 2\delta\frac{\text{Re}\langle\nabla\overline{\nabla}u,\overline{\nabla}u\otimes\nabla u\rangle}{u^{3}}-\delta\frac{|\nabla u|^{4}}{u^{4}} \\
= & -\left(\delta-\frac{2}{\alpha}\right)\left|\frac{\nabla\overline{\nabla}u}{u}-\frac{\delta}{\delta-\frac{2}{\alpha}}\frac{\nabla u\otimes\overline{\nabla}u}{u^{2}}\right|^{2}+\delta\left(\frac{\delta}{\delta-\frac{2}{\alpha}}-1\right)\frac{|\nabla u|^{4}}{u^{4}}.
\end{align*}
Again combining expressions, and using $|\beta|\leq C$ and $|\nabla v|^{2}\leq Cv$, we obtain 
\begin{align*}
&(\partial_{t}-\Delta)F-2\text{Re}\left\langle \nabla F,\frac{\overline{\nabla} u}{u}\right\rangle \\
\leq & -\left(\delta-\frac{2}{\alpha}\right)\left|-\frac{\nabla\overline{\nabla}u}{u}+\frac{\delta}{\delta-\frac{2}{\alpha}}\frac{\nabla u\otimes\overline{\nabla}u}{u^{2}}\right|^{2}+\delta\left(\frac{\delta}{\delta-\frac{2}{\alpha}}-1\right)\frac{|\nabla u|^{4}}{u^{4}}+C\alpha\frac{|\nabla u|^{2}}{u^{2}}v \\
& +C\frac{|\nabla u|^{2}}{u^{2}}-\left(\frac{\gamma}{1-t}-\frac{C\alpha}{1-t}-\frac{CA\alpha^{2}}{(1-t)^{2}}\right)v +2\gamma\frac{\text{Re}\langle\nabla v,\overline{\nabla}u\rangle}{u} +\frac{C\gamma}{1-t}+\frac{CA\alpha^2}{(1-t)^{2}}.
\end{align*}
Next, we observe that 
$$
-\frac{\Delta u}{u}+\frac{\delta}{\delta-\frac{2}{\alpha}}\frac{|\nabla u|^{2}}{u^{2}}=F+\left(\frac{\delta}{\delta-\frac{2}{\alpha}}-\delta\right)\frac{|\nabla u|^{2}}{u^{2}}-\alpha|\nabla v|^{2}-\text{tr}_{\omega_t}\widetilde{\omega}-\text{tr}_{\omega}\beta_0+\gamma v.
$$
We assume that $\gamma\geq\underline{\gamma}(\alpha)$ has been chosen so that
$$
\gamma v\geq2\alpha|\nabla v|^{2}+2\text{tr}_{\omega_t}\widetilde{\omega}+2\text{tr}_{\omega_t}\beta_0.
$$
At any point $(x^*,t^*)\in M\times(t_0, t_1]$ where $Q:=(t-t_0)F$ achieves a positive maximum, we then have
\begin{align*}
& 0\leq (t-t_0)(\partial_{t}-\Delta)Q=(t-t_0)F+(t-t_0)^{2}(\partial_{t}-\Delta)F\\
\leq & (t-t_0)F-\frac{(t-t_0)^{2}}{n}\left(\delta-\frac{2}{\alpha}\right)\left(F+\left(\frac{\delta}{\delta-\frac{2}{\alpha}}-\delta\right)\frac{|\nabla u|^{2}}{u^{2}}+\frac{\gamma}{2}v\right)^{2}\\
& +C(t-t_0)^{2}\alpha\frac{|\nabla u|^{2}}{u^{2}}v+C(t-t_0)^{2}\frac{|\nabla u|^{2}}{u^{2}}-\left(\frac{\gamma}{1-t}-\frac{C\alpha}{1-t}-\frac{CA\alpha^{2}}{(1-t)^{2}}\right)(t-t_0)^{2}v\\
& +2\gamma (t-t_0)^{2}\frac{\text{Re}\langle\nabla v,\overline{\nabla}u\rangle}{u}+\delta\left(\frac{\delta}{\delta-\frac{2}{\alpha}}-1\right)(t-t_0)^{2}\frac{|\nabla u|^{4}}{u^{4}}+C(\gamma + \alpha) (t-t_0)^{2}\\
\leq & -\frac{\delta-\frac{2}{\alpha}}{n}Q^{2}+Q-\left(\frac{1}{n}\left(\delta-\frac{2}{\alpha}\right)\left(\frac{\delta}{\delta-\frac{2}{\alpha}}-\delta\right)^{2}-\delta\left(\frac{\delta}{\delta-\frac{2}{\alpha}}-1\right)\right)(t-t_0)^{2}\frac{|\nabla u|^{4}}{u^{4}}\\
& -\frac{(t-t_0)^{2}}{4n}\left(\delta-\frac{2}{\alpha}\right)\gamma^{2}v^{2} -\left(\frac{1}{n}\left(\delta-\frac{2}{\alpha}\right)\left(\frac{\delta}{\delta-\frac{2}{\alpha}}-\delta\right)\gamma-C\alpha\right)(t-t_0)^{2}\frac{|\nabla u|^{2}}{u^{2}}v\\
& +(t-t_0)^{2}\left(C\frac{|\nabla u|^{2}}{u^{2}}-\left(\gamma-\frac{C\alpha^{2}}{1-t}\right)\frac{v}{1-t}+2\gamma \frac{|\nabla u|}{u}|\nabla v|+C(\gamma + \alpha) \right).
\end{align*}
Next, we use Cauchy's inequality to estimate
$$
C(t-t_0)^{2}\frac{|\nabla u|^{2}}{u^{2}}\leq\gamma^{-1}(t-t_0)^{2}\frac{|\nabla u|^{4}}{u^{4}}+C\gamma (t-t_0)^{2},
$$
\begin{align*}
2\gamma (t-t_0)^{2}\frac{|\nabla u|}{u}|\nabla v|\leq\frac{1}{2n} & \left(\delta-\frac{2}{\alpha}\right)\left(\frac{\delta}{\delta-\frac{2}{\alpha}}-\delta\right)\gamma (t-t_0)^{2}\frac{|\nabla u|^{2}}{u^{2}}v\\
& +C\left(\delta-\frac{2}{\alpha}\right)^{-1}\left(\frac{\delta}{\delta-\frac{2}{\alpha}}-\delta\right)^{-1}\gamma (t-t_0)^{2} ,\\
\end{align*}
so combining expressions gives the following at $(x^*,t^*)$:
\begin{align*}
0\leq & -\frac{\delta-\frac{2}{\alpha}}{n}Q^{2}+Q-\left(\frac{1}{2n}\left(\delta-\frac{2}{\alpha}\right)\left(\frac{\delta}{\delta-\frac{2}{\alpha}}-\delta\right)\gamma-C\alpha\right)(t-t_0)^{2}\frac{|\nabla u|^{2}}{u^{2}}v\\
& -\left(\frac{1}{n}\left(\delta-\frac{2}{\alpha}\right)\left(\frac{\delta}{\delta-\frac{2}{\alpha}}-\delta\right)^{2}-\delta\left(\frac{\delta}{\delta-\frac{2}{\alpha}}-1\right)-\gamma^{-1}\right)(t-t_0)^{2}\frac{|\nabla u|^{4}}{u^{4}}\\
& -\frac{(t-t_0)^{2}}{4n}\left(\delta-\frac{2}{\alpha}\right)\gamma^{2}v^{2} -(t-t_0)^{2}\left(\gamma-\frac{C\alpha^{2}}{1-t}\right)\frac{v}{1-t}+C(\alpha, \gamma, \delta, t_0).
\end{align*}
We now choose $\delta=\frac{1}{2}$, $\alpha=20n$, so that $\delta>\frac{10}{\alpha},$ then we have
$$
\frac{1}{n}\left(\delta-\frac{2}{\alpha}\right)\left(\frac{\delta}{\delta-\frac{2}{\alpha}}-\delta\right)^{2}-\delta\left(\frac{\delta}{\delta-\frac{2}{\alpha}}-1\right)=\frac{\left(\left(\frac{1}{2}+\frac{1}{10n}\right)^2-\frac{1}{5}\right)}{2n-\frac{2}{5}}=c_0(n)>0.
$$
We next choose $\gamma\geq\underline{\gamma}(\alpha)$ large so that
$$
\gamma>\frac{C\alpha^{2}}{1-t} + c_0(n)^{-1} , ~~~~  \frac{1}{2n}\left(\delta-\frac{2}{\alpha}\right)\left(\frac{\delta}{\delta-\frac{2}{\alpha}}-\delta\right)\gamma>C\alpha ,
$$
then at at $(x^*,t^*)$, we have
$$
0\leq -\frac{1}{4n}Q^{2}+Q+C(t_0)\leq-\frac{1}{8n}Q^{2}+C(t_0),
$$
so that $Q\leq C(t_{0})$ on $M\times(t_0, t_1]$. Hence we have
$$
-\frac{\Delta u}{u}+\frac{1}{2}\frac{|\nabla u|^{2}}{u^{2}}+20n|\nabla v|^{2}+\text{tr}_{\omega_{t}}\widetilde{\omega}+\text{tr}_{\omega_{t}}\beta_0-\gamma v\leq\frac{C(t_{0})}{t-t_0}, 
$$
which completes the proof.
\end{proof}
As an application of Theorem \ref{HarnackiofKRF}, we have the following global weak distortion estimate.
\begin{proposition}\label{gwdde}
Let $t_1,t_2\in [-T, 0]$, $0<t_2-t_1<1$, there exists $C=C(t_1)<\infty$, such that the following statement holds.

For any $x_{1},x_{2}\in M$, if $d_{t_{1}}(x_{1},x_{2})\geq \sqrt{t_2-t_1}$, then we have
$$
\frac{d_{t_{2}}(x_{1},x_{2})}{d_{t_{1}}(x_{1},x_{2})}\leq C\exp\left((t_{2}-t_{1})\left(d_{t_{1}}^{2}(x_{1},p)+d_{t_{1}}^{2}(x_{2},p)\right)+d_{t_{1}}(x_{1},p)+d_{t_{1}}(x_{2},p)\right);
$$
if $d_{t_{1}}(x_{1},x_{2})\leq \sqrt{t_2-t_1}$, then we have
$$
\frac{d_{t_{2}}(x_{1},x_{2})}{\sqrt{t_2-t_1}}\leq C\exp\left((t_{2}-t_{1})\left(d_{t_{1}}^{2}(x_{1},p)+d_{t_{1}}^{2}(x_{2},p)\right)+d_{t_{1}}(x_{1},p)+d_{t_{1}}(x_{2},p)\right).
$$
\end{proposition}
\begin{proof}
Denote by $r_0=\sqrt{t_2-t_1}$, then consider the rescaling $\hat g_{t}=r_0^{-2}g_{r_0^2t}$. Denote by $\hat{t}_1=r_0^{-2}t_1$, $\hat{t}_2=r_0^{-2}t_2$. Let $\gamma:[0,d_{\hat g_{\hat t_{1}}}(x_{1},x_{2})]\to M$ be a unit-speed $\hat g_{\hat t_{1}}$-minimizing geodesic. Define $N:=\lfloor d_{\hat g_{\hat t_{1}}}(x_{1},x_{2})\rfloor$, and consider the restrictions $\gamma_{j}:=\gamma|_{[j,\min\{j+1,d_{\hat g_{\hat t_{1}}}(x_{1},x_{2})\}]}$ for $j=0,...,N$. Define $s_{j}:=j$ for $j=0,...,N$, and $s_{N+1}:=d_{\hat g_{\hat t_{1}}}(x_{1},x_{2})$; also set $z_{j}:=\gamma(s_{j})$. For each $j$, let $(z_{j}',\hat t_{1}-1)\in M$ be a $\ell_{2n}$-center of $(z_{j}, \hat t_{1})$, so that $K(z_{j},\hat t_{1};z_{j}', \hat t_{1}-1)\geq c(n)$. By Qi Zhang's gradient estimate \cite[Theorem 3.2]{ZhQ06}, we then have
\[
\inf_{s\in[s_{j},s_{j+1}]}K(\gamma_{j}(s),\hat t_{1};z_{j}',\hat t_{1}-1)\geq c>0.
\]

Next, we apply the above Harnack inequality, say Theorem \ref{HarnackiofKRF}, to $u:=K(\cdot,\cdot;z_{j}', \hat t_{1}-1)$ to obtain (by the triangle inequality)
\begin{align*}
-\partial_{t}\log u(\gamma_{j}(s),t)\leq & C(t_1)(t_{2}-t_{1})\left(1+d_{g_{t_{1}}}^{2}(\gamma_{j}(s),p))\right) \\
\leq & C(t_1)(t_{2}-t_{1})\left(1+d_{g_{t_{1}}}^{2}(x_{1},p)+d_{g_{t_{1}}}^{2}(x_{2},p)\right),
\end{align*}
for all $t\in [\hat t_{1} , \hat t_{2}]$. Integrating in time gives
\[
\log\left(\frac{u(\gamma_{j}(s),\hat t_{1})}{u(\gamma_{j}(s),\hat t_{2})}\right)\leq C(t_1)(t_{2}-t_{1})\left(1+d_{g_{t_{1}}}^{2}(x_{1},p)+d_{g_{t_{1}}}^{2}(x_{2},p)\right).
\]
That is, 
\[
u(\gamma_{j}(s),\hat t_{2})\geq c\exp\left(-C(t_1)(t_{2}-t_{1})\left(d_{g_{t_{1}}}^{2}(x_{1},p)+d_{g_{t_{1}}}^{2}(x_{2},p)\right)\right).
\]
Again by \cite[Theorem 3.2]{ZhQ06}, we have
\[
\left|\nabla\sqrt{\log\left(\frac{C}{u(\cdot,\hat t_{2})}\right)}\right|\leq C,
\]
so that 
\[
\log\left(\frac{C}{u(\gamma_{j}(s), \hat t_{2})}\right)\leq Cd_{\hat g_{\hat t_{2}}}^2(y,\gamma_{j}(s))+2\log\left(\frac{C}{u(y,\hat t_{2})}\right),
\]
hence for all $s\in[s_{j},s_{j+1}]$ and $y\in B(\gamma_{j}(s),\hat t_{2},1)$, we have
\[
K\left(y,\hat t_{2};z_{j}',\hat t_{1}-1\right)\geq c\exp\left(-C(t_1)(t_{2}-t_{1})\left(d_{g_{t_{1}}}^{2}(x_{1},p)+d_{g_{t_{1}}}^{2}(x_{2},p)\right)\right).
\]
Fix $j\in\mathbb{N}$, and let $N_{j}\in\mathbb{N}$ be maximal such that there exists $s_{j}\leq s_{j,1}\leq\cdots\leq s_{j,N_{j}}\leq s_{j+1}$ such that $B(\gamma_{j}(s_{j,i}),t_{2},1)$ are pairwise disjoint. By arguing as in \cite{BZ17}, we have $d_{\hat g_{\hat t_{2}}}(z_{j},z_{j+1})\leq2N_{j}$. On the other hand, we have 
\[
|B(\gamma_{j}(s_{j,i}),\hat t_{2},1)|_{\hat g_{\hat t_{2}}}\geq\frac{1}{C(t_1)(1+d_{g_{t_{1}}}^{2}(x_{1},p)+d_{g_{t_{1}}}^{2}(x_{2},p))^{2n}} ,
\]
for $i=1,...,N_{j}$, hence
\begin{align*}
e^{n}\geq & \int_{M}K\left(y,\hat t_{2}; z_{j}',\hat t_{1}-1\right)d\hat g_{\hat t_{2}}(y)\geq\sum_{i=1}^{N_{j}}\int_{B(\gamma_{j}(s),\hat t_{2},1)}K\left(y,\hat t_{2};z_{j}', \hat t_{1}-1\right)d\hat g_{\hat t_{2}}(y)  \\
\geq & \frac{c}{(1+d_{g_{t_{1}}}^{2}(x_{1},p)+d_{g_{t_{1}}}^{2}(x_{2},p))^{2n}}\exp\left(-C(t_{2}-t_{1})\left(d_{g_{t_{1}}}^{2}(x_{1},p)+d_{g_{t_{1}}}^{2}(x_{2},p)\right)\right)N_{j}.
\end{align*}
Combining estimates gives
\begin{align*}
d_{\hat g_{\hat t_{2}}}(z_{j},z_{j+1})\leq C(1+d_{g_{t_{1}}}^{2}(x_{1},p)+d_{g_{t_{1}}}^{2}(x_{2},p))^{2n}\exp\left(C(t_{2}-t_{1})\left(d_{g_{t_{1}}}^{2}(x_{1},p)+d_{g_{t_{1}}}^{2}(x_{2},p)\right)\right) ,
\end{align*}
this proves the case $N=0$. When $N>0$, summing in $j$ gives 
\[
\frac{d_{\hat g_{\hat t_{2}}}(x_{1},x_{2})}{d_{\hat g_{\hat t_{1}}}(x_{1},x_{2})}\leq C(1+d_{g_{t_{1}}}^{2}(x_{1},p)+d_{g_{t_{1}}}^{2}(x_{2},p))^{2n}\exp\left(C(t_{2}-t_{1})\left(d_{g_{t_{1}}}^{2}(x_{1},p)+d_{g_{t_{1}}}^{2}(x_{2},p)\right)\right) .
\]
This completes the proof.
\end{proof}
%


%
\medskip
\subsection{Good distance distortion upper bound}\label{gddub}

As in subsection \ref{hkandgddl}, we will omit all the subscript $i$ of $(M_i, (g_{i,t})_{t\in [-T_i , 0]})$ in this subsection.

\begin{proposition}\label{roughvb}
For all $t\in[-T,0]$, $r>0$, and $x \in M$, we have
\[
|B(x,t,r)|_{t}\leq C(t)r^{2n}e^{C(t)r \log( 10 + d_t(x,p))}.
\]
\end{proposition}
\begin{proof}
Throughout the proof, all the constants depend at most on $t$. By Perelman's $\kappa$-noncollapsing estimate and Bamler's noninflating estimate in \cite{Bam20a}, we
have 
\[
\frac{cr^{2n}}{(1+d_{t}(x,p))^{4n}}\leq|B(x,t,r)|_{t}\leq Cr^{2n},
\]
for all $r\in(0,1]$ and $x\in M$. We will argue in a single time slice, so fix $t\in[-T,0]$, and write $B(x,r)=B(x,t,r)$, etc. 

We will show by induction on $k\in\mathbb{N}$ that 
$$|B(x,\frac{1}{4}k)|\leq C^k (1+d_{t}(x, p))^{4nk}$$
for all $x\in M$. The cases $k\leq4$ follow from Bamler's noninflating estimate. Suppose the claim holds for some $k\geq4$. Given $x\in M$, let $\{x_{1},...,x_{N}\}$ be a maximal subset of $B(x,1)$ such that $\{B(x_{j},\frac{1}{8})\}_{j=1}^{N}$ is a pairwise disjoint collection. Then 
\[
C\geq|B(x,1)|\geq\sum_{j=1}^{N}|B(x_{j},\frac{1}{8})|\geq N\frac{c}{(1+d_{t}(x, p))^{4n}},
\]
so that 
$$
N\leq C_1(1+d_{t}(x, p))^{4n}.
$$ 
Moreover, we know $B(x,1)\subseteq\bigcup_{j=1}^{N}B(x_{j},\frac{1}{4})$, so that $\bigcup_{j=1}^{N}B(x_{j},\frac{k}{4})$ cover $B(x,\frac{1}{4}(k+1))$. Then the induction hypothesis gives
\begin{align*}
|B(x,\frac{1}{4}(k+1))|\leq \sum_{j=1}^{N}|B(x_{j},\frac{k}{4})| \leq C_1C^k(1+d_{t}(x, p))^{4n(k+1)}  ,
\end{align*}
so if we choose $C\geq C_{1}$, then the claim follows. Hence for all $k\in\mathbb{N}$, we have
\begin{align*}
|B(x,k/4)| \leq C^k \left(1+d_t(x,p)\right)^{4nk} = e^{Ck+4nk\log(1+d_t(x,p)) } ,
\end{align*}
for all $k\in \mathbb{N}$.

Now fix $r>0$. If $r\in (0,1]$, then Bamler's volume non-inflating gives 
$$
|B(x,t,r)|_t \leq Cr^{2n} \leq Cr^{2n} e^{Cr\log(1+d_t(x,p))}.
$$
If instead $r\geq 1$, we choose $k \in \mathbb{N}^{+}$ such that $\frac{k-1}{4} \leq r\leq  \frac{k}{4}$, then 
\begin{align*}
|B(x,t,r)|_t \leq & e^{Ck+4nk\log(1+d_t(x,p)) } \leq  Cr^{2n}e^{Cr+Cr\log(1+d_t(x,p))} .
\end{align*}
This completes the proof.
\end{proof}
\begin{lemma}\label{gmp}
For any $t_0,t_1\in [-T, 0]$, $t_0<t_1$, there exist constants $\theta=\theta(t_0)>0$, $C=C(t_0)<\infty$, such that the following statement holds.

If $u\in C^{\infty}(M\times[t_{0},t_{1}])$ is a solution of $(\partial_{t}+\Delta)u=0$, with $|\nabla u(\cdot,t_1)|\leq1$ and $\text{supp}(|\nabla u(\cdot,t_1)|)\subseteq B(p,t_{1},D)$, then we have
\[
|\nabla u|^{2}(y,t_{0})\leq\exp\left(C\sqrt{t_{1}-t_{0}}(D^{2}+1)\right)  ,
\]
for all $y\in M$.
\end{lemma}
\begin{proof}
Throughout the proof, all the constants depend at most on $t_0$.

Let $\alpha\in (0,1)$, $A<\infty$ be constants to be determined. We first use
\begin{align*}
&(\partial_{t}-\Delta_{x})\left(e^{A(t-t_{0})^{\alpha}v(x,t)}K(x,t;y,t_{0})\right)\\
= & \left(A\alpha(t-t_{0})^{\alpha-1}v(x,t)-A^{2}(t-t_{0})^{2\alpha}|\nabla v|^{2}(x,t)+A(t-t_{0})^{\alpha}(\partial_{t}-\Delta)v(x,t)\right) \\
& \qquad\qquad \cdot e^{A(t-t_{0})^{\alpha}v(x,t)}K(x,t;y,t_{0})\\
& -2A(t-t_{0})^{\alpha}\langle\nabla^{\mathbb{R}}v(x,t),\nabla_{x}^{\mathbb{R}}\log K(x,t;y,t_{0})\rangle e^{A(t-t_{0})^{\alpha}v(x,t)}K(x,t;y,t_{0})\\
\geq & \left(A\alpha(t-t_{0})^{\alpha-1}-CA^{2}(t-t_{0})^{2\alpha}-CA(t-t_{0})^{\alpha}\right) v(x,t)  \\
& \qquad\qquad \cdot e^{A(t-t_{0})^{\alpha}v(x,t)}K(x,t;y,t_{0})\\
& -2A(t-t_{0})^{\alpha}\langle\nabla^{\mathbb{R}}v(x,t),\nabla_{x}^{\mathbb{R}}\log K(x,t;y,t_{0})\rangle e^{A(t-t_{0})^{\alpha}v(x,t)}K(x,t;y,t_{0}),
\end{align*}
and (here we use the Schwarz lemma)
\begin{align*}
(\partial_{t}+\Delta_{x})|\nabla u|^{2}(x,t)= & \left(|\nabla\nabla u|^{2}+|\nabla\overline{\nabla}u|^{2}\right)(x,t)+2Rc(\nabla u,\overline{\nabla}u)\\
\geq & \left(|\nabla\nabla u|^{2}+|\nabla\overline{\nabla}u|^{2}\right)(x,t)-2\nabla\overline{\nabla}v(\nabla u,\overline{\nabla}u)-C|\nabla u|^{2} ,
\end{align*}
to estimate (here we use $R\leq Cv$)
\begin{align*}
&\frac{d}{dt}\int_{M}|\nabla u|^{2}(x,t)e^{A(t-t_{0})^{\alpha}v(x,t)}K(x,t;y,t_{0})dg_{t}(x)\\
= & \int_{M}e^{A(t-t_{0})^{\alpha}v(x,t)}K(x,t;y,t_{0})(\partial_{t}+\Delta-R)|\nabla u|^{2}(x,t)  \\
& \qquad\qquad +|\nabla u|^{2}(x,t)(\partial_{t}-\Delta)\left(e^{A(t-t_{0})^{\alpha}v(x,t)}K(x,t;y,t_{0})\right) dg_{t}(x)\\
\geq & \int_{M}\left(\left(|\nabla\nabla u|^{2}+|\nabla\overline{\nabla}u|^{2}\right)(x,t)-2\nabla\overline{\nabla}v(\nabla u,\overline{\nabla}u)-C|\nabla u|^{2}\right)  \\
& \qquad\qquad \cdot e^{A(t-t_{0})^{\alpha}v(x,t)}K(x,t;y,t_{0})dg_{t}(x)\\
& +\left(A\alpha(t-t_{0})^{\alpha-1}-CA^{2}(t-t_{0})^{2\alpha}-CA(t-t_{0})^{\alpha}-C\right)  \\
& \qquad\qquad \cdot \int_{M}|\nabla u|^{2}(x,t)v(x,t)e^{A(t-t_{0})^{\alpha}v(x,t)}K(x,t;y,t_{0})dg_{t}(x)\\
& -2A(t-t_{0})^{\alpha}\int_{M}|\nabla u|^{2}(x,t)\langle\nabla^{\mathbb{R}}v(x,t),\nabla_{x}^{\mathbb{R}}\log K(x,t;y,t_{0})\rangle  \\
& \qquad\qquad  \cdot e^{A(t-t_{0})^{\alpha}v(x,t)}K(x,t;y,t_{0})dg_{t}(x).
\end{align*}
Next, we integrate by parts to obtain
\begin{align*}
-2A(t-t_{0})^{\alpha} & \int_{M}|\nabla u|^{2}(x,t)\langle\nabla^{\mathbb{R}}v(x,t),\nabla_{x}^{\mathbb{R}}\log K(x,t;y,t_{0})\rangle  \\
& \qquad\qquad  \cdot e^{A(t-t_{0})^{\alpha}v(x,t)}K(x,t;y,t_{0})dg_{t}(x)\\
= & -2\int_{M}\langle\nabla^{\mathbb{R}}e^{A(t-t_{0})^{\alpha}v(x,t)},|\nabla u|^{2}(x,t)\nabla_{x}^{\mathbb{R}}K(x,t;y,t_{0})\rangle dg_{t}(x)\\
\geq & 2\int_{M}|\nabla u|^{2}(x,t)\partial_{t}\log K(x,t;y,t_{0})e^{A(t-t_{0})^{\alpha}v(x,t)}K(x,t;y,t_{0})dg_{t}(x)\\
& -2\int_{M}\left(|\nabla\nabla u|+|\nabla\overline{\nabla}u|\right)(x,t)|\nabla u|(x,t)|\nabla_{x}^{\mathbb{R}}\log K(x,t;y,t_{0})|  \\
& \qquad\qquad  \cdot e^{A(t-t_{0})^{\alpha}v(x,t)}K(x,t;y,t_{0})dg_{t}(x) ,
\end{align*}
and
\begin{align*}
& -2\int_{M}\nabla\overline{\nabla} v(\nabla u,\overline{\nabla}u)e^{A(t-t_{0})^{\alpha}v(x,t)}K(x,t;y,t_{0})dg_{t}(x)\\\geq & -2\int_{M}(|\nabla\nabla u|+|\nabla\overline{\nabla}u|)(x,t)|\nabla u|(x,t)|\nabla v|(x,t)e^{A(t-t_{0})^{\alpha}v(x,t)}K(x,t;y,t_{0})dg_{t}(x)\\
& -2\int_{M}|\nabla v|(x,t)|\nabla u|^{2}(x,t)|\nabla_{x}\log K(x,t;y,t_{0})|e^{A(t-t_{0})^{\alpha}v(x,t)}K(x,t;y,t_{0})dg_{t}(x)\\
& +2A(t-t_{0})^{\alpha}\int_{M}|\langle\nabla v,\overline{\nabla}u\rangle|^{2}e^{A(t-t_{0})^{\alpha}v(x,t)}K(x,t;y,t_{0})dg_{t}(x)
\end{align*}
We also apply our Harnack inequality, say Theorem \ref{HarnackiofKRF}, to estimate
\[
\partial_{t}\log K(x,t;y,t_{0})\geq\frac{1}{2}|\nabla\log K(x,t;y,t_{0})|^{2}-C\left(\frac{1}{t-t_{0}}+v(x,t)\right).
\]
Combining all of these expressions yields (where $\epsilon>0$ is to be determined)
\begin{align*}
& \frac{d}{dt}\int_{M}|\nabla u|^{2}(x,t)e^{A(t-t_{0})^{\alpha}v(x,t)}K(x,t;y,t_{0})dg_{t}(x) \\ 
\geq & \int_{M}\left(\left(|\nabla\nabla u|^{2}+|\nabla\overline{\nabla}u|^{2}\right)(x,t)-C|\nabla u|^{2}\right)e^{A(t-t_{0})^{\alpha}v(x,t)}K(x,t;y,t_{0})dg_{t}(x)\\
& +\left(A\alpha(t-t_{0})^{\alpha-1}-CA^{2}(t-t_{0})^{2\alpha}-CA(t-t_{0})^{\alpha}-C\right)  \\
& \qquad\qquad  \cdot \int_{M}|\nabla u|^{2}(x,t)v(x,t)e^{A(t-t_{0})^{\alpha}v(x,t)}K(x,t;y,t_{0})dg_{t}(x)\\
& -2\int_{M}(|\nabla\nabla u|+|\nabla\overline{\nabla}u|)(x,t)|\nabla u|(x,t)|\nabla v|(x,t)e^{A(t-t_{0})^{\alpha}v(x,t)}K(x,t;y,t_{0})dg_{t}(x)\\
& -2\int_{M}|\nabla v|(x,t)|\nabla u|^{2}(x,t)|\nabla_{x}\log K(x,t;y,t_{0})|e^{A(t-t_{0})^{\alpha}v(x,t)}K(x,t;y,t_{0})dg_{t}(x)\\
& +2A(t-t_{0})^{\alpha}\int_{M}|\langle\nabla v,\overline{\nabla}u\rangle|^{2}e^{A(t-t_{0})^{\alpha}v(x,t)}K(x,t;y,t_{0})dg_{t}(x)\\
& +\epsilon\int_{M}|\nabla u|^{2}(x,t)|\nabla_x\log K(x,t;y,t_{0})|^{2}e^{A(t-t_{0})^{\alpha}v(x,t)}K(x,t;y,t_{0})dg_{t}(x)\\
& -C\epsilon\int_{M}\left(\frac{1}{t-t_{0}}+v(x,t)\right)|\nabla u|^{2}(x,t)e^{A(t-t_{0})^{\alpha}v(x,t)}K(x,t;y,t_{0})dg_{t}(x)\\
& -2\epsilon\int_{M}\left(|\nabla\nabla u|+|\nabla\overline{\nabla}u|\right)(x,t)|\nabla u|(x,t)|\nabla_{x}\log K(x,t;y,t_{0})|  \\
& \qquad\qquad  \cdot e^{A(t-t_{0})^{\alpha}v(x,t)}K(x,t;y,t_{0})dg_{t}(x)\\
& -2(1-\epsilon)A(t-t_{0})^{\alpha}\int_{M}|\nabla u|^{2}(x,t)|\nabla v|(x,t)|\nabla_{x}\log K(x,t;y,t_{0})|  \\
& \qquad\qquad  \cdot e^{A(t-t_{0})^{\alpha}v(x,t)}K(x,t;y,t_{0})dg_{t}(x)
\end{align*}
Next, we use Cauchy's inequality to estimate
\begin{align*}
-2\int_{M}(|\nabla\nabla u|&+|\nabla\overline{\nabla}u|)(x,t)|\nabla u|(x,t)|\nabla v|(x,t)e^{A(t-t_{0})^{\alpha}v(x,t)}K(x,t;y,t_{0})dg_{t}(x)\\ \geq & -\frac{1}{4}\int_{M}\left(|\nabla\nabla u|^{2}+|\nabla\overline{\nabla}u|^{2}\right)(x,t)e^{A(t-t_{0})^{\alpha}v(x,t)}K(x,t;y,t_{0})dg_{t}(x)\\
 & -8\int_{M}|\nabla v|^{2}(x,t)|\nabla u|^{2}(x,t)e^{A(t-t_{0})^{\alpha}v(x,t)}K(x,t;y,t_{0})dg_{t}(x),
\end{align*}
\begin{align*}
&  -2\epsilon\int_{M}\left(|\nabla\nabla u|+|\nabla\overline{\nabla}u|\right)(x,t)|\nabla u|(x,t)|\nabla_{x}\log K(x,t;y,t_{0})|  \\
& \qquad\qquad  \cdot e^{A(t-t_{0})^{\alpha}v(x,t)}K(x,t;y,t_{0})dg_{t}(x)  \\ 
\geq & -\frac{1}{4}\int_{M}\left(|\nabla\nabla u|^{2}+|\nabla\overline{\nabla}u|^{2}\right)(x,t)e^{A(t-t_{0})^{\alpha}v(x,t)}K(x,t;y,t_{0})dg_{t}(x)\\
 & -8\epsilon^{2}\int_{M}|\nabla u|^{2}(x,t)|\nabla_{x}\log K(x,t;y,t_{0})|^{2}e^{A(t-t_{0})^{\alpha}v(x,t)}K(x,t;y,t_{0})dg_{t}(x).
\end{align*}
Again combining expressions, we obtain
\begin{align*}
&\frac{d}{dt}\int_{M}|\nabla u|^{2}(x,t)e^{A(t-t_{0})^{\alpha}v(x,t)}K(x,t;y,t_{0})dg_{t}(x) \\
\geq & \left(A\alpha(t-t_{0})^{\alpha-1}-CA^{2}(t-t_{0})^{2\alpha}-CA(t-t_{0})^{\alpha}-C\right)  \\
& \qquad\qquad  \cdot  \int_{M}|\nabla u|^{2}(x,t)v(x,t)e^{A(t-t_{0})^{\alpha}v(x,t)}K(x,t;y,t_{0})dg_{t}(x)\\
& +\epsilon(1-8\epsilon)\int_{M}|\nabla u|^{2}(x,t)|\nabla_x\log K(x,t;y,t_{0})|^{2}e^{A(t-t_{0})^{\alpha}v(x,t)}K(x,t;y,t_{0})dg_{t}(x)\\
& -\frac{C\epsilon}{t-t_{0}}\int_{M}|\nabla u|^{2}(x,t)e^{A(t-t_{0})^{\alpha}v(x,t)}K(x,t;y,t_{0})dg_{t}(x)\\
& -2\left((1-\epsilon)A(t-t_{0})^{\alpha}+1\right)\int_{M}|\nabla v|(x,t)|\nabla u|^{2}(x,t)|\nabla_{x}\log K(x,t;y,t_{0})|  \\
& \qquad\qquad  \cdot e^{A(t-t_{0})^{\alpha}v(x,t)}K(x,t;y,t_{0})dg_{t}(x).
\end{align*}
Next, Cauchy's inequality gives
\begin{align*}
-2 & \left((1-\epsilon)A(t-t_{0})^{\alpha}+1\right)\int_{M}|\nabla v|(x,t)|\nabla u|^{2}(x,t)|\nabla_{x}\log K(x,t;y,t_{0})|  \\
& \qquad\qquad  \cdot e^{A(t-t_{0})^{\alpha}v(x,t)}K(x,t;y,t_{0})dg_{t}(x)\\
\geq & -\frac{\epsilon}{2}\int_{M}|\nabla u|^{2}(x,t)|\nabla_x\log K(x,t;y,t_{0})|^{2}e^{A(t-t_{0})^{\alpha}v(x,t)}K(x,t;y,t_{0})dg_{t}(x)\\
& -C\left(A(t-t_{0})^{\alpha}+1\right)^2\epsilon^{-1}\int_{M}|\nabla v|^{2}(x,t)|\nabla u|^{2}(x,t)  \\
& \qquad\qquad  \cdot e^{A(t-t_{0})^{\alpha}v(x,t)}K(x,t;y,t_{0})dg_{t}(x)  ,
\end{align*}
so that (assuming $|t-t_{0}|$ is sufficiently small and $\alpha\in(0,1)$)
\begin{align*}
& \frac{d}{dt}\int_{M}|\nabla u|^{2}(x,t)e^{A(t-t_{0})^{\alpha}v(x,t)}K(x,t;y,t_{0})dg_{t}(x)\\
\geq & \left(A\alpha(t-t_{0})^{\alpha-1}-C\epsilon^{-1}\right)\int_{M}|\nabla u|^{2}(x,t)v(x,t)e^{A(t-t_{0})^{\alpha}v(x,t)}K(x,t;y,t_{0})dg_{t}(x)\\
& -\frac{C\epsilon}{t-t_{0}}\int_{M}|\nabla u|^{2}(x,t)e^{A(t-t_{0})^{\alpha}v(x,t)}K(x,t;y,t_{0})dg_{t}(x).
\end{align*}
If we choose $\epsilon:=(t-t_{0})^{1-\alpha}$, then assuming $A<\infty$ is sufficiently large, we obtain
\begin{align*}
& \frac{d}{dt}\int_{M}|\nabla u|^{2}(x,t)e^{A(t-t_{0})^{\alpha}v(x,t)}K(x,t;y,t_{0})dg_{t}(x)\\
\geq & -\frac{C}{(t-t_{0})^{\alpha}}\int_{M}|\nabla u|^{2}(x,t)e^{A(t-t_{0})^{\alpha}v(x,t)}K(x,t;y,t_{0})dg_{t}(x),
\end{align*}
or equivalently,
\[
\frac{d}{dt}\log\left(\int_{M}|\nabla u|^{2}(x,t)e^{A(t-t_{0})^{\alpha}v(x,t)}K(x,t;y,t_{0})dg_{t}(x)\right)\geq-\frac{C}{(t-t_{0})^{\alpha}}.
\]
Integrating from $t_{0}$ to $t_{1}$ then yields 
\[
|\nabla u|^{2}(y,t_{0})\leq\exp\left(C(t_{1}-t_{0})^{1-\alpha}\right)\int_{M}|\nabla u|^{2}(x,t_{1})e^{A(t_{1}-t_{0})^{\alpha}v(x,t_{1})}K(x,t_{1};y,t_{0})dg_{t}(x).
\]
If $|\nabla u|(\cdot,t_{1})\leq1$ and $\text{supp}(|\nabla u(\cdot,t_1)|)\subseteq B(p,t_{0},D)$, then
\begin{align*}
|\nabla u|^{2}(y,t_{0})\leq & \exp\left(C(t_{1}-t_{0})^{1-\alpha}+C(t_{1}-t_{0})^{\alpha}(D^{2}+1)\right).
\end{align*}
For optimal estimates, we choose $\alpha=\frac{1}{2}$. This completes the proof.
\end{proof}
Now we can prove the good distance distortion upper bound.
\begin{proposition}\label{gddub2}
For any $A<T$, $D<\infty$, there exist constants $\theta=\theta(A,D)>0$, $C=C(A,D)<\infty$, such that the following statement holds.

Assume $-A\leq t_{1}< t_{2}\leq0$ satisfies $t_2-t_1\leq \theta$. Assume $x_1, x_2 \in\bigcup_{ t\in[ t_{1}, t_{2} ] } B( p, t, D)$. Then we have
$$
d_{t_{2}}(x_1, x_2)\leq d_{t_{1}}(x_1, x_2) + C\sqrt{t_{2}-t_{1}}.
$$
\end{proposition}
\begin{proof}
Throughout the proof, all the constants may depend on $A, D$. We let $u_{i}\in C^{\infty}(M\times[t_{1},t_{2}))\cap C^{0}(M\times[t_{1},t_{2}])$ solve 
\begin{align*}
(-\partial_{t}-\Delta)u_{i}= & 0  \\
u_{i}(\cdot,t_{2})= & \min\{d_{t_{2}}(x_{i},\cdot),2d_{t_{2}}(x_{1},x_{2})\}.
\end{align*}
Then $u_{i}$ satisfies the hypotheses of Lemma \ref{gmp} and $\max_{M}|\nabla u_{i}|(\cdot,t_2)=1$, so we can estimate
\[
|\nabla u_{i}|(x,t_{1})\leq 1+C\sqrt{t_{2}-t_{1}} ,
\]
for all $x\in M$. Moreover, the triangle inequality gives $u_{1}(\cdot,t_{2})+u_{2}(\cdot,t_{2})\geq d_{t_{2}}(x_{1},x_{2})$, so the maximum principle implies 
\[
d_{t_{2}}(x_{1},x_{2})\leq u_{1}+u_{2}\leq4d_{t_{2}}(x_{1},x_{2}).
\]
In particular, $u_{1}(x_{2},t_{1})\geq d_{t_{2}}(x_{1},x_{2})-u_{2}(x_{2},t_{1}),$ so that 
\begin{align*}
d_{t_{2}}(x_{1},x_{2})-\left(u_{1}(x_{1},t_{1})+u_{2}(x_{2},t_{1})\right) \leq & u_{1}(x_{2},t_{1})-u_{1}(x_{1},t_{1})  \\
\leq & \left(1+C\sqrt{t_{2}-t_{1}}\right)d_{t_{1}}(x_{1},x_{2}).  
\end{align*}
By \cite[Proposition 7.6]{JST23a}, if we choose $\theta$ small enough, then we have 
$$d_{t}(x_{i},p)\leq C ,$$
for all $t\in[ t_{1}, t_{2} ]$, hence $d_{t}(x_{1},x_{2})\leq C$, hence we obtain
\begin{equation}\label{dt2boundbyu}
d_{t_{2}}(x_{1},x_{2})\leq d_{t_{1}}(x_1, x_2) + C\sqrt{t_{2}-t_{1}} + \left(u_{1}(x_{1},t_{1})+u_{2}(x_{2},t_{1})\right) .    
\end{equation}
It remains to estimate $u_{i}(x_{i},t_{1})$.

For any $x\in M$, $t\in[ t_{1}, t_{2} ]$, we have
\begin{equation}\label{lscbbyv}
R(x,t)\leq Cv(x,t)\leq Cv(x,t_1)\leq C(1+d_{t_1}^2(x, p))\leq C(1+d_{t_1}^2(x, x_i)) ,
\end{equation}
hence by Lemma \ref{ghke0}, if we choose $\theta$ small enough, then we have
\begin{equation}\label{hkupbaxi}
K(x,t;x_i,t_1)\leq \frac{C}{(t-t_1)^{n}}\exp\left(-\frac{d_{t_1}^{2}(x,x_i)}{C(t-t_1)}\right)  ,
\end{equation}
for all $x\in M$, $t\in[ t_{1}, t_{2} ]$. We then claim that, if we choose $\theta$ small enough, then for any $t\in ( t_{1}, t_{2} ]$, $k=0,1,2$, we have
\begin{equation}\label{ihkupbaxi}
\int_{M} K(x,t;x_{i},t_{1})\exp\left( d_{t_1}^2(x, x_i) \right) d_{t_1}^k(x, x_i)  dg_{t_1} (x) \leq C(t-t_1)^{k/2} .
\end{equation}
Indeed, for $j=0,1,2\dots$, by (\ref{hkupbaxi}) and Proposition \ref{roughvb} we have
\begin{align*}
& \int_{B(x_i, t_1, 2^{j+1}(t-t_1)^{1/2})\setminus B(x_i, t_1, 2^{j}(t-t_1)^{1/2})} K(x,t;x_{i},t_{1})\exp\left( d_{t_1}^2(x, x_i) \right) d_{t_1}^k(x, x_i)  dg_{t_1} (x) \\
& \leq \int_{B(x_i, t_1, 2^{j+1}(t-t_1)^{1/2})\setminus B(x_i, t_1, 2^{j}(t-t_1)^{1/2})} \frac{C}{(t-t_1)^{n}}\exp\left(-\frac{d_{t_1}^{2}(x,x_i)}{C(t-t_1)}\right)  \\
& \qquad\qquad \cdot\exp\left( d_{t_1}^2(x, x_i) \right) d_{t_1}^k(x, x_i)  dg_{t_1} (x)   \\
& \leq \frac{C}{(t-t_1)^{n}} \exp\left(-C^{-1}2^{2j}\right) 2^{jk}(t-t_1)^{k/2} |B(x_i, t_1, 2^{j+1}(t-t_1)^{1/2})|_{t}  \\
& \leq \frac{C}{(t-t_1)^{n}} \exp\left(-C^{-1}2^{2j}\right) 2^{jk}(t-t_1)^{k/2}2^{2nj}(t-t_1)^{n} \exp\left(C2^{j+1}(t-t_1)^{1/2}\right)  \\
& \leq C \exp\left(-C^{-1}2^{2j}\right)(t-t_1)^{k/2}  ,
\end{align*}
summing over $j$ proves (\ref{ihkupbaxi}).

Now from $(-\partial_{t}-\Delta+R)u_{i}=Ru_{i}$, and the fact that $dg_t\leq Cdg_{t_1}$ for all $t\in [t_1, t_2]$, we have
\begin{align*}
u_{i}(x_{i},t_{1}) \leq C\int_{M} & K(x,t_{2};x_{i},t_{1})d_{t_{2}}(x,x_{i})dg_{t_{1}}(x) \\
& +Cd_{t_{2}}(x_{1},x_{2})\int_{t_{1}}^{t_{2}}\int_{M}K(x,t;x_{i},t_{1})R(x,t)dg_{t_1}(x)dt . \\
\end{align*}
For the first term, by Proposition \ref{gwdde} and (\ref{ihkupbaxi}), if we let $\theta$ small enough, we have
\begin{align*}
& \int_{M} K(x,t_{2};x_{i},t_{1})d_{t_{2}}(x,x_{i})dg_{t_{1}}(x) \\
\leq & \int_{B(x_i, t_1, (t_2-t_1)^{1/2})} K(x,t_{2};x_{i},t_{1})C(t_2-t_1)^{1/2} dg_{t_{1}}(x)  \\
& \qquad  + \int_{M} K(x,t_{2};x_{i},t_{1})C\exp\left( d_{t_1}^2(x, x_i) \right)d_{t_{1}}(x,x_{i})dg_{t_{1}}(x)  \\
\leq &  C (t_2-t_1)^{1/2} .
\end{align*}
For the second term, by (\ref{lscbbyv}) and (\ref{ihkupbaxi}), we have
\begin{align*}
& \int_{t_{1}}^{t_{2}}\int_{M}K(x,t;x_{i},t_{1})R(x,t)dg_{t_1}(x)dt \\
\leq & \int_{t_{1}}^{t_{2}}\int_{M}K(x,t;x_{i},t_{1})C(1+d_{t_1}^2(x, x_i))dg_{t_1}(x)dt  \leq C (t_2-t_1) .
\end{align*}
Combining expressions gives
\[
u_{i}(x_{i},t_{1})\leq C\sqrt{t_{2}-t_{1}},
\]
combine this with (\ref{dt2boundbyu}), we finish the proof.
\end{proof}
%


%
\medskip
\subsection{Continuity of the limiting metric flow}\label{ctoflmf}

In this section, we first use the good distance distortion estimates to prove the Gromov-Hausdorff continuity of the limiting metric flow $\cX$, then we prove the Gromov-$W_1$ continuity.

\begin{proposition}\label{GHctoflmf}
For every $t\in (-\infty, 0]$, we have pointed Gromov-Hausdorff convergence 
\[
(M_{i},d_{g_{i,t}},p_{i})\to(\mathcal{X}_{t},d_{t},q_{t}),
\]
where $q_{t}\in\mathcal{X}_{t}$ satisfy $d^{\cX_s}_{W_{1}}(\nu_{q_{t};s},\delta_{q_{s}})\leq C\sqrt{t-s}$ for $s,t\in (-\infty, 0]$ with $s\leq t$. Moreover, the convergence is locally uniform in time in the sense that 
\[
\lim_{i\to\infty}\sup_{t\in J}d_{PGH}\left((M_{i},d_{g_{i,t}},p_{i}),(\mathcal{X}_{t},d_{t},q_{t})\right)=0
\]
for any compact subset $J\subseteq (-\infty, 0]$. Finally, the map
\[
t\mapsto(\mathcal{X}_{t},d_{t},q_{t}), ~~ t\in (-\infty, 0],
\]
is continuous in the pointed Gromov-Hausdorff topology. 
\end{proposition}
\begin{proof}
Let $I'\subseteq (-\infty, 0]$ be the set of times where $\bF$-convergence (\ref{FcorKRF'3}) is time-wise; by passing to a subsequence, we may assume $|(-\infty, 0]\setminus I'|=0$. Let $\psi_{i}:U_{i}\to M_{i}$ be the diffeomorphisms realizing locally smooth convergence of the $\bF$-convergence. Let $I''\subseteq I'$ be a dense countable subset of $(-\infty, 0]$. By \cite[Theorem 7.3]{JST23a}, passing to a subsequence, for every $t\in I''$, we have 
$$
\lim_{i\to\infty}d_{PGH}\left((M_{i},d_{g_{i,t}},p_{i}),(\mathcal{X}_{t},d_{t},q_{t})\right)=0 ,
$$
for some $q_t\in \mathcal{X}_{t}$. Moreover, by \cite[Proposition 2.7]{Hal}, we can pass to a subsequence so that $\psi_{i}$ extend to (not necessarily continuous) $r_{i}^{-1}$-Gromov-Hausdorff approximations
\[
\psi_{i,t}:B(q_{t},r_{i})\to B(p_{i}, t, r_{i})
\]
for all $t\in I''$, where $r_{i}\nearrow\infty$ and $q_{t}\in\mathcal{X}_{t}$. 

Note that if $t\in (-\infty, 0]$ and $x_{i}\in M_{i}$ are such that $(x_{i},t)\to x_{\infty}\in\mathcal{X}_{t}$
with respect to the Gromov-Hausdorff convergence, this means $\varphi_{t}^{i}(x_{i})\to\varphi_{t}^{\infty}(x_{\infty})$ in $(Z_{t},d_{t}^{Z})$, so $x_{i}$ converge strictly to $x_{\infty}$ within $\mathfrak{C}$ in the sense of \cite[Definition 6.22]{Bam20b}. By \cite[Theorem 6.23]{Bam20b}, we thus have $(x_{i},t)\xrightarrow[i\to\infty]{\mathfrak{C}}x_{\infty}$. In particular, we then have
\[
\lim_{i\to\infty}d^{\cX_s}_{W_{1}}\left((\varphi_{s}^{i})_{\ast}\nu_{x_{i},t;s}^{i},(\varphi_{s}^{\infty})_{\ast}\nu_{x_{\infty};s}\right)=0
\]
for all $s\in(-\infty,t)\cap I'$. 

\begin{claim}\label{W1betqtqs}
$d_{W_{1}}^{\mathcal{X}_{s}}(\nu_{q_{t};s},\delta_{q_{s}})\leq C(s)\sqrt{t-s}$
for all $s,t\in I''$ with $s\leq t$, $t-s\leq \theta(s)$, where $\theta(s)>0$ is a small constant.     
\end{claim}

%
\begin{proof}
For sufficiently large $i\in\mathbb{N}$, we can estimate
\begin{align*}
d_{W_{1}}^{\mathcal{X}_{s}}(\nu_{q_{t};s},\delta_{q_{s}})\leq & d_{W_{1}}^{Z_{s}}\left((\varphi_{s}^{\infty})_{\ast}\nu_{q_{t};s},(\varphi_{s}^{i})_{\ast}\nu_{\psi_{i,t}(q_{t}),t;s}^{i}\right)+d_{W_{1}}^{g_{i,s}}\left(\nu_{\psi_{i,t}(q_{t}),t;s}^{i},\delta_{\psi_{i,t}(q_{t})}\right)\\
& +d_{g_{i,s}}(\psi_{i,t}(q_{t}),\psi_{i,s}(q_{s}))+d_{s}^{Z}\left((\varphi_{s}^{i}\circ\psi_{i,s})(q_{s}),\varphi_{s}^{\infty}(q_{s})\right),
\end{align*}
and previous arguments give 
$$
\lim_{i\to\infty}d_{W_{1}}^{Z_{s}}\left((\varphi_{s}^{\infty})_{\ast}\nu_{q_{t};s},(\varphi_{s}^{i})_{\ast}\nu_{\psi_{i,t}(q_{t}),t;s}^{i}\right)=0=\lim_{i\to\infty}d_{s}^{Z}\left((\varphi_{s}^{i}\circ\psi_{i,s})(q_{s}),\varphi_{s}^{\infty}(q_{s})\right) .
$$
Moreover, Lemma \ref{W1balongwl} gives
\[
d_{W_{1}}^{g_{i,s}}\left(\nu_{\psi_{i,t}(q_{t}),t;s}^{i},\delta_{\psi_{i,t}(q_{t})}\right)\leq C(s)\sqrt{t-s}.
\]
Using the short-time distortion estimate Proposition \ref{gddlb}, we moreover have
\[
d_{g_{i,s}}(\psi_{i,t}(q_{t}),\psi_{i,s}(q_{s}))\leq d_{g_{i,t}}(\psi_{i,t}(q_{t}),p_{i})+d_{g_{i,s}}(\psi_{i,s}(q_{s}),p_{i})+C(s)\sqrt{t-s},
\]
for all $t$ close to $s$, so combining estimates gives the claim.
\end{proof}

\begin{claim}\label{extoggam}
$\gamma:I''\to\mathcal{X}$, $t\mapsto q_{t}$ uniquely extends to a continuous path $\gamma:(-\infty,0)\to\mathcal{X}$ (with respect to the natural topology of $\mathcal{X}$), satisfying 
\[
d_{W_{1}}^{\mathcal{X}_{s}}(\nu_{\gamma(t);s},\delta_{\gamma(s)})\leq C(s)\sqrt{t-s} ,
\]
for all $s\leq t$, $t-s\leq \theta(s)$, where $\theta(s)>0$ is a small constant. 
\end{claim}
\begin{proof}
Fix $t_{0}\in (-\infty,0)$, and define $\mu_{t_{0}}:I''\cap(t_{0},0)\to\mathcal{P}(\mathcal{X}_{t_{0}})$,
$t\mapsto\nu_{\gamma(t);t_{0}}$. For any $t_{1},t_{2}\in I''\cap(t_{0}, 0)$
with $t_{1}\leq t_{2}$, Claim \ref{W1betqtqs} gives
\[
d_{W_{1}}^{\mathcal{X}_{t_{0}}}(\mu_{t_{0}}(t_{1}),\mu_{t_{0}}(t_{2}))\leq d_{W_{1}}^{\mathcal{X}_{t_{1}}}(\nu_{q_{t_{2}};t_{1}},\delta_{q_{t_{1}}})\leq C(t_0)\sqrt{t_{2}-t_{1}}.
\]
In particular, $\mu_{t_{0}}$ is local $\frac{1}{2}$-H\"older continuous, so extends uniquely to a local $\frac{1}{2}$-H\"older continuous map $\mu_{t_{0}}:[t_{0},0)\to\mathcal{P}(\mathcal{X}_{t_{0}})$. Moreover, by \cite[Lemma 2.10]{Bam20b}, $\text{Var}(\mu_{t_{0}}(t))\leq H_{2n}(t-t_{0})$ passes to the limit to give $\text{Var}(\mu_{t_{0}}(t_{0}))=0$. That is, $\mu_{t_{0}}(t_{0})=\delta_{q_{t_{0}}}$for some $q_{t_{0}}\in\mathcal{X}_{t_{0}}$; this allows us to define $\gamma$ as (not a priori continuous) map $(-\infty, 0)\to\mathcal{X}$. Also observe that 
\begin{equation}\label{W1eagam}
d_{W_{1}}^{\cX_{t_0}}(\delta_{\gamma(t_{0})},\nu_{\gamma(t);t_{0}})\leq C(t_0)\sqrt{t-t_{0}} ,    
\end{equation}
for any $t\in[t_{0},t_{0}+\theta(t_{0}))\cap I''$; by Claim \ref{W1betqtqs}, we have $\gamma|_{I''}$ agrees with its original definition. 

Suppose $t_{i}\in (-\infty,0)$ satisfy $t_{i}\to t\in (-\infty , 0)$. Fix $s\in(-\infty,t)\cap I''$, choose $t_{i}^{\ast}\in(\max(t_{i},t),\max(t_{i},t)+i^{-1})\cap I''$, by (\ref{W1eagam}) we can estimate
\begin{align*}
d_{W_{1}}^{\mathcal{X}_{s}}(\nu_{\gamma(t_{i});s},\nu_{\gamma(t);s})\leq & d_{W_{1}}^{\mathcal{X}_{s}}(\nu_{\gamma(t_{i});s},\nu_{\gamma(t_{i}^{\ast});s})+d_{W_{1}}^{\mathcal{X}_{s}}(\nu_{\gamma(t);s},\nu_{\gamma(t_{i}^{\ast});s})\\
\leq & d_{W_{1}}^{\mathcal{X}_{t_{i}}}(\delta_{\gamma(t_{i})},\nu_{\gamma(t_{i}^{\ast});t_{i}})+d_{W_{1}}^{\mathcal{X}_{t}}(\delta_{\gamma(t)},\nu_{\gamma(t_{i}^{\ast});t})\\
\leq & C(s)\sqrt{i^{-1}+|t-t_{i}|}.
\end{align*}
Therefore, we have $\lim_{i\to\infty}d_{W_{1}}^{\mathcal{X}_{s}}(\nu_{\gamma(t_{i});s},\nu_{\gamma(t);s})=0$, hence $\gamma:(-\infty,0)\to\mathcal{X}$ is continuous. 
\end{proof}
\begin{claim}\label{nonvanofvol}
For any $t\in (-\infty , 0)\cap I'$, $D<\infty$ and $r>0$, there exists $c=c(t, D, r)>0$, such that $\nu_{p_{\infty};t}^{\infty}(B(x,r))\geq c$ for all $x\in B(p_{t},D)$.     
\end{claim}
\begin{proof}
In fact, for any such $x$, we know that 
\[
\nu_{p_{i},0;t}^{i}\left( B(\psi_{i,t}(x),\frac{r}{2})\right) \geq c(t,D, r)>0,
\]
for all sufficiently large $i\in\mathbb{N}$, see the proof of \cite[Theorem 7.3]{JST23a}. Then the weak convergence
\[
(\varphi_{t}^{i})_{\ast}\nu_{p_{i},0;t}^{i}\to(\varphi_{t})_{\ast}\nu_{p_{\infty};t}^{\infty}
\]
and $(\varphi_{t}^{i}\circ\psi_{i,t})(x)\to\varphi_{t}^{\infty}(x)$ imply
\begin{align*}
\nu_{p_{\infty};t}^{\infty}(B(x,r))\geq & \left((\varphi_{t})_{\ast}\nu_{p_{\infty};t}^{\infty}\right)(B^{Z_{t}}(\varphi_{t}(x),r))  \\
\geq & \limsup_{i\to\infty}\left((\varphi_{t}^{i})_{\ast}\nu_{p_{i},0;t}^{i}\right)\left(\overline{B}^{Z_{t}}(\varphi_{t}(x),\frac{3r}{4})\right)\\
\geq & \limsup_{i\to\infty}\left((\varphi_{t}^{i})_{\ast}\nu_{p_{i},0;t}^{i}\right)\left(B^{Z_{t}}((\varphi_{t}^{i}\circ\psi_{i,t})(x),\frac{r}{2})\right)\\
= & \limsup_{i\to\infty}\nu_{p_{i},0;t}^{i}(B(\psi_{i,t}(x),\frac{r}{2}))\geq c(t,D,r).
\end{align*}
\end{proof}

\begin{claim}\label{GHconoflim}
For every $t_{0}\in (-\infty, 0)$, $(\mathcal{X}_{t_{0}},d_{t_{0}},q_{t_{0}})$ is the pointed Gromov-Hausdorff limit of $(\mathcal{X}_{t},d_{t},q_{t})$ as $t\searrow t_{0}$ for $t\in I''$.
\end{claim}
\begin{proof}
Because $\mathcal{X}$ is future-continuous, we know that 
\[
\lim_{I''\ni t\searrow t_{0}}d_{GW_{1}}\left((\mathcal{X}_{t},d_{t},\nu_{p_{\infty};t}^{\infty}),(\mathcal{X}_{t_{0}},d_{t_{0}},\nu_{p_{\infty};t_0}^{\infty})\right)=0.
\]
Moreover, Claim \ref{nonvanofvol} implies that the measures $\nu_{p_{\infty};t}^{\infty},\nu_{p_{\infty};t_0}^{\infty}$ satisfy the hypotheses of \cite[Proposition 2.7]{Hal}, so that after passing to a subsequence, we have
\[
\lim_{I''\ni t\searrow t_{0}}d_{PGH}\left((\mathcal{X}_{t},d_{t},q_{t}),(\mathcal{X}_{t_{0}},d_{t_{0}},q_{t_{0}}')\right)=0
\]
for some $q_{t_{0}}'\in\mathcal{X}_{t_{0}}$. Moreover, the Gromov-Hausdorff convergence can be realized by metric embeddings 
\[
\phi_{t}':(\mathcal{X}_{t},d_{t})\hookrightarrow Z_{t}'
\]
\[
\phi_{t_0}':(\mathcal{X}_{t_{0}},d_{t_{0}})\hookrightarrow Z_{t}'
\]
such that 
\[
\lim_{I''\ni t\searrow t_{0}}d_{W_{1}}^{Z_{t}'}((\phi_{t}')_{\ast}\nu_{p_{\infty};t}^{\infty},(\phi_{t_0}')_{\ast}\nu_{p_{\infty};t_0}^{\infty})=0.
\]
The proof of \cite[Lemma 4.18]{Bam20b} shows that we can choose $Z_{t}':=\mathcal{X}_{t_{0}}\sqcup\mathcal{X}_{t}$, along with the natural embeddings, where 
\[
d_{t}^{Z'}(x,y):=\inf_{z\in\mathcal{X}_{t}}\left(d_{t}(y,z)+d_{W_{1}}^{\mathcal{X}_{t_{0}}}(\nu_{z;t_{0}},\delta_{x})\right)+\epsilon_{t}
\]
for $x\in\mathcal{X}_{t_{0}}$ and $y\in\mathcal{X}_{t}$, where $\lim_{t\searrow t_{0}}\epsilon_{t}=0$. Claim \ref{extoggam} then implies 
\[
d_{W_{1}}^{\mathcal{X}_{t_{0}}}(\nu_{q_{t};t_{0}},\delta_{q_{t_{0}}})\leq C(t_{0})\sqrt{t-t_{0}},
\]
for all $t$ close to $t_0$, so that 
\[
d_{t}^{Z'}(q_{t_{0}}, q_{t})\leq d_{W_{1}}^{\mathcal{X}_{t_{0}}}(\nu_{q_{t};t_{0}},\delta_{q_{t_{0}}})+\epsilon_{t},
\]
hence $\lim_{t\searrow t_{0}}d_{t}^{Z'}(q_{t_{0}},q_{t})=0$. That is, $q_{t}$ converge to $q_{t_{0}}$ with respect to the Gromov-Hausdorff convergence, so we must have $q_{t_{0}}'=q_{t_{0}}$. 
\end{proof}

Now, for any $t\in (-\infty,0)$, choose a sequence $t_{i}\in I''$ with $t_{i}\searrow t$. Then we can estimate 
\begin{align*}
d_{PGH} &\left((M_{i},d_{g_{i,t}},p_{i}),(\mathcal{X}_{t},d_{t},q_{t})\right)\leq d_{PGH}\left((M_{i},d_{g_{i,t}},p_{i}),(M_{i},d_{g_{i,t_{i}}},p_{i})\right)\\
& +d_{PGH}\left((M_{i},d_{g_{i,t_{i}}},p_{i}),(\mathcal{X}_{t_{i}},d_{t_{i}},p_{t_{i}})\right) +d_{PGH}\left((\mathcal{X}_{t_{i}},d_{t_{i}},p_{t_{i}}),(\mathcal{X}_{t},d_{t},q_{t})\right),
\end{align*}
but $\lim_{i\to\infty}d_{PGH}\left((M_{i},d_{g_{i,t}},p_{i}),(M_{i},d_{g_{i,t_{i}}},p_{i})\right)=0$
via the identity map and locally uniform distortion estimates, say Proposition \ref{gddlb} and Proposition \ref{gddub2}, 
$$
\lim_{i\to\infty}d_{PGH}\left((M_{i},d_{g_{i,t_{i}}},p_{i}),(\mathcal{X}_{t_{i}},d_{t_{i}},q_{t_{i}})\right)=0
$$
by the Gromov-Hausdorff convergence at times in $I''$, and 
$$
\lim_{i\to\infty}d_{PGH}\left((\mathcal{X}_{t_{i}},d_{t_{i}},q_{t_{i}}),(\mathcal{X}_{t},d_{t},q_{t})\right)=0
$$
by Claim \ref{GHconoflim}. Thus, the locally uniformly continuous maps
\[
t\mapsto(M_{i},d_{g_{i,t}},p_{i})
\]
converge pointwise to $t\mapsto(\mathcal{X}_{t},d_{t},q_{t})$; this
implies that the convergence is actually locally uniform in $t$,
and in particular that the limit is continuous. 
\end{proof}

\begin{proposition}\label{GW1conofX}
$\mathcal{X}$ is a Gromov-$W_1$ continuous metric flow.
\end{proposition}
\begin{proof}
We use the notations from the Proposition \ref{GHctoflmf}.

Let $t_0\in (-\infty, 0)$ be any fixed time. By Theorem 4.31 in \cite{Bam20b}, it suffices to show that for all $x_{1},x_{2}\in\mathcal{X}_{t_{0}}$, we have
\[
\lim_{t\nearrow t_{0}}d_{W_{1}}^{\mathcal{X}_{t}}(\nu_{x_{1};t},\nu_{x_{2};t})\geq d_{t_{0}}(x_{1},x_{2}).
\]
Fix $\epsilon>0$, and let $I''$ be as in the Proposition \ref{GHctoflmf}. By an extension of Proposition 4.40 of \cite{Bam20b} (using the fact that $\mathcal{X}$ is future-continuous), we have

\begin{claim}\label{perofxj}
For all $t'\in I''\cap(t_{0},\infty)$ sufficiently close to $t_{0}$, we can find $x_{1}',x_{2}'\in\mathcal{X}_{t'}$ such that
\[
d_{W_{1}}^{\mathcal{X}_{t_{0}}}(\delta_{x_{j}},\nu_{x_{j}';t_{0}})<\epsilon,\hfill|d_{t'}(x_{1}',x_{2}')-d_{t_{0}}(x_{1},x_{2})|<\epsilon,
\]
\[
|d_{t'}(x_{j}',q_{t'})-d_{t_{0}}(x_{j},q_{t_{0}})|<1
\]
for $j=1,2$.
\end{claim}
\begin{proof}
Let $\phi_{t'}':(\mathcal{X}_{t'},d_{t'})\hookrightarrow Z_{t'}'$, $\phi_{t_0}':(\mathcal{X}_{t_{0}},d_{t_{0}})\hookrightarrow Z_{t'}'$ be as in the Proposition \ref{GHctoflmf}, so that
$$
\lim_{I''\ni t'\searrow t_{0}}d(\phi_{t'}'(q_{t'}),\phi_{t_{0}}'(q_{t_0}))=0 ,
$$
and $\lim_{t'\searrow t_{0}}d_{W_{1}}^{\mathcal{X}_{t_{0}}}(\delta_{q_{t_{0}}},\nu_{q_{t'};t_{0}})=0$. Because $\mathcal{X}$ is future continuous, Claim 4.41 of \cite{Bam20b} gives $x_{1}',x_{2}'\in\mathcal{X}_{t'}$ such that $d_{W_{1}}^{\mathcal{X}_{t_{0}}}(\delta_{x_{j}},\nu_{x_{j}';t_{0}})<\frac{\epsilon}{2}$ and $d^{Z_{t'}'}(\phi_{t'}'(x_{j}'),\phi_{t_{0}}'(x_{j}))<\frac{\epsilon}{2}$ if $t'\in I''\cap(t_{0},\infty)$ is sufficiently close to $t_{0}$. Then 
\begin{align*}
|d_{t'}(x_{1}',x_{2}') & -d_{t_{0}}(x_{1},x_{2})|=|d^{Z_{t'}'}(\phi_{t'}'(x_{1}'),\phi_{t'}'(x_{2}'))-d^{Z_{t'}'}(\phi_{t_{0}}'(x_{1}),\phi_{t_{0}}'(x_{2}))|\\
\leq & d^{Z_{t'}'}(\phi_{t'}'(x_{1}'),\phi_{t_{0}}'(x_{1}))+d^{Z_{t'}'}(\phi_{t'}'(x_{2}'),\phi_{t_{0}}'(x_{2})) \leq \epsilon,
\end{align*}
and the estimate for $|d_{t'}(x_{j}',q_{t'})-d_{t_{0}}(x_{j},q_{t_{0}})|$ is similar.
\end{proof}

Assuming $t'\in I''$, we also have $\psi_{i,t'}(x_{j}')\xrightarrow[i\to\infty]{}  x_{j}'$ with respect to the Gromov-Hausdorff convergence, and in particular $\psi_{i,t'}(x_{j}')\xrightarrow[i\to\infty]{\mathfrak{C}}x_{j}'$. Thus, we have 
\[
d_{t'}(x_{1}',x_{2}')=\lim_{i\to\infty}d_{g_{i,t'}}(\psi_{i,t'}(x_{1}'),\psi_{i,t'}(x_{2}'))  .
\]
Now for any $t\in I''\cap(-\infty,t_{0})$, by triangle inequality we have
\begin{equation}\label{dbbx1x2}
\begin{split}
&d_{W_{1}}^{\mathcal{X}_{t}}(\nu_{x_{1};t},\nu_{x_{2};t})  \\
\geq & d_{W_{1}}^{\mathcal{X}_{t}}(\nu_{x_{1}';t},\nu_{x_{2}';t})-d_{W_{1}}^{\mathcal{X}_{t}}(\nu_{x_{1}';t},\nu_{x_{1};t})-d_{W_{1}}^{\mathcal{X}_{t}}(\nu_{x_{2}';t},\nu_{x_{2};t})\\
\geq & \limsup_{i\to\infty}\big\{d_{W_{1}}^{g_{i,t}}(\nu_{\psi_{i,t'}(x_{1}'),t';t}^{i},\nu_{\psi_{i,t'}(x_{2}'),t';t}^{i})-d_{W_{1}}^{Z_{t}}((\varphi_{t}^{\infty})_{\ast}\nu_{x_{1}';t},(\varphi_{t}^{i})_{\ast}\nu^i_{\psi_{i,t'}(x_{1}'),t';t})\\
& -d_{W_{1}}^{Z_{t}}((\varphi_{t}^{\infty})_{\ast}\nu_{x_{2}';t},(\varphi_{t}^{i})_{\ast}\nu^i_{\psi_{i,t'}(x_{2}'),t';t})\big\}-d_{W_{1}}^{\mathcal{X}_{t_{0}}}(\nu_{x_{1}';t_{0}},\delta_{x_{1}})-d_{W_{1}}^{\mathcal{X}_{t_{0}}}(\nu_{x_{2}';t_{0}},\delta_{x_{2}})\\
\geq & \limsup_{i\to\infty}d_{W_{1}}^{g_{i,t}}(\nu_{\psi_{i,t'}(x_{1}'),t';t}^{i},\nu_{\psi_{i,t'}(x_{2}'),t';t}^{i})-2\epsilon.    
\end{split}
\end{equation}
Next, observe (by the Claim \ref{perofxj})
\begin{equation}\label{dbbxjpp}
\lim_{i\to\infty}d_{g_{i,t'}}(\psi_{i,t'}(x_{j}'),p_{i})=d_{t'}(x_{j}',q_{t'})\leq1+2\epsilon+d_{t_{0}}(x_{j},p_{t_{0}}),    
\end{equation}
hence by Proposition \ref{gddlb} and Proposition \ref{gddub2}, we have
\[
\limsup_{i\to\infty}|d_{g_{i,t'}}(\psi_{i,t'}(x_{1}'),\psi_{i,t'}(x_{2}'))-d_{g_{i,t}}(\psi_{i,t'}(x_{1}'),\psi_{i,t'}(x_{2}'))|\leq C(t_0,x_{1},x_{2})\sqrt{t'-t},
\]
and by Lemma \ref{W1balongwl} we have
\[
\limsup_{i\to\infty}d_{W_{1}}^{g_{i,t}}(\nu_{\psi_{i,t'}(x_{j}'),t';t},\delta_{\psi_{i,t'}(x_{j}')})\leq C(t_0,x_{j})\sqrt{t'-t} ,
\]
both for all $t'\in I''\cap(t_{0},\infty)$ and $t\in I''\cap(-\infty,t_{0})$ both sufficiently close to $t_{0}$. Combining estimates and Claim \ref{perofxj}, we get
\[
d_{W_{1}}^{\mathcal{X}_{t}}(\nu_{x_{1};t},\nu_{x_{2};t})-d_{t_0}(x_{1},x_{2})\geq -3\epsilon - C(t_0,x_{1},x_{2})\sqrt{t'-t}
\]
for all $t'\in I''\cap(t_{0},\infty)$ and $t\in I''\cap(-\infty,t_{0})$ both sufficiently close to $t_{0}$. Taking $t'\searrow t_{0}$ along $t'\in I''$ then yields
\[
d_{W_{1}}^{\mathcal{X}_{t}}(\nu_{x_{1};t},\nu_{x_{2};t})-d_{t_0}(x_{1},x_{2})\geq -3\epsilon - C(t_0,x_{1},x_{2})\sqrt{t_{0}-t},
\]
so that
\[
\lim_{t\nearrow t_{0}}d_{W_{1}}^{\mathcal{X}_{t}}(\nu_{x_{1};t},\nu_{x_{2};t})-d_{t_0}(x_{1},x_{2})\geq -3\epsilon.
\]
Since $\epsilon>0$ was arbitrary, the claim follows. 
\end{proof}

Now we can finish the proof of the main results.

\begin{proof}[\bf{Proof of Theorem \ref{mainforKRF1}}]
Theorem \ref{mainforKRF1} follows immediately from Proposition \ref{GHctoflmf} and Proposition \ref{GW1conofX}.
\end{proof}

\begin{remark}
In Proposition \ref{GHctoflmf} and Proposition \ref{GW1conofX}, we only consider the time $t\in (-\infty,0)$, but the results also hold for the time $t=0$. This can be seen in the following way. When we form the Type I blow-up limits of our original K\"ahler-Ricci flow, we take any sequence of times $t_i\nearrow 1$, and choose the base point as the Ricci vertex $p_{t_i}$ associated to the given $\theta_Y$. Here we can perturb the time $t_i$ to $\hat t_i= t_i+(1-t_i)/2$, but still with the base point as the Ricci vertex $p_{t_i}$. After this perturbation, the time $t=0$ in this section is the time $t=-1$ in the new sequence and limit, and Theorem 6.40 of \cite{Bam20b} would allow us to change the base point of the $\bF$-convergence from $(p_i,0)$ to $(p_i, -1)$. Now our Proposition \ref{GHctoflmf} and Proposition \ref{GW1conofX} apply for the new sequence of flows and limit, which implies the continuity of the original limit at time $t=0$.  
\end{remark}


\bigskip
\section{Structure of noncollapsed Ricci flows with locally bounded scalar curvature}\label{prfofthmmain}

In this section, we consider the more general set-up of Ricci flows with locally bounded scalar curvature, which is already appeared in \cite{JST23a}. First, we recall the definition of based barrier of the scalar curvature. Let $(M, g(t))_{t\in I}$ be a smooth Ricci flow on a compact $n$-dimensional manifold with the interval $I\subset \mathbb{R}$.

\begin{definition}[Based barrier of the scalar curvature] \label{bbofsc}
Let $v:M\times I\to \mathbb{R}$ be a $C^1$-function and $C<\infty$ be a constant.  We call $v$ a \textbf{$C$-barrier of $R_g$} if the following hold on $M\times I$:
\begin{enumerate}
\item $v\geq 1$;
\item $|\partial_t \ln v| + |\nabla \ln v|^2 \leq C$;
\item $R_{g} \leq Cv$.
\end{enumerate}
Let $(x_0, t_0)\in M\times I$ and $B<\infty$. Then we say $v$ is \textbf{$B$-based at $(x_0, t_0)$} if
$$
v(x_0, t_0)\leq B.
$$
\end{definition}
\begin{remark}
For finite time solution of K\"ahler-Ricci flow on projective manifolds, such based barrier functions arise naturally from the normalized Ricci potential.
Suppose $I=[-T, 0]$ for some $T\in (0, \infty]$. Let $\lambda>0$ be a rescaling factor. Denote by $\tilde g_t= \lambda^{-2} g_{\lambda^2 t}$ and $\tilde v(t)=v(\lambda^2 t)$, where $t\in [-\lambda^{-2}T, 0]$. If $v$ is $C$-barrier $B$-based at $(x_0, t_0)$ of $R_g$, then $\tilde v$ is $\lambda^{2}C$-barrier $B$-based at $(x_0, \lambda^{-2}t_0)$ of $R_{\tilde g}$.
\end{remark}
Now, let $( M_i, (g_{i,t})_{t\in [-T_i, 0]} , (p_i, 0) )$ be a sequence of pointed Ricci flows on compact manifolds of dimension $n$ and $T_{\infty}:= \lim_{i\to\infty} T_i$. By the results of \cite{Bam20b}, passing to a subsequence, we can obtain $\bF$-convergence (see Definition \ref{mfpairs} and Definition \ref{corrspandFcon}) on compact time-intervals
\begin{equation}\label{FcofRF'1}
(M_i, (g_{i,t})_{t\in [-T_i , 0]}, (\nu_{p_i,0; t})_{t\in [-T_i , 0]}) \xrightarrow[i\to\infty]{\bF,\CCC}  (\cX, (\nu_{p_\infty; t})_{t\in (-T_{\infty} , 0]}),
\end{equation}
within some correspondence $\CCC$, where $\cX$ is a future continuous and $H_{n}$-concentrated metric flow of full support over $(-T_{\infty} , 0]$.
For the non-collapsing assumption, we assume that, for some uniform $Y_0<\infty$, we have
\begin{equation}\label{lbofNashatbasep}
\nu[g_{i, -T_i}, 2T_i]\geq -Y_0.
\end{equation}
According to \cite{Bam20c}, we can decompose $\cX$ into its regular and singular part 
\begin{equation}\label{rsd'1}
\cX= \cR \sqcup \cS,
\end{equation}
where $\cR$ is a dense open subset of $\cX$. The singular set $\cS$ has parabolic Minkowski dimension $\leq n-2$. Also, $\cR$ carries the structure of a Ricci flow spacetime $(\cR, \mathfrak{t}, \partial_{\mathfrak{t}}, g)$. For any $t\in (-T_{\infty}, 0)$, $\cR_t=\cX_t\cap\cR$, we have that $(\cX_t, d_t)$ is the metric completion of $(\cR_t, g_t)$.
For the local scalar curvature bound assumption, we suppose there exist a sequence of constants $C_i<\infty$ and a sequence of functions $v_i$, such that $v_i$ is a $C_i$-barrier of $R_{g_i}$ and $Y_0$-based at $(p_i, 0)$ for each $i$. 

We have the following improvement on the convergence.
\begin{theorem}[Theorem 7.3 in \cite{JST23a}]\label{FtoGHfromJST}
Suppose we have $C_i\leq Y_0$ for all $i$. Then for every $t\in (-T_\infty, 0)$ where (\ref{FcofRF'1}) is time-wise, passing to a subsequence, we have that $(M_{i}, d_{g_{i,t}}, p_i)$ converge to $(\cX_{t}, d_{t}, q_{t})$ in the Gromov-Hausdorff topology for some $q_{t}\in \cX_{t}$.
\end{theorem}
The main result of this section is the characterization of the time-slices of the limiting metric flow $\cX$.

\begin{theorem}\label{mainforRF1}
Suppose we have $C_i\leq Y_0$ for all $i$. Then for every $t\in (-T_\infty, 0)$, the following statements hold.
\begin{enumerate}
%
\item $(\cX_t, d_t,\cR_t, g_t)$ is a singular space of dimension $n$, in the sense of Definition \ref{sp}.
\item We have the Minkowski dimension estimate 
$$\dim_{\mathcal{M}}\cS_{t}\leq n-4.$$
\end{enumerate}
\end{theorem}
We first remark that, when the sequence of Ricci flows have globally uniformly bounded scalar curvature, Theorem \ref{mainforRF1} is proved by Bamler (cf. \cite{Bam18}). Here we extend Bamler's results to Ricci flows with locally bounded scalar curvature.

We also remark that, if we assume $\lim_{i\to\infty}C_i= 0$ and $\lim_{i\to\infty} \inf_{M_i\times \left\{-T_i\right\} } R_{g_i}\geq 0$, then $\cX$ is a static limit, hence the conclusion of \cite[Theorem 2.16]{Bam20c} holds for the limiting metric flow $\cX$. Hence in such case, Theorem \ref{mainforRF1} is already proved by Bamler. In conclusion, we have the following corollary.
\begin{corollary}\label{CorforRF1}
Suppose $\lim_{i\to\infty}C_i= 0$ and $\lim_{i\to\infty} \inf_{M_i\times \left\{-T_i\right\} } R_{g_i}\geq 0$. Then the conclusion of \cite[Theorem 2.16]{Bam20c} holds for $\cX$. Moreover, for every $t\in (-T_\infty, 0)$, passing to a subsequence, we have $(M_{i}, d_{g_{i,t}}, p_i)$ converge to $(\cX_{t}, d_{t}, q_{\infty})$ in the Gromov-Hausdorff topology for some $q_{\infty}\in \cX_{t}$.
\end{corollary}
%


%
Throughout this section, unless otherwise stated, all the constants will depend on $n , Y_0$. We will omit this dependence in this section for convenience.


\medskip
\subsection{Preliminary results} \label{preofbar}

In this section, we recall some results established in \cite{JST23a}. First, we have
\begin{lemma}[Lemma 7.4 in \cite{JST23a}] \label{viprop2} 
For any $A, D<\infty$, there exists a constant $C<\infty$ depending on $A, D$, such that the following statements hold on the Ricci flow $(M_i, (g_{i,t})_{t\in [-T_i , 0]})$.
For any $(x_0, t_0)\in M_i\times [-T_i, 0]$, if
$$  
v_i(x_0, t_0)\leq A,
$$
then we have
\begin{enumerate}
\item $v_i(x_0, t)\leq C$ for all $t\in [ t_0-D , t_0+D ]\cap [-T_i, 0]$;
\item $v_i(x, t_0)\leq C$ for all $x\in B_{g_i}( x_0, t_0, D )$.
\end{enumerate}
\end{lemma}
The next lemma states that the boundedness of the barrier function propagated in the $P^*$-parabolic neighborhoods.
\begin{proposition}[Proposition 7.5 in \cite{JST23a}] \label{viprop3} 
For any $\eta\in (0, 1)$, $A, D, T^{\pm}<\infty$, there exists a constant $C<\infty$ depends on $\eta, A, D, T^{\pm}$, such that the following statements hold on the flow $(M_i, (g_{i,t})_{t\in [-T_i , 0]})$.
Suppose $(x_0, t_0)\in M_i\times [-T_i+ T^- +10\eta , 0]$ satisfies
$$  
v_i(x_0, t_0)\leq A .
$$
Then for any $(y_0, s_0)\in P^*(x_0, t_0; D, -T^-, T^+)$, we have
$$  
v_i(y_0, s_0)\leq C .
$$
\end{proposition}

Next, we have the following short time distance distortion estimate.
\begin{proposition}[Proposition 7.6 in \cite{JST23a}]\label{stdde}
For any $\eta\in (0, 1)$, $A, D<\infty$, there exist constants $\delta\in (0, \eta)$, $C<\infty$, both depending on $\eta, A, D$, such that the following statements hold on the Ricci flow $(M_i, (g_{i,t})_{t\in [-T_i , 0]})$.
Suppose $(x_0, t_0)\in M_i\times [-T_i+10\eta, 0]$ satisfies that
$$  
v_i(x_0, t_0)\leq A,
$$
then for any $y_0\in B_{g_i}( x_0, t_0, D )$, we have
$$  
d_{g_{i,t}}( y_0 , x_0 ) \leq C,
$$
for all $t\in [ t_0-\delta , \min\left\{t_0+\delta, 0\right\} ]$.
\end{proposition}
Finally, we have the following heat kernel lower bound estimate.
\begin{proposition}[Proposition 7.7 in \cite{JST23a}]\label{hklb}
For any $\eta\in (0, 1)$, $A, D<\infty$, there exists constant $C<\infty$ depends on $\eta, A, D$, such that the following statements hold on the Ricci flow $(M_i, (g_{i,t})_{t\in [-T_i , 0]})$.
If $(x_0, t_0)\in M_i\times [-T_i+10\eta, 0]$ satisfies
$$  
v_i(x_0, t_0)\leq A,
$$
then for any $s_0\in [ \max\left\{t_0-\eta^{-1}, -T_i+\eta\right\}  , t_0-\eta ]$ and $y_0\in B_{g_i}( x_0, s_0, D )$, we have
$$  
K_i(x_0,t_0;y_0, s_0) \geq C^{-1},
$$
where $K_i(x,t; y,s)$, $s<t$ denotes the heat kernel along the flow $g_{i,t}$.
\end{proposition}
%


\medskip
\subsection{$\bF$-limit and local smooth convergence} \label{setupofFlimit}

Let us start with (\ref{FcofRF'1}), that is, we have the $\bF$-convergence on compact time-intervals
\begin{equation}\label{FcorKRF'2}
(M_i, (g_{i,t})_{t\in [-T_i , 0]}, (\nu_{p_i,0; t})_{t\in [-T_i , 0]}) \xrightarrow[i\to\infty]{\bF,\CCC}  (\cX, (\nu_{p_\infty; t})_{t\in (-T_{\infty} , 0]}),
\end{equation}
within some correspondence $\CCC$, where $\cX$ is a future continuous and $H_{n}$-concentrated metric flow of full support over $(-T_\infty , 0]$. We can decompose $\cX$ into it's regular and singular part 
\begin{equation}\label{rsd}
\cX= \cR \sqcup \cS,
\end{equation}
where $\cR$ is dense open subset of $\cX$. Also, $\cR$ carries the structure of a Ricci flow spacetime $(\cR, \mathfrak{t}, \partial_{\mathfrak{t}}, g)$. For any $t\in (-T_\infty , 0]$, $\cR_t=\cX_t\cap\cR$, we have $(\cX_t, d_t)$ is the metric completion of $(\cR_t, g_t)$.

We denote the conjugate heat kernels restricted to $\cR$ as follows:
$$ d\nu_{x;s}=:K(x; \cdot)dg_s,~~~x\in \cX,~~s\in (-T_\infty, \mathfrak{t}(x)), $$
where
$$ K:\left\{ (x;y)\in \cX\times\cR~~~:~~~ \mathfrak{t}(x) > \mathfrak{t}(y) \right\} \to \mathbb{R}_+, $$
is a continuous function. For any $x\in\cX$, the function $K(x;\cdot):\cR_{<\mathfrak{t}(x)}\to \mathbb{R}_+$ is a smooth function and satisfies the conjugate heat equation $\Box^*K(x;\cdot)=0$; for any $y\in \cR$, the function $K(\cdot;y):\cR_{>\mathfrak{t}(y)}\to \mathbb{R}_+$ is a smooth function and satisfies the heat equation $\Box K(\cdot;y)=0$.

We can find an increasing sequence of open subsets $U_1\subset U_2\subset \cdots \subset \cR $ with $\cup_{i=1}^{\infty}U_i=\cR$, open subset $V_i\subset M_i\times [-T_i, 0]$ and time-preserving diffeomorphisms $\psi_i:U_i\to V_i$ such that on $\cR$, we have
\begin{equation}\label{lsc'1}
\psi_i^*g_i \xrightarrow[i\to\infty]{C^{\infty}_{loc}} g ,~~~~ \psi_i^*\partial^i_{\mathfrak{t}} \xrightarrow[i\to\infty]{C^{\infty}_{loc}} \partial_{\mathfrak{t}},  ~~~~ K_i(x_i, t_i; \cdot)\circ \psi_i\xrightarrow[i\to\infty]{C^{\infty}_{loc}} K(x_{\infty}; \cdot)
\end{equation}
for any sequence $(x_i, t_i)\xrightarrow[i\to\infty]{\CCC}x_\infty \in \cX$ (see Definition \ref{cwithincor}). On 
$$
\left\{ (x;y)\in \cR\times\cR~~~:~~~ \mathfrak{t}(x) > \mathfrak{t}(y) \right\},
$$
we have the convergence of the heat kernels
\begin{equation}\label{lsc'2}
K_i\circ (\psi_i, \psi_i)\xrightarrow[i\to\infty]{C^{\infty}_{loc}} K.
\end{equation}
%


\bigskip
\subsection{Compactness of points under uniform geometry}
In this subsection, we prove compactness of points with uniform bounded geometry under the $\mathbb{F}$-convergence (\ref{FcorKRF'2}). For the definition of $P^\circ$-parabolic neighborhood, see \cite[Section 9.5]{Bam20b}.
\begin{proposition}\label{copuug'1}
For any $t_{\infty}\in (-T_\infty, 0)$, $D<\infty$, there exists $\Bar{\sigma}=\Bar{\sigma}( t_\infty, D)>0$, such that for any $\sigma\in (0, \Bar{\sigma}]$, the following statement holds.
Assume $(x_i, t_i)\in M_i\times (-T_i, 0)$ satisfy  
\begin{enumerate}
    \item $\lim_{i\to\infty}t_i=t_\infty$;
    \item $d_{g_{i, t_i}}(p_i, x_i)\leq D$;
    \item $\sigma\leq r_{\Rm}(x_i, t_i) \leq \sigma^{-1},$
\end{enumerate}
for all $i\in \mathbb{N}^+$. Then, after passing to a subsequence, we can find a point $x_\infty\in \mathcal{R}_{t_\infty}$, such that the following hold:
\begin{enumerate}
    \item $(x_i, t_i)\xrightarrow[i\to\infty]{\CCC}x_\infty$;
    \item $\sigma\leq \tilde r_{\Rm}(x_\infty) \leq \sigma^{-1}$;
    \item There exists a compact subset $K\subset\subset \cR$, such that whenever $i$ is large enough, we have $(x_i, t_i)\in V_i$ and $\psi_i^{-1}(x_i, t_i)\in K \subset\subset U_i$;
    \item $P^{\circ}(x_\infty; c_0\sigma, (c_0\sigma)^2, -(c_0\sigma)^2) \subset \cR$ is unscathed with $\tilde r_\Rm\geq c_0\sigma$ everywhere, where $c_0=c_0(n)>0$ is a dimensional constant, and $P(x_i, t_i; c_0\sigma)$ converge to $P^{\circ}(x_\infty; c_0\sigma, (c_0\sigma)^2, -(c_0\sigma)^2)$ in the Cheeger-Gromov sense.
\end{enumerate}
\end{proposition}

\begin{proof}
The proof of Proposition \ref{copuug'1} consists of a series of lemmas. Throughout the proof, unless otherwise stated, all the constants will depend at most on $t_\infty, D$.
Let $\eta:=(T_\infty + t_\infty)/100>0$ (if $T_\infty=\infty$, simply take $\eta=1$). Hence passing to a subsequence, we may assume that $t_i\geq -T_i+50\eta$ and $|t_i-t_\infty|\leq \eta$ for all $i$. All the times for each $i$ we consider in this proof is in $[-T_i+ \eta , 0]$, hence we have $R_{g_i}\geq -C$ when we need the lower scalar curvature bound.

To start, since $v_i(p_i, 0)\leq C$, by Lemma \ref{viprop2}, we have 
\begin{equation}\label{vibdatpi}
v_i(p_i, t)\leq C,
\end{equation}
for all $t\in [t_i-10\eta, 0]$, hence we have $R_{g_i}(p_i, t)\leq C$ for all $t\in [t_i, 0]$. Hence by Lemma \ref{W1balongwl}, we have
$$
d^{g_{i, t_i}}_{W_1} (\nu_{p_i, 0; t_i} , \delta_{p_i} )  \leq C,
$$
combining with assumption (2), we have
$$
d^{g_{i, t_i}}_{W_1} (\nu_{p_i, 0; t_i} , \delta_{x_i} )  \leq C.
$$
Hence we can apply \cite[Theorem 6.49]{Bam20b} to obtain that, after passing to a subsequence, we can find a conjugate heat flow $(\mu^\infty_t)_{t\in (-T_\infty, t_\infty)}$ on $\cX$ with
\begin{equation}\label{varomu'1}
\lim_{t\nearrow t_\infty}\var (\mu^\infty_t)=0,
\end{equation}
such that on compact time-intervals, 
\begin{equation}\label{conxiti}
(\nu_{x_i, t_i; t})_{t\in (-T_i, t_i)} \xrightarrow[i\to\infty]{\CCC} (\mu^\infty_t)_{t\in (-T_\infty, t_\infty)}.
\end{equation}
Since $\cX$ is $H_{n}$-concentrated, from (\ref{varomu'1}), we have
\begin{equation}\label{varomu'2}
\var (\mu^\infty_t)\leq H_{n}(t_\infty-t), ~~ for ~~ all ~~ t\in (-T_\infty, t_\infty).
\end{equation}
By \cite[Lemma 6.7]{Bam20b}, passing to a subsequence if necessary, we can find a subset $E_{\infty}\subset (-T_\infty, 0)$, which is of measure zero, such that both (\ref{FcorKRF'2}) and (\ref{conxiti}) are time-wise outside of $E_{\infty}$.

Now for any $j\in \mathbb{N}^+$, we let $\ep_j\to 0^+$, $\ep_j \leq \eta$, such that
$$  s_j:=t_\infty-\ep_j\notin E_{\infty}.  $$
Then for each $j$, we let $y_j \in \cX _{s_j}$ be an $H_{n}$-center of $(\mu^\infty_t)_{t\in (-T_\infty, t_\infty)}$, and hence
\begin{equation}\label{varomu'3}
\var (\mu^\infty_{s_j}, \delta_{y_j})\leq H_{n}(t_\infty-s_j) = H_{n}\ep_j .
\end{equation}
Now we apply \cite[Theorem 6.45]{Bam20b} to each $y_j$ to find $y^i_j\in M_i$, such that
\begin{equation}\label{coyijtoyj}
(y^i_j, s_j) \xrightarrow[i\to\infty]{\CCC} y_j,
\end{equation}
and (\ref{coyijtoyj}) is time-wise outside of $E_{\infty}$ for each $j$.

Then for each $j$, we can find $\delta_j\in (0, \ep_j]$, such that
$$
s_j - \delta_j \notin E_{\infty},
$$
hence (\ref{coyijtoyj}) is time-wise at $s_j - \delta_j$. For each $j$, we let $(z^i_j, s_j-\delta_j)$ be $\ell_{n}$-center of $(y^i_j, s_j)$, that is 
\begin{equation}\label{llengthofyijtozij}
\ell_{(y^i_j, s_j)} (z^i_j, s_j-\delta_j) \leq n/2,
\end{equation}
hence we can find smooth spacetime curve $\gamma_{y^i_j}: [0, \delta_j]\to M_i\times [s_j-\delta_j , s_j]$ connecting $(y^i_j, s_j)$ to $(z^i_j, s_j-\delta_j)$ such that $\cL(\gamma_{y^i_j})\leq n\delta_j^{1/2}$. 

First, we have the following lemma.
\begin{lemma}\label{neandb'1}
There exist constant $C<\infty$, such that for each $j$, there exists $i_0(j)<\infty$, such that for all $i\geq i_0(j)$, we have:
\begin{equation}\label{neandb'2}
d_{g_{i, s_j-\delta_j}}(x_i, z^i_j)\leq C\ep_j^{1/2}.    
\end{equation}
\end{lemma}
\begin{proof}
We will apply the triangle inequality to the $d_{W_1}^{Z_{s_j-\delta_j}}$-distance along the following diagram:
\begin{equation}\label{diagramofti'1}
\delta_{x_i} \stackrel{(a)}\longleftrightarrow \nu_{x_i, t_i ; s_j-\delta_j} \stackrel{(b)}\longleftrightarrow \mu^\infty_{s_j-\delta_j} \stackrel{(c)}\longleftrightarrow \nu_{y_j ; s_j-\delta_j} \stackrel{(d)}\longleftrightarrow \nu_{y^i_j, s_j ; s_j-\delta_j} \stackrel{(e)}\longleftrightarrow \delta_{z^i_j}.    
\end{equation}
For (a) in (\ref{diagramofti'1}), from (\ref{vibdatpi}), we have $v_i(p_i, t_i)\leq C$, hence by the assumption $d_{g_{i, t_i}}(p_i, x_i)\leq D$ and Lemma \ref{viprop2}, we have $v_i(x_i, t_i)\leq C$, hence by Lemma \ref{viprop2} again, we have
\begin{equation}\label{vibdatxit}
v_i(x_i, t)\leq C ,
\end{equation}
for all $t\in [t_i-1, 0]$. Hence we have $R_{g_i}(x_i, t)\leq C$ for all $t\in [t_i-1, 0]$. Hence we can apply Lemma \ref{W1balongwl} to obtain that
\begin{equation}\label{W1boxi}
d^{g_{i, s_j-\delta_j}}_{W_1} (\nu_{x_i, t_i; s_j-\delta_j} , \delta_{x_i} )\leq C\ep_j^{1/2},
\end{equation}
if we choose $i_0(j)$ large enough such that $|t_i-t_{\infty}|<\ep_j$ for all $i\geq i_0$. Hence we can compute
\begin{equation}\label{W1bbxizi'2}
\begin{split}
d^{Z_{s_j-\delta_j}}_{W_1} ( (\varphi^i_{s_j-\delta_j})_*(\delta_{x_i}) , & (\varphi^i_{s_j-\delta_j})_*( \nu_{x_i, t_i ; s_j-\delta_j} ) )\\
&\leq d^{g_{i, s_j-\delta_j}}_{W_1} (\delta_{x_i}, \nu_{x_i, t_i ; s_j-\delta_j} ) \leq C\ep_j^{1/2}.  
\end{split}
\end{equation}
where $(Z_t, d_t)$ and $\varphi^i_t$ are from Definition \ref{corrspandFcon}.
For (b) in (\ref{diagramofti'1}), since $(\nu_{x_i, t_i; t})_{t\in (-T_i, t_i)} \xrightarrow[i\to\infty]{\CCC} (\mu^\infty_t)_{t\in (-T_\infty, t_\infty)}$ is time-wise at $s_j-\delta_j$, we have
\begin{equation}\label{W1bbximuinfty}
d^{Z_{s_j-\delta_j}}_{W_1} ( (\varphi^i_{s_j-\delta_j})_*( \nu_{x_i, t_i ; s_j-\delta_j} ) , (\varphi^\infty_{s_j-\delta_j})_*( \mu^\infty_{s_j-\delta_j} ) ) \leq \ep_j^{1/2}, 
\end{equation}
for $i\geq i_0(j)$, $i_0(j)$ large enough.

For (c) in (\ref{diagramofti'1}), since $y_j \in \cX _{s_j}$ is $H_{n}$-center of $(\mu^\infty_t)_{t\in (-T_\infty, t_\infty)}$ and $\cX$ is $H_{n}$-concentrated, we can compute that
\begin{equation}\label{W1bbmuyj}
\begin{split}
&d^{Z_{s_j-\delta_j}}_{W_1} ( (\varphi^\infty_{s_j-\delta_j})_*( \mu^\infty_{s_j-\delta_j} ) , (\varphi^\infty_{s_j-\delta_j})_*( \nu_{y_j ; s_j-\delta_j} ) ) \leq \sqrt{\var(\mu^\infty_{s_j-\delta_j} , \nu_{y_j ; s_j-\delta_j})} \\
&\leq \sqrt{\var(\mu^\infty_{s_j} , \delta_{y_j} ) + H_{n}\delta_j } \leq \sqrt{ H_{2n}(t_\infty-s_j)+H_{n}\delta_j } \leq C\ep_j^{1/2}.      
\end{split}
\end{equation}

For (d) in (\ref{diagramofti'1}), since $(y^i_j, s_j) \xrightarrow[i\to\infty]{\CCC} y_j$ is time-wise at $s_j-\delta_j$, we have
\begin{equation}\label{W1bbyjyij}
d^{Z_{s_j-\delta_j}}_{W_1} ( (\varphi^\infty_{s_j-\delta_j})_*( \nu_{y_j ; s_j-\delta_j} ) , (\varphi^i_{s_j-\delta_j})_*( \nu_{y^i_j, s_j ; s_j-\delta_j} ) ) \leq \ep_j^{1/2}, 
\end{equation}
for $i\geq i_0(j)$, $i_0(j)$ large enough.

For (e) in (\ref{diagramofti'1}), we can apply Lemma \ref{wdn} to obtain that
\begin{equation}\label{W1bbzijyij}
\begin{split}
&d^{Z_{s_j-\delta_j}}_{W_1} ( (\varphi^i_{s_j-\delta_j})_*( \nu_{y^i_j, s_j ; s_j-\delta_j} ) ,  (\varphi^i_{s_j-\delta_j})_*(\delta_{z^i_j}) )\\
& \leq \left(1+ \frac{\cL(\gamma_{y^i_j})}{2\delta_j^{1/2}} -\cN^*_{s_j-\delta_j}(y^i_j, s_j)  \right) ^{1/2} \delta_j^{1/2} \leq C\delta_j^{1/2} . 
\end{split}
\end{equation}
Now, we can combine (\ref{W1bbxizi'2}), (\ref{W1bbximuinfty}), (\ref{W1bbmuyj}), (\ref{W1bbyjyij}), (\ref{W1bbzijyij}) and the triangle inequality along the order (a), (b), (c), (d), (e) in (\ref{diagramofti'1}) to obtain that
\begin{equation}\label{W1bbziyij}
d^{Z_{s_j-\delta_j}}_{W_1} ( (\varphi^i_{s_j-\delta_j})_*(\delta_{x_i}) , (\varphi^i_{s_j-\delta_j})_*( \delta_{z^i_j} ) )\leq  C\ep_j^{1/2}.
\end{equation}
Since $\varphi^i_{s_j-\delta_j}: (M_i, d_{g_{i, s_j-\delta_j}})\to (Z_{s_j-\delta_j}, d^Z_{s_j-\delta_j})$ is isometric embedding, we obtain from (\ref{W1bbziyij}) that
$$d_{g_{i, s_j-\delta_j}}(x_i, z^i_j)\leq C\ep_j^{1/2}.$$
This completes the proof of Lemma \ref{neandb'1}.
\end{proof}
Next, we have the following lemma.
\begin{lemma}\label{rrmbofyij}
There exists $j_0<\infty$, such that for all $j\geq j_0$, there exists $i_0(j)<\infty$, such that for all $i\geq i_0(j)$, we have
\begin{enumerate}
    \item $d_{g_{i, s_j-\delta_j}}(y^i_j, x_i)\leq C\ep_j^{1/2}$;
    \item $d_{g_{i, s_j-\delta_j}}(y^i_j, p_i)\leq C$;
    \item $r_{\Rm}(y^i_j, s_j)\geq \sigma/2$.
\end{enumerate}
\end{lemma}
\begin{proof}
First, recall that $\lim_{i\to\infty}t_i=t_\infty=s_j+j^{-1}$. From (\ref{vibdatxit}) we have $v_i(x_i, s_j-\delta_j) \leq C$, hence by Lemma \ref{viprop2} and Lemma \ref{neandb'1}, we have $v_i(z^i_j, s_j-\delta_j) \leq C$. Hence we can apply Proposition \ref{viprop3} to obtain that $v_i(y^i_j, s_j)\leq C$, hence by Lemma \ref{viprop2}, we have $v_i(y^i_j, t)\leq C$ for all $t\in [t_i-1, 0]$. Hence we have
$$
R_{g_i}(y^i_j, t)\leq C ,
$$
for all $t\in [t_i-1, 0]$. Hence we can apply Lemma \ref{W1balongwl} to obtain that
\begin{equation}\label{W1boyij}
d^{g_{i, s_j-\delta_j}}_{W_1} (\nu_{y^i_j, s_j; s_j-\delta_j} , \delta_{y^i_j} )\leq C\delta_j^{1/2} .
\end{equation}
Combining this with (\ref{W1bbzijyij}), we can compute
$$
d_{g_{i, s_j-\delta_j}} ( y^i_j , z^i_j )\leq d^{g_{i, s_j-\delta_j}}_{W_1} ( \delta_{y^i_j} , \nu_{y^i_j, s_j ; s_j-\delta_j} ) + d^{g_{i, s_j-\delta_j}}_{W_1} (\nu_{y^i_j, s_j ; s_j-\delta_j} , \delta_{z^i_j} ) \leq C \delta_j^{1/2} .
$$
Combining this with Lemma \ref{neandb'1}, we have $d_{g_{i, s_j-\delta_j}}(y^i_j, x_i)\leq C\ep_j^{1/2}$, this proves item (1).
Next, since $d_{g_{i, t_i}}(p_i, x_i)\leq D$, by applying Proposition \ref{stdde}, we have 
$$
d_{g_{i, s_j-\delta_j}}(p_i, x_i)\leq C,
$$
provided $j\geq j_0$, $j_0$ large, and $i\geq i_0(j)$. This proves item (2).
Finally, from $r_{\Rm}(x_i, t_i)\geq \sigma$, from Lemma \ref{deofrrm}, we have $r_{\Rm}(x_i, s_j-\delta_j)\geq 9\sigma/10$ provided $j\geq j_0$, $j_0$ large, and $i\geq i_0(j)$. Hence by Lemma \ref{deofrrm} and item (2) of this lemma, we have $r_{\Rm}(y^i_j, s_j-\delta_j)\geq 8\sigma/10$ provided $j\geq j_0$ and $i\geq i_0(j)$. Hence by Lemma \ref{deofrrm} again, we have $r_{\Rm}(y^i_j, s_j)\geq \sigma/2$ provided $j\geq j_0$. This proves item (3). 

This completes the proof of lemma \ref{rrmbofyij}.
\end{proof}

Next, we let $\Bar{\sigma}\leq \eta$. Then we have
\begin{lemma}\label{hkbayij'1}
There exists $j_0<\infty$, such that for all $j\geq j_0$, there exists $i_0(j)<\infty$, such that for all $i\geq i_0(j)$, for all $y\in B(y^i_j, s_j, \sigma/2)$, for all $t\in [ s_j-(\sigma/2)^2 , s_j+(\sigma/2)^2 ]$, we have
\begin{equation}\label{hkbayij'2}
K(p_i, 0; y, t)\geq C^{-1}.
\end{equation}
\end{lemma}
\begin{proof}
From item (2) of Lemma \ref{rrmbofyij} and (\ref{vibdatpi}), we apply Proposition \ref{stdde} to get
\begin{equation}\label{dbyijpi}
d_{g_{i, s_j}}(y^i_j, p_i)\leq C.
\end{equation}
provided $j\geq j_0$ and $i\geq i_0(j)$. Hence for any $y\in B(y^i_j, s_j, \sigma/2)$, we have $d_{g_{i, s_j}}(y, p_i)\leq C$. By Proposition \ref{stdde} again, if we choose $\Bar{\sigma}>0$ small enough, then we have
$$ d_{g_{i, t}}(y, p_i)\leq C. $$
for all $t\in [ s_j-(\sigma/2)^2 , s_j+(\sigma/2)^2 ]$. Then (\ref{hkbayij'2}) follows immediately from Proposition \ref{hklb}. This completes the proof.
\end{proof}

From Lemma \ref{rrmbofyij}, we have for $j\geq j_0$, $i\geq i_0(j)$, $r_{\Rm}(y^i_j, s_j) \geq \sigma/2$, and since $(y^i_j, s_j) \xrightarrow[i\to\infty]{\CCC} y_j$, we can apply \cite[Theorem 9.58]{Bam20b} to each $j\geq j_0$, to obtain that, there is a maximal $T^*_j \in (0,  (\sigma/2)^2]$ such that the open parabolic neighborhood (see \cite[Section 9.5]{Bam20b}) $P^\circ(y_j; \sigma/2, -(\sigma/2)^2, T^*_j) \subset \cR$ is unscathed. The point is to obtain uniform lower bound of $T^*_j$. We have the following lemma.

\begin{lemma}\label{ulboTj*}
There exists $j_0<\infty$, such that for all $j\geq j_0$, we have
$$T^*_j = (\sigma/2)^2.$$
In particular, we have $P(y^i_j, s_j; \sigma/2)$ converge to $P^\circ(y_j; \sigma/2, -(\sigma/2)^2, (\sigma/2)^2)\subset \cR$ in the Cheeger-Gromov sense as $i\to\infty$.
\end{lemma}
\begin{proof}
Choose $j_0$ large such that Lemma \ref{hkbayij'1} holds. Assume that for some $j\geq j_0$, we have $T^*_j < (\sigma/2)^2$. Then according to \cite[Theorem 9.58]{Bam20b}, we have 
\begin{equation}\label{hkbayij'3}
\lim_{t\nearrow s_j + T^*_j} K(p_\infty; y_j(t)) = 0 ,
\end{equation}
where $y_j(t)\in \cR(t)$ denotes the point survive from $y_j\in \cR(s_j)$ to time $t$.

On the other hand, we have the smooth convergence of the parabolic neighborhoods $P(y^i_j, s_j; \sigma/2, -(\sigma/2)^2, T^*_j)$ to the unscathed open parabolic neighborhood $P^\circ(y_j; \sigma/2, -(\sigma/2)^2, T^*_j) \subset \cR$, hence for any $t \in [s_j , s_j + T^*_j)$, by \cite[Theorem 9.31, (e)]{Bam20b} we have
\begin{equation}\label{cohkatyij}
K(p_\infty; y_j(t)) = \lim_{i\to\infty} K(p_i, 0 ; y^i_j, t).
\end{equation}
Hence combining (\ref{cohkatyij}) and Lemma \ref{hkbayij'1}, we obtain
$$ K(p_\infty; y_j(t)) = \lim_{i\to\infty} K(p_i, 0 ; y^i_j, t)\geq C^{-1}, $$
for all $t \in [s_j , s_j + T^*_j)$, this however contradicts (\ref{hkbayij'3}). This contradiction means that $T^*_j = (\sigma/2)^2$, hence completes the proof.
\end{proof}
Next, we have the following distance distortion estimate, which holds trivially true on ordinary Ricci flows.

\begin{lemma}\label{ddeunrmb}
Assume $r\in (0,1)$, $x_0\in \cR_{t_0}$ satisfy that $P^\circ(x_0; r, -r^2, r^2) \subset \cR$ is unscathed and $ \tilde r_{\Rm}(x_0)\geq r$. Then for any $x\in B(x_0, \frac{r}{4})$ and $t\in [t_0-r^2, t_0+r^2]$, we have
$$e^{-r^{-2}|t-t_0|}d_{t_0}(x_0, x)\leq d_t(x_0(t), x(t))\leq e^{r^{-2}|t-t_0|}d_{t_0}(x_0, x).$$
\end{lemma}
\begin{proof}
Since $P^\circ(x_0; r, -r^2, r^2) \subset \cR$ is unscathed, for any $x\in B(x_0, \frac{r}{4})$, we can find a smooth curve $\gamma:[0,1]\to B(x_0, r)$ connecting $x_0$ to $x$, such that $\ell_{g_{t_0}}(\gamma)=d_{g_{t_0}}(x_0, x)$. Hence for any $t\in [t_0-r^2, t_0+r^2]$, we have
$$-r^{-2}\ell_{g_t}(\gamma)\leq \frac{d}{dt} \ell_{g_t}(\gamma)=\int_{0}^{1} \frac{-\ric_{g_t}(\dot{\gamma}(s), \dot{\gamma}(s))}{|\dot{\gamma}(s)|_{g_t}}ds\leq r^{-2}\ell_{g_t}(\gamma).  $$
Integrating this we obtain 
$$e^{-r^{-2}|t-t_0|}d_{g_{t_0}}(x_0, x)\leq d_{g_t}(x_0(t), x(t))\leq e^{r^{-2}|t-t_0|}d_{g_{t_0}}(x_0, x),$$
then we recall that for any $t<0$, $(\cX_t, d_t)$ is the metric completion of $(\cR_t, g_t)$, which completes the proof. 
\end{proof}
Now we can finish the proof of Theorem \ref{copuug'1}. 

First, due to Lemma \ref{rrmbofyij} and Lemma \ref{ulboTj*} and \cite[Lemma 15.16]{Bam20c}, for $j\geq j_0$, we have $\tilde r_{\Rm}(y_j) \geq \sigma/2$ and $P^\circ(y_j; \sigma/2, -(\sigma/2)^2, (\sigma/2)^2) \subset \cR$ is unscathed. Hence we can apply \cite[Lemma 15.16, (e),(f)]{Bam20c} to find a dimensional constant $c_0=c_0(n)>0$, such that
\begin{equation}\label{rrmeonPnbhdofyj}
\tilde r_{\Rm}(y) \geq \sigma/4, ~~ for~~ any~~ y\in P^\circ(y_j; 2c_0\sigma, -(2c_0\sigma)^2, (2c_0\sigma)^2).
\end{equation}
Indeed, given $y\in P^\circ(y_j; 2c_0\sigma, -(2c_0\sigma)^2,(2c_0\sigma)^2)$, we have $|\tilde r_{\Rm}(y(s_j))-\tilde r_{\Rm}(y_j)|\leq 2c_0\sigma$, hence $\tilde r_{\Rm}(y(s_j))\geq \sigma/3$ if $c_0\leq 1/1000$; then from $|\tilde r_{\Rm}^2(y)-\tilde r_{\Rm}^2(y(s_j))|\leq C_0(n)(2c_0\sigma)^2$, hence $\tilde r_{\Rm}(y) \geq \sigma/4$ is $c_0(n)$ is small enough.

Next, we claim that if we choose $c_0(n)$ small enough, then for any $y\in P^\circ(y_j; 2c_0\sigma, -(2c_0\sigma)^2,(2c_0\sigma)^2)$, we have
\begin{equation}\label{croPyangPyj}
P^\circ(y; 2c_0\sigma, -(2c_0\sigma)^2, (2c_0\sigma)^2)\subset P^\circ(y_j; \sigma/2, -(\sigma/2)^2, (\sigma/2)^2).
\end{equation}
Indeed, let $\tilde t:=\mathfrak{t}(y)$, $x\in P^\circ(y; 2c_0\sigma, -(2c_0\sigma)^2,(2c_0\sigma)^2)$, then we have $d_{\tilde t}(x(\tilde t), y)\leq 2c_0\sigma$. But from Lemma \ref{ddeunrmb}, we have $d_{\tilde t}(y, y_j(\tilde t))\leq ed_{s_j}(y(s_j), y_j)\leq 6c_0\sigma$, hence we have $d_{\tilde t}(x(\tilde t), y_j(\tilde t))\leq 8c_0\sigma$. But we have $\tilde r_{\Rm}(y_j(\tilde t)) \geq \sigma/4\geq 4\cdot8c_0\sigma$, hence again by Lemma \ref{ddeunrmb}, we have $d_{s_j}(x(s_j), y_j)\leq ed_{\tilde t}(x(\tilde t), y_j(\tilde t))\leq 20c_0\sigma\leq \sigma/10$, hence $x\in P^\circ(y_j; \sigma/2, -(\sigma/2)^2, (\sigma/2)^2)$. This proves (\ref{croPyangPyj}). Similar proof of (\ref{rrmeonPnbhdofyj}) shows that we can further require that
\begin{equation}\label{rrmeonPnbhdofy}
\tilde r_{\Rm}(x) \geq c_0\sigma, ~~ for~~ any~~ x\in P^\circ(y; 2c_0\sigma, -(2c_0\sigma)^2, (2c_0\sigma)^2).
\end{equation}
Next, the same arguments in the proof of Lemma \ref{ulboTj*} show that
\begin{equation}\label{cohkaty'1}
K(p_\infty; y) \geq C^{-1} , ~~ \forall ~~ y\in P^\circ(y_j; \sigma/5, -(\sigma/5)^2, (\sigma/5)^2)
\end{equation}
but from \cite[Lemma 15.9, (a)]{Bam20c}, we have
\begin{equation}\label{cohkaty'2}
K(p_\infty; y) \leq C \exp \left( -\frac{  (d_{W_1}^{\cX_{\mathfrak{t}(y)}} ( \nu_{p_\infty ; \mathfrak{t}(y)} , \delta_y ))^2 }{C}\right) ,
\end{equation}
hence we can combine (\ref{cohkaty'1}) and (\ref{cohkaty'2}) to obtain that 
$$ d_{W_1}^{\cX_{\mathfrak{t}(y)}} ( \nu_{p_\infty ; \mathfrak{t}(y)} , \delta_y ) \leq A , $$
for all $y\in P^\circ(y_j; \sigma/5, -(\sigma/5)^2, (\sigma/5)^2)$. Hence if we set
$$ 
W := \left\{ \tilde r_{\Rm} \geq \sigma/10 \right\} \cap P^*(p_\infty; A, t_\infty-10\eta) \cap \cR_{[ t_\infty - \sigma^2 , t_\infty + \sigma^2 ]}, 
$$
then from (\ref{rrmeonPnbhdofyj}), for $j\geq j_0$, we have $P^\circ(y_j; 2c_0\sigma, -(2c_0\sigma)^2, (2c_0\sigma)^2) \subset W$. Using \cite[Lemma 15.16, (h)]{Bam20c}, we know $ W \subset \cR $ is a relatively compact subset of $\cR$, hence for $i$ large enough, we have $ W \subset U_i $, then from Lemma \ref{ulboTj*}, we have
\begin{equation}\label{pnbhdoyijinVi}
P(y^i_j, s_j ; c_0\sigma ) \subset V_i = \psi_i(U_i) ,
\end{equation}
for $i\geq i_0(j)$, $i_0(j)$ large enough, and we have 
$$\psi_i^{-1}(P(y^i_j, s_j ; c_0\sigma ))\subset P^\circ(y_j; 2c_0\sigma, -(2c_0\sigma)^2, (2c_0\sigma)^2).$$
Recall that $ r_{\Rm}(x_i , t_i)\geq \sigma$ and $\lim_{i\to\infty}t_i=t_\infty=s_j+\ep_j$, we have $r_{\Rm}(x_i , s_j-\delta_j)\geq \sigma/2$ for $j\geq j_0$, $i\geq i_0(j)$. From item (2) of Lemma \ref{rrmbofyij}, we have $d_{g_{i, s_j-\delta_j}}(y^i_j, x_i)\leq C\ep_j^{1/2}$, hence by Lemma \ref{ddeunrmb}, we have $(x_i, t_i)\in P(y^i_j, s_j ; c_0\sigma )$ for $j\geq j_0$, $i\geq i_0(j)$. We can fix one such $j$ from now on. Hence $\psi_i^{-1}(x_i, t_i)$ is well-defined, and 
$$\psi_i^{-1}(x_i, t_i) \in P^\circ(y_j; 2c_0\sigma, -(2c_0\sigma)^2, (2c_0\sigma)^2) \subset W.$$
Since $ W \subset \cR $ is relatively compact, passing to a subsequence, we have
$$ \psi_i^{-1}(x_i, t_i)\xrightarrow[i\to\infty]{} x_\infty\in \overline{P^\circ(y_j; 2c_0\sigma, -(2c_0\sigma)^2, (2c_0\sigma)^2)} \subset K:= \overline{W} \subset \cR, $$
hence apply \cite[Theorem 9.31, (c)]{Bam20b}, we have $(x_i, t_i)\xrightarrow[i\to\infty]{\CCC}x_\infty$, this proves items (1) and (3). Item (2) follows since $ \tilde r_{\Rm}(x_\infty) = \lim_{i\to\infty} r_{\Rm}(x_i, t_i) $. Item (4) follows from Lemma \ref{ulboTj*}, (\ref{croPyangPyj}) and (\ref{rrmeonPnbhdofy}).

This completes the proof of Proposition \ref{copuug'1}.
\end{proof}
%


\medskip
\subsection{Long-time distance distortion estimates at small scales}

In this section, we prove the following long-time distance distortion estimates at small scales.

\begin{proposition}\label{Tboundonsmallr}

%
For any $\eta\in (0, 1)$, $A, D<\infty$, there exist constants $C<\infty$ and $\Bar{r}>0$, both depend on $\eta, A, D$, such that for any $0<r\leq \Bar{r}$, the following statements hold on the Ricci flow $(M_i, (g_{i,t})_{t\in [-T_i , 0]})$.
Suppose $(x_0, t_0)\in M_i\times [-T_i+10\eta, 0]$ satisfies that
$$  
v_i(x_0, t_0)\leq A,
$$
then for any $t_1\in [t_0-r^2, t_0+r^2]$ and $y_0\in B_{g_i}( x_0, t_0, Dr )$, we have
$$
d_{g_{i, t_1}} (x_0, y_0)\leq Cr.
$$
\end{proposition}
\begin{proof}
Throughout the proof, unless otherwise stated, all the constants will depend at most on $\eta, A, D$. All the times we consider in this proof is in $[-T_i+ \eta , 0]$, hence we have $R_{g_i}\geq -C$ when we need the lower scalar curvature bound.

By Lemma \ref{viprop2}, we have
\begin{equation}\label{boundvi2}
v_i(x_0, t)\leq C ,
\end{equation}
for all $t\in [t_0-\eta , 0]$. We then have
\begin{lemma}\label{stdse}
There exists a constant $0<\alpha<\frac{1}{100}$, $C_0<\infty$, $0<\tilde{r}<1$, such that the following statement holds.

Assume $0< r\leq \tilde{r}$, $t',t''\in [t_0-\eta, 0]$ with $|t'-t''|\leq \alpha r^2$, and $(x_1,t')\in M_i\times \left\{t'\right\}$ with $d_{g_{i,t'}}(x_0, x_1)\leq r$, then we have
$$d_{g_{i,t''}}(x_0, x_1)\leq C_0 r.$$
\end{lemma}
\begin{proof}
Consider the rescaled flow $(M_i, (\tilde g_{i,t})_{t\in [-r^{-2}T_i, 0]})$ with $\tilde g_{i,t}=r^{-2}g_{i, r^2t}$. Denote by $t^*:=r^{-2}t'$ and $t^{**}:=r^{-2}t''$, then we have $|t^*-t^{**}|\leq \alpha$. Let $\gamma:[0,1]\to M_i$ be a $\tilde g_{i,t^*}$-minimizing geodesic between $(x_0, t^*)$ and $(x_1, t^*)$.

Let $(z_0, t^*-1)$ be an $\ell_{n}$-center of $(x_0, t^*)$, hence we have
$$\ell_{(x_0, t^*)}(z_0, t^*-1)\leq n/2.$$
Then we consider the function $K(x,t):=K(x,t;z_0, t^*-1)$, which satisfies that $\partial_t K=\Delta K$. Then we have
$$\frac{d}{dt} \int_{M_i} K(\cdot, t)d\tilde g_{i,t}=\int_{M_i} (\Delta K(\cdot, t) - R_{\tilde g_{i}}(\cdot, t)K(\cdot, t))d\tilde g_{i,t}\leq C\int_{M_i} K(\cdot, t)d\tilde g_{i,t}. $$
Hence for $t\in [t^*-\frac{1}{2}, t^*+\frac{1}{2}]$, we have
\begin{equation}\label{ibofhk}
\int_{M_i} K(\cdot, t)d\tilde g_{i,t}\leq e^{2n(t-(t^*-1))}. 
\end{equation}
Also, for all $t\in [t^*-\frac{1}{2}, t^*+\frac{1}{2}]$, we have
\begin{equation}\label{ubofhk}
K(\cdot, t)\leq B_1, 
\end{equation}
on $M_i$, for some constant $B_1=B_1(n,\omega_0)<\infty$. Hence we can apply \cite[Theorem 3.2]{ZhQ06} to obtain that
\begin{equation}\label{gbofhk}
\left| \nabla \sqrt{\ln  \frac{B_1}{K(\cdot, t)} } \right|_{\tilde g_{i,t}} \leq \sqrt{\frac{1}{ t - (t^*-\frac{1}{4}) }}\leq 10,
\end{equation}
for all $t\in [t^*-\frac{1}{8}, t^*+\frac{1}{8}]$. But we have
$$K(\gamma(0), t^*)=K(x_0, t^*)\geq \frac{1}{(4\pi(t^*-(t^*-1)))^{n/2}}e^{-\ell_{(x_0, t^*)}(z_0, t^*-1)}\geq c_0,$$
for some constant $c_0=c_0(n)>0$, hence we can integrate (\ref{gbofhk}) at $t=t^*$ along $\gamma$ to obtain that
\begin{equation}\label{hllbalonggamma}
K(\gamma(s), t^*)\geq c_0, ~~for ~~all~~s\in[0,1].
\end{equation}
Now, for any time $t$ between $t^*$ and $t^{**}$, we can apply \cite[Lemma 3.1]{BZ17} to obtain that 
\begin{equation}\label{dbohk}
|\partial_t K(\gamma(s), t)|\leq B_1(R_{\tilde g_{i}}(\gamma(s), t)+C(n,\omega_0)),
\end{equation}
for all $s\in[0,1]$, $t\in [t^*-\frac{1}{8}, t^*+\frac{1}{8}]$. From (\ref{boundvi2}), we have $v_i(\gamma(s), t')\leq C$ for all $s\in [0,1]$. Hence we have $v_i(\gamma(s), t)\leq C$ for all $s\in [0,1]$ and $t\in [-T-\eta, 0]$, which gives $R_{g_i}(\gamma(s), t)\leq C$ for all $s\in [0,1]$ and $t\in [-T-\eta, 0]$. Hence we have
\begin{equation}\label{scbatgam1}
R_{\tilde g_i}(\gamma(s), t)\leq Cr^2\leq 1 , 
\end{equation}
for all $s\in[0,1]$, $t\in [t^*-\frac{1}{8}, t^*+\frac{1}{8}]$. Hence from (\ref{dbohk}), we have
\begin{equation}\label{dbohk'2}
|\partial_t K(\gamma(s), t)|\leq 2B_1,
\end{equation}
for all $s\in[0,1]$, $t\in [t^*-\frac{1}{8}, t^*+\frac{1}{8}]$. Hence if we choose $\alpha>0$ small enough, we can integrate (\ref{dbohk'2}) and use (\ref{hllbalonggamma}) to obtain 
\begin{equation}\label{hllbalonggamma'2}
K(\gamma(s), t^{**})\geq c_1>0, ~~for ~~all~~s\in[0,1].
\end{equation}
Then we can integrate (\ref{gbofhk}) at $t=t^{**}$ to obtain that, for all $s\in[0,1]$
\begin{equation}\label{hllbalonggamma'3}
K(\cdot, t^{**})\geq c_2>0, ~~on~~B_{\tilde g_{i}}(\gamma(s),t^{**},1).
\end{equation}
Now we let $Q\geq 1$ be maximal subject to the fact that there are parameters $0\leq s_1 <s_2<\dots< s_Q\leq 1$ such that the balls $B_{\tilde g_{i}}(\gamma(s_1),t^{**},1)$, $\dots$, $B_{\tilde g_{i}}(\gamma(s_Q),t^{**},1)$ are mutually disjoint. Then the balls $B_{\tilde g_{i}}(\gamma(s_1),t^{**},2)$, $\dots$, $B_{\tilde g_{i}}(\gamma(s_Q),t^{**},2)$ cover $\gamma([0,1])$. Hence we have $d_{\tilde g_{i,t^{**}}}(x_0, x_1)\leq 4Q$. 

So we only need to bound $Q$. From $v_i(\gamma(s), t'')\leq C$ for all $s\in [0,1]$, for any $x\in B_{\tilde g_{i}}(\gamma(s),t^{**},1)$, we have $v_i(x, t'')\leq C$, hence $R_{g_i}(x, t'')\leq C$, hence we have 
$$R_{\tilde g_{i}}(\cdot, t^{**})\leq Cr^2\leq 1,~~on~~B_{\tilde g_{i}}(\gamma(s_j),t^{**},1),$$
for all $1\leq j\leq Q$, hence we can apply Lemma \ref{lvnc} to obtain that
$$\vol_{\tilde g_{i,t^{**}}}(B_{\tilde g_{i}}(\gamma(s_j),t^{**},1))\geq c_3>0.$$
Combining this with (\ref{ibofhk}) and (\ref{hllbalonggamma'3}), we have
$$e^{10n}\geq \int_{M_i} K(\cdot, t^{**})d\tilde g_{i,t^{**}}\geq \sum_{j=1}^{Q}\int_{B_{\tilde g_{i}}(\gamma(s_j),t^{**},1)} K(\cdot, t^{**})d\tilde g_{i,t^{**}}\geq Q\cdot c_2 \cdot c_3, $$
hence we have $Q\leq C_0$. This completes the proof of Lemma \ref{stdse}.
\end{proof}
Given $t_1 \in [t_0-r^2,t_0+r^2]$, then let $y_0\in B(x_0, t_0, Dr )$. We only consider the case $t_0<t_1$, the other case is similar. Let $Q:=\lfloor \alpha^{-1} \rfloor +1$, $\delta_0:=(t_1-t_0)/Q$, $s_j:=t_0+j\delta_0$ for $j = 0 , 1 , \dots, Q$. We then choose $\Bar{r}>0$ small enough such that
$$ C_0^Q D\Bar{r}\leq \tilde{r}, $$
where $C_0$ and $\tilde{r}$ are the constant from Lemma \ref{stdse}. We have $d_{g_{i,s_0}}(x_0, y_0)\leq Dr$. We prove by induction that 
\begin{equation}\label{iaonsj}
d_{g_{i,s_j}}(x_0, x_1)\leq C_0^j Dr,    
\end{equation}
for each $j = 1 , \dots, Q$. Indeed, if (\ref{iaonsj}) holds for $j-1$, then we have
$$ s_{j}-s_{j-1}=\delta_0\leq r^2/Q \leq \alpha r^2\leq \alpha (C_0^{j-1} D  r)^2, $$
and we have $C_0^{j-1} Dr\leq C_0^Q D\Bar{r}\leq \tilde{r} $, hence we can apply Lemma \ref{stdse} with $r \leftarrow C_0^{j-1} r$ to obtain (\ref{iaonsj}) holds for $j$. Apply (\ref{iaonsj}) with $j=Q$ shows that 
$$d_{g_{i,t_1}}(x_0, x_1)\leq C r.$$
This completes the proof of Proposition \ref{Tboundonsmallr}.
\end{proof}
%


\medskip
\subsection{Integral curvature bounds}
In this subsection, we will first prove integral curvature bounds on the flows $(M_i, (g_{i,t})_{t\in [-T_i , 0]})$, then pass it to the limit metric flow $\cX$.

First, we have the following covering result.

\begin{proposition}\label{eocnoMi}
For any $T\in (0, T_\infty)$, $A, B<\infty$, $0<\mathbf{p}<4$, there exists constants $H=H(A, B, T, \mathbf{p})<\infty$, $B_0=B_0(A, T)<\infty$, $\Bar{\lambda}=\Bar{\lambda}(A, B, T, \mathbf{p})>0$, such that whenever $B\geq B_0$, then the following statement holds.

For any $t_0\in(-T, 0)$, $\lambda\in (0, \Bar{\lambda})$, there exists constant $\Bar{r}>0$ depends on $\lambda, t_0, A, B$, $T$,$\mathbf{p}$, such that the following statement holds for all $i$ large enough.

Given any $i\geq i_0$, for any $(x_0, t_0)\in P^*(p_i, 0; A, -T)$, $0<r\leq\Bar{r}$, we can find $Q$-many points $y_1, \dots, y_Q \in M_i$, such that
\begin{enumerate}
    \item $Q\leq H\lambda^{\mathbf{p}-n}$;
    \item $(y_k, t_0) \in P^*(x_0, t_0+r^2; Br, -2r^2)$;
    \item $\left\{ r_{\Rm}\leq \lambda r \right\}\cap P^*(x_0, t_0+r^2; Br, -2r^2)\cap (M_i\times \left\{t_0\right\})\subset \bigcup_{j=1}^{Q}P^*(y_j, t_0+(\lambda r)^2; B\lambda r, -2(\lambda r)^2)$.
\end{enumerate}
\end{proposition}
\begin{proof}
We choose $\eta\in (0,1)$ such that $-T-100\eta> -T_\infty$, then we consider all $i$ large enough such that $-T-50\eta> -T_i$. We will determine $0<\Bar{r}=\Bar{r}(\lambda, t_0, A, B, T, \mathbf{p})<1$ in the course of the proof. We always require that $B\Bar{r}\leq 1$, $t_0-\Bar{r}^2\geq -T$, $t_0+\Bar{r}^2<0$, $0<r\leq\Bar{r}$.

We choose points 
$$(y_1, t_0), \dots, (y_Q, t_0) \in \left\{ r_{\Rm}\leq \lambda r \right\}\cap P^*(x_0, t_0+r^2; Br, -2r^2)\cap (M_i\times \left\{t_0\right\}),$$
with $Q$ being maximal subject to the fact that the subsets 
$$
P^*(y_j, t_0+(\lambda r)^2; 3H_{n}^{1/2}\lambda r, -2(\lambda r)^2)\cap (M_i\times \left\{t_0\right\})
$$
are mutually disjoint for $1\leq j \leq Q$.

First, we have
\begin{lemma}\label{P*nbhdofyix0}
There exists constant $B_1=B_1(A, B, T)<\infty$, such that if we choose $i_0=i_0(A, T)<\infty$ large enough, then we have
$$P^*(y_j, t_0+(\lambda r)^2; 3H_{n}^{1/2}\lambda r, -2(\lambda r)^2) \subset P^*(x_0, t_0+r^2; B_1 r, -2r^2),$$
for all $1\leq j \leq Q$.
\end{lemma}
\begin{proof}
First, since $(x_0, t_0)\in P^*(p_i, 0; A, -T)$, we can apply Proposition \ref{viprop3} to obtain that $v_i(x_0, t_0)\leq C$, hence
\begin{equation}\label{boundvi3}
v_i(x_0, t)\leq C(A, T) ,
\end{equation}
for all $t\in [-T, 0]$. Hence we have $R_{g_i}(x_0, t)\leq C(A, T)$ for all $t\in [-T, 0]$, then we can apply Lemma \ref{W1balongwl} to obtain
\begin{equation}\label{W1dbx0}
d^{g_{i, t_0-r^2}}_{W_1} (\nu_{x_0, t_0+r^2; t_0-r^2} , \nu_{x_0, t_0; t_0-r^2} ) \leq d^{g_{i, t_0}}_{W_1} (\nu_{x_0, t_0+r^2; t_0} , \delta_{x_0} ) \leq C(A, T)r.
\end{equation}
Next, from $(y_j, t_0) \in P^*(x_0, t_0+r^2; Br, -2r^2)$ we have
\begin{equation}\label{W1dbyjx0t0}
d^{g_{i, t_0-r^2}}_{W_1} (\nu_{x_0, t_0+r^2; t_0-r^2} , \nu_{y_j, t_0; t_0-r^2} ) \leq Br , 
\end{equation}
hence we can combine (\ref{W1dbx0}) and (\ref{W1dbyjx0t0}) to obtain that
\begin{equation}\label{W1dbyjx0t0+r2}
d^{g_{i, -T}}_{W_1} (\nu_{x_0, t_0; -T} , \nu_{y_j, t_0; -T} )\leq d^{g_{i, t_0-r^2}}_{W_1} (\nu_{x_0, t_0; t_0-r^2} , \nu_{y_j, t_0; t_0-r^2} ) \leq C(A, B, T)r.
\end{equation}
Combining this with $(x_0, t_0)\in P^*(p_i, 0; A, -T)$, we have
\begin{equation}\label{W1dbyjpi}
d^{g_{i, -T}}_{W_1} (\nu_{p_i, 0; -T} , \nu_{y_j, t_0; -T} ) \leq A + C(A, B, T)r \leq A+1 ,
\end{equation}
if we choose $\Bar{r}>0$ small enough. Using (\ref{W1dbyjpi}), we can apply Proposition \ref{viprop3} to obtain that $v_i(y_j, t_0)\leq C(A, T)$, hence
\begin{equation}\label{boundvi4}
v_i(y_j, t)\leq C(A, T) ,
\end{equation}
for all $t\in [-T, 0]$. Hence we have $R_{g_i}(y_j, t)\leq C(A, T)$ for all $t\in [-T, 0]$, then we can apply Lemma \ref{W1balongwl} to obtain
\begin{equation}\label{W1dbyj}
\begin{split}
d^{g_{i, t_0-r^2}}_{W_1} (\nu_{y_j, t_0+(\lambda r)^2; t_0-r^2} , \nu_{y_j, t_0; t_0-r^2} ) \leq d^{g_{i, t_0}}_{W_1} (\nu_{y_j, t_0+(\lambda r)^2; t_0} , \delta_{y_j} ) \leq C(A, T)r.
\end{split}
\end{equation}
Hence we can combine (\ref{W1dbyjx0t0}) and (\ref{W1dbyj}) to obtain that
\begin{equation}\label{W1dbyjx0'2}
\begin{split}
d^{g_{i, t_0-r^2}}_{W_1} ( \nu_{x_0, t_0+r^2; t_0-r^2} , \nu_{y_j, t_0+(\lambda r)^2; t_0-r^2} ) \leq C(A, B, T)r.
\end{split}
\end{equation}
On the other hand, for any $(\hat{y}, \hat{t})\in P^*(y_j, t_0+(\lambda r)^2; 3H_{n}^{1/2}\lambda r, -2(\lambda r)^2)$, we have
\begin{equation}\label{W1dbhatyyj}
\begin{split}
d^{g_{i, t_0-r^2}}_{W_1} &( \nu_{y_j, t_0+(\lambda r)^2; t_0-r^2} , \nu_{\hat{y}, \hat{t}; t_0-r^2} )\\
&\leq d^{g_{i, t_0-(\lambda r)^2}}_{W_1} ( \nu_{y_j, t_0+(\lambda r)^2; t_0-(\lambda r)^2} , \nu_{\hat{y}, \hat{t}; t_0-(\lambda r)^2} ) \leq 3H_{n}^{1/2}\lambda r .
\end{split}
\end{equation}
Combining (\ref{W1dbyjx0'2}) and (\ref{W1dbhatyyj}), we obtain
$$ d^{g_{i, t_0-r^2}}_{W_1} ( \nu_{x_0, t_0+r^2; t_0-r^2} , \nu_{\hat{y}, \hat{t}; t_0-r^2} ) \leq C(A, B, T)r. $$
This completes the proof of Lemma \ref{P*nbhdofyix0}.
\end{proof}
Next, we have
\begin{lemma}\label{P*nbhdofyirm}
There exists constant $B_2=B_2(A, T)<\infty$, such that
$$P^*(y_j, t_0+(\lambda r)^2; 3H_{n}^{1/2}\lambda r, -2(\lambda r)^2) \subset \left\{ r_{\Rm}\leq B_2\lambda r \right\},$$
for all $1\leq j \leq Q$.
\end{lemma}
\begin{proof}
First note that $r_{\Rm}(y_j, t_0)\leq \lambda r$, using Lemma \ref{deofrrm}, we have
\begin{equation}\label{eormatyj}
r_{\Rm}(y_j, t_0-(\lambda r)^2)\leq C_0(n)\lambda r.
\end{equation}
Next, we fix $(\hat{y}, \hat{t})\in P^*(y_j, t_0+(\lambda r)^2; 3H_{n}^{1/2}\lambda r, -2(\lambda r)^2)$. Then from (\ref{boundvi4}) and Proposition \ref{viprop3}, we can get $v_i(\hat{y}, \hat{t})\leq C(A, T)$, hence
\begin{equation}\label{boundvi5}
v_i(\hat{y}, t)\leq C(A, T) ,
\end{equation}
for all $t\in [-T, 0]$. Hence we have $R_{g_i}(\hat{y}, t)\leq C(A, T)$ for all $t\in [-T, 0]$, then we can apply Lemma \ref{W1balongwl} to obtain
\begin{equation}\label{W1dbhaty}
d^{g_{i, t_0-(\lambda r)^2}}_{W_1} ( \delta_{\hat{y}} , \nu_{\hat{y}, \hat{t}; t_0-(\lambda r)^2} ) \leq C(A, T)\lambda r .
\end{equation}
Next, we have
\begin{equation}\label{W1dbhatyyi}
d^{g_{i, t_0-(\lambda r)^2}}_{W_1} ( \nu_{\hat{y}, \hat{t}; t_0-(\lambda r)^2} ,  \nu_{y_j, t_0+(\lambda r)^2; t_0-(\lambda r)^2} ) \leq 3H_{n}^{1/2}\lambda r .
\end{equation}
Finally, similar to (\ref{W1dbyj}), we have
\begin{equation}\label{W1dbyj2}
\begin{split}
d^{g_{i, t_0-(\lambda r)^2}}_{W_1} (\nu_{y_j, t_0+(\lambda r)^2; t_0-(\lambda r)^2} , \delta_{y_j} ) \leq C(A, T)r.
\end{split}
\end{equation}
Combining (\ref{W1dbhaty}), (\ref{W1dbhatyyi}) and (\ref{W1dbyj2}), we obtain from the triangle inequality that
\begin{equation}\label{dbbyjhatz}
\begin{split}
d_{g_{i, t_0-(\lambda r)^2}} ( y_j , \hat{y} ) \leq C(A, T)\lambda r.
\end{split}
\end{equation}
Combining Lemma \ref{deofrrm}, (\ref{eormatyj}) and (\ref{dbbyjhatz}) we obtain
\begin{equation}\label{eormathatz}
r_{\Rm}(\hat{y}, t_0-(\lambda r)^2)\leq C(A, T)\lambda r.
\end{equation}
But we have $\hat{t}-(t_0-(\lambda r)^2)<2(\lambda r)^2$, hence by Lemma \ref{deofrrm} again, we have
$$
r_{\Rm}(\hat{y}, \hat{t} )\leq C(A, T)\lambda r.
$$
This completes the proof of Lemma \ref{P*nbhdofyirm}.
\end{proof}
Next, we have
\begin{lemma}\label{icbax0}
There exists constant $E_{\mathbf{p}}=E_{\mathbf{p}}(A, B, T)<\infty$, $0<\Bar{r}=\Bar{r}(\lambda, t_0, A, B, T, \mathbf{p})<1$, such that for all $i$ large enough, we have
\begin{equation}
\left| \left\{ r_{\Rm}\leq B_2\lambda r \right\}\cap P^*(x_0, t_0+r^2; B_1r, -2r^2)\cap (M_i\times \left\{t_0\right\}) \right|_{g_{i, t_0}}   \leq E_{\mathbf{p}}\lambda^{\mathbf{p}}r^{n}.  
\end{equation}
\end{lemma}
\begin{proof}
We will determine the constant $E_{\mathbf{p}}=E_{\mathbf{p}}(A, B, T)<\infty$ in the course of the proof. 

Assume such $\Bar{r}$ does not exist. Then passing to a subsequence, for each $i$, we can find a point $(x_i, t_0)\in P^*(p_i, 0; A, -T)$, $r_i\to 0^+$ such that
\begin{equation}\label{coicbx0}
\left| \left\{ r_{\Rm}\leq B_2\lambda r_i \right\}\cap P^*(x_i, t_0+r_i^2; B_1r_i, -2r_i^2)\cap (M_i\times \left\{t_0\right\}) \right|_{g_{i, t_0}}   \geq E_{\mathbf{p}}\lambda^{\mathbf{p}}r_i^{n} ,
\end{equation}
for all $i\in \mathbb{N}$. Let $(\hat M_{i}, (\hat g_{i,t})_{t\in [-\hat T_{i} , 0]})$ be the flows arise from $(M_i, (g_{i,t})_{t\in [-T_i , 0]})$ by setting 
$$ \hat M_{i}:=M_i,~~~ \hat g_{i, t}:=r_i^{-2}g_{i,r_i^2t+(t_0+r_i^2)},~~~ t\in [-\hat T_{i} , 0], $$
where $\hat T_{i}\to\infty$ as $i\to\infty$ are chosen such that $-T_i=-r_i^2\hat T_{i}+(t_0+r_i^2)$. Then from (\ref{coicbx0}), on the flow $(M_{i,k}, (g_{i,k,t})_{t\in [-T_{i,k} , 0]})$, we have
\begin{equation}\label{ceooxk'1}
\left| \left\{ r_{\Rm}\leq B_2\lambda  \right\}\cap P^*(x_i, 0; B_1, -2)\cap (\hat M_{i}\times \left\{-1\right\}) \right|_{\hat g_{i, -1}} \geq E_{\mathbf{p}}\lambda^{\mathbf{p}}.  
\end{equation}
Passing to a subsequence, we have $\bF$-convergence on compact time-intervals
\begin{equation}\label{FcorKRFofxk}
(\hat M_{i}, (\hat g_{i,t})_{t\in [-\hat T_{i} , 0]}, (\nu_{x_i,0; t})_{t\in [-\hat T_{i} , 0]}) \xrightarrow[i\to\infty]{\bF,\CCC}  (\hat{\cX}, (\nu_{x_\infty; t})_{t\in (-\infty , 0]}),
\end{equation}
within some correspondence $\CCC$, where $\hat{\cX}$ is a future continuous and $H_{n}$-concentrated metric flow of full support over $(-\infty , 0]$. We denote by $\hat{\cX}= \hat{\cR} \sqcup \hat{\cS}$ the regular-singular decomposition with $\hat{\cR}$ carries the structure of a Ricci flow spacetime $(\hat{\cR}, \mathfrak{t}, \partial_{\mathfrak{t}}, \hat{g})$, and let $\hat U_{i}\subset \hat{\cR}$, $\hat V_{i}\subset M_{i}\times [-\hat T_{i}, 0]$ be the open subsets where we have local smooth convergence, with time-preserving diffeomorphisms $\hat\psi_{i}:\hat U_{i}\to \hat V_{i}$ for each $i=1,2,\dots$.

Let $\delta=\delta(\lambda, E_{\mathbf{p}}, A, B, T)>0$ be a small constant to be determined. Then we can apply \cite[Lemma 15.27, (a)]{Bam20c} to obtain that
\begin{equation}\label{ceooxk'2}
\left| \left\{ r_{\Rm}\leq 2\delta  \right\}\cap P^*(x_i, 0; B_1, -2)\cap (\hat M_{i}\times \left\{-1\right\}) \right|_{\hat g_{i,-1}} \leq C(A, B, T)\delta,
\end{equation}
hence we can combine (\ref{ceooxk'1}) and (\ref{ceooxk'2}) to get
\begin{equation}\label{ceooxk'3}
\left| W_{i} \right|_{\hat g_{i,-1}} \geq \frac{1}{2}E_{\mathbf{p}}\lambda^{\mathbf{p}},
\end{equation}
where $W_{i}\subset \hat M_{i}$ is defined by
$$W_{i}:=\left\{ 2\delta\leq r_{\Rm}\leq B_2\lambda \right\}\cap P^*(x_i, 0; B_1, -2)\cap (\hat M_{i}\times \left\{-1\right\}),$$
provided that $\delta$ is chosen small enough. 

We claim that, for all $i$ large, for any $(y,-1)\in P^*(x_i, 0; B_1, -2)$, we have
\begin{equation}\label{dbyxi}
\begin{split}
d_{\hat g_{i,-1}}(y, x_i)\leq C(A, B, T) .
\end{split}
\end{equation}
Indeed, by Proposition \ref{viprop3}: since $(x_i, t_0)\in P^*(p_i, 0; A, -T)$, we have $v_i(x_i, t_0)\leq C(A, T)$; then from $(y,t_0)\in P^*(x_i, t_0+r_i^2; B_1r_i, -2r_i^2)$, we have $v_i(y, t_0)\leq C(A, B, T)$. Hence $v_i(x_i, t), v_i(y, t)\leq C(A, B, T)$ for all $t\in [-T, 0]$, and we then have
\begin{equation}\label{scbatxiy}
R_{g_i}(x_i, t), R_{g_i}(y, t)\leq C(A, B, T),
\end{equation}
for all $t\in [-T, 0]$. Hence we can apply Lemma \ref{W1balongwl} to get
\begin{equation}\label{dW1bx0viP2}
\begin{split}
d^{g_{i, t_0-r_i^2}}_{W_1} (\nu_{x_i, t_0+ r_i^2; t_0-r_i^2} , \delta_{x_i} )  \leq C(A, B, T)r_i ,
\end{split}
\end{equation}
and 
\begin{equation}\label{dW1byviP}
\begin{split}
d^{g_{i, t_0-r_i^2}}_{W_1} (\nu_{y, t_0; t_0-r_i^2} , \delta_{y} )  \leq C(A, B, T)r_i ,
\end{split}
\end{equation}
hence by the triangle inequality, we have $d_{g_{i, t_0-r_i^2}} (y , x_i )  \leq C(A, B, T)r_i$. Then for all $i$ large enough (hence $r_i$ small enough), we can apply Proposition \ref{Tboundonsmallr} to obtain (\ref{dbyxi}).

Now, if we set $\hat v_i(t)=v_i(r_i^2t+(t_0+r_i^2))$, then $\hat v_i$ is a $r_i^2C_i$-barrier of $R_{\hat g_i}$, and is $C(A, B, T)$-based at $(x_i, 0)$ . Hence by (\ref{dbyxi}), we can apply Proposition \ref{copuug'1} to obtain that $W_{i}\subset \hat V_{i}$ for all $i$ large enough (if this is not true, we can pass to a subsequence to find points $(y_i,-1)\in W_{i}\setminus \hat V_{i}$, this contradicts to item (3) of Proposition \ref{copuug'1}). Hence for $i$ large enough, $\hat\psi_{i}^{-1}(W_{i})\subset \hat U_{i}$ is well-defined. Then we have

\begin{claim}\label{cowik}
For $i$ large enough, we have
$$ 
\hat\psi_{i}^{-1}(W_{i})\subset \left\{ \delta\leq \tilde r_{\Rm}\leq 2B_2\lambda \right\}\cap P^*(x_\infty; 2B_1, -3)\cap \hat{\cR}_{-1}.
$$
\end{claim}
\begin{proof}
If this is not true, passing to a subsequence, we can find a sequence of points $(y_i,-1)\in W_{i}$ but $\hat\psi_{i}^{-1}(y_i,-1)\notin \left\{ \delta\leq \tilde r_{\Rm}\leq 2B_2\lambda \right\}\cap P^*(x_\infty; 2B_1, -3)$. Using Proposition \ref{copuug'1} again, passing to a subsequence, there exists a point $y_\infty\in \hat{\cR}_{-1}$, such that we have $(y_i,-1)\xrightarrow[i\to\infty]{\CCC}y_\infty$, and $ \tilde r_{\Rm}(y_\infty) = \lim_{i\to\infty} r_{\Rm}(y_i,-1)\in [2\delta, B_2\lambda]$. According to \cite[Lemma 15.8]{Bam20c}, we have $y_\infty\in P^*(x_\infty; 2B_1, -3)$, and since $\hat\psi_{i}^{-1}(y_i,-1)\to y_\infty$ in $\hat{\cR}_{-1}$, we have $\hat\psi_{i}^{-1}(y_i,-1)\in \left\{ \delta\leq \tilde r_{\Rm}\leq 2B_2\lambda \right\}\cap P^*(x_\infty; 2B_1, -3)$ for $i$ large enough. This is a contradiction which proves the claim.
\end{proof}
Due to \cite[Lemma 15.16, (h)]{Bam20c}, the set $\left\{ \delta\leq \tilde r_{\Rm}\leq 2B_2\lambda \right\}\cap P^*(x_\infty; 2B_1, -3)\cap \hat{\cR}_{-1}$ is relatively compact in $\hat{\cR}_{-1}$, hence due to the local smooth convergence on $\hat{\cR}_{-1}$, we have from (\ref{ceooxk'3}) and Claim \ref{cowik} that
\begin{equation}\label{ceooxinfty'1}
\left| \left\{ \delta\leq \tilde r_{\Rm}\leq 2B_2\lambda \right\}\cap P^*(x_\infty; 2B_1, -3)\cap \hat{\cR}_{-1} \right|_{\hat{g}_{-1}} \geq \frac{1}{2} E_{\mathbf{p}}\lambda^{\mathbf{p}}. 
\end{equation}
On the other hand, from (\ref{scbatxiy}), we have
\begin{equation}\label{scbatxk}
R_{g_i}(x_i, t_0+r_i^2)\leq C(A, B, T), 
\end{equation}
for all $i$ large. Also, we have $R_{g_i}\geq -C(T)$ on $M_i\times [-T, 0]$. Hence we have
$$R_{\hat g_{i}}(x_i, 0)\leq C(A, B, T)r_i^2\to 0,~~as~~i\to\infty,$$
and $R_{\hat g_{i}}\geq -C(n, \omega_0)r_i^2$ on $\hat M_i\times [-T', 0]$ for any $T'<\infty$, for $i$ large enough. Hence we have
\begin{equation}\label{scbatxk2}
\int_{-T'}^{0}\int_{\hat M_{i}} \left| \ric_{\hat g_{i}} \right|^2 d\nu_{x_i,0;t}dt\leq R_{\hat g_{i}}(x_i,0)+C(A,B, T)r_i^2\to 0,~~as~~i\to\infty . 
\end{equation}
Hence we can apply \cite[Theorem 15.60]{Bam20c} to obtain that $(\hat{\cX}, (\nu_{x_\infty; t})_{t\in (-\infty , 0]})$ is a static limit. 

We claim that, there exists $C_0=C_0(Y_0)<\infty$, such that for any $x\in P^*(x_\infty; 2B_1, -3)\cap \hat{\cR}_{-1}$, for any $t\in [-2, -1]$, we have
\begin{equation}\label{containofxt}
x(t)\in P^*(x_\infty; 2B_1+C_0, -3)\cap \hat{\cR}_{t}  ,
\end{equation}
where $x(t)$ denotes the point survive to time $t$. Indeed, let $t\in [-2, 1)$, consider the curve $\gamma(\tau)=x(-1-\tau)$, $\tau\in [0, -1-t]$. Since $\ric\equiv 0$ on $\hat{\cR}$, we have $\cL(\gamma)=0$, hence from \cite[Lemma 22.2]{Bam20c}, we have 
$$
d^{\hat{\chi}_{-3}}_{W_1} (\nu_{x; -3} , \nu_{x(t); -3} ) \leq d^{\hat{\chi}_{-3}}_{W_1} (\nu_{x; t} , \delta_{x(t)} ) < C(Y_0) ,
$$
hence by triangle inequality,
$$
d^{\hat{\chi}_{-3}}_{W_1} (\nu_{x_\infty; -3} , \nu_{x(t); -3} ) \leq d^{\hat{\chi}_{-3}}_{W_1} (\nu_{x_\infty; -3} , \nu_{x; -3} ) +  d^{\hat{\chi}_{-3}}_{W_1} (\nu_{x; -3} , \nu_{x(t); -3} ) < 2B_1 + C(Y_0) ,
$$
this proves (\ref{containofxt}).

Then we can apply the first inequality in \cite[Lemma 15.27, (a)]{Bam20c} (taking $I_\infty=[-2,-1]$ there is enough) to obtain that (note that $\mathbf{p}\in (0,4)$)
$$
\int_{-2}^{-1} \int_{\left\{ 0< \tilde r_{\Rm}\leq 2B_2\lambda \right\}\cap P^*(x_\infty; 2B_1+C_0, -3)\cap \hat{\cR}_{t}} d\hat{g}_t dt \leq B_3\lambda^{\mathbf{p}},
$$
for some constant $B_3=B_3(\mathbf{p}, B_1, B_2)=B_3(\mathbf{p}, A, B, T)<\infty$. Using (\ref{containofxt}) and the fact that $\hat{\cX}$ is static limit, we conclude
\begin{equation}\label{ceooxinfty'2}
\left| \left\{ 0< \tilde r_{\Rm}\leq 2B_2\lambda \right\}\cap P^*(x_\infty; 2B_1, -3)\cap \hat{\cR}_{-1} \right|_{\hat{g}_{-1}} \leq B_3\lambda^{\mathbf{p}},
\end{equation}
Hence if we choose $E_{\mathbf{p}}\geq 10B_3$, then we obtain a contradiction from (\ref{ceooxinfty'1}) and (\ref{ceooxinfty'2}). 

This completes the proof of Lemma \ref{icbax0}.
\end{proof}

\begin{lemma}\label{eoQ}
There exists constants $H=H(\mathbf{p}, A, B, T)<\infty$, $0<\Bar{r}=\Bar{r}(\lambda, t_0, A, B, T, \mathbf{p})<1$, such that
\begin{equation}
Q\leq H\lambda^{\mathbf{p}-n}.  
\end{equation}
\end{lemma}
\begin{proof}
Let $0<\Bar{r}=\Bar{r}(\lambda, t_0, A, B, T, \mathbf{p})<1$ be the constant from Lemma \ref{icbax0}.

For each $1\leq j\leq Q$, we let $(z_j, t_0)$ be an $H_{n}$-center of $(y_j, t_0+(\lambda r)^2)$, then we have
\begin{equation}\label{eoWdbyjzj}
d^{g_{i, t_0-(\lambda r)^2}}_{W_1} (\nu_{y_j, t_0+(\lambda r)^2; t_0-(\lambda r)^2} , \nu_{z_j, t_0; t_0-(\lambda r)^2} ) \leq d^{g_{i, t_0}}_{W_1} (\nu_{y_j, t_0+(\lambda r)^2; t_0} , \delta_{z_j} ) \leq H_{n}^{1/2}\lambda r.    
\end{equation}
But for any $(z, t_0)\in B_{g_i}(z_j, t_0, (2H_{n})^{1/2}\lambda r)$, we have
\begin{equation}\label{eoWdbzjz}
d^{g_{i, t_0-(\lambda r)^2}}_{W_1} (\nu_{z_j, t_0; t_0-(\lambda r)^2} , \nu_{z, t_0; t_0-(\lambda r)^2} ) \leq d_{g_{i, t_0}} (z_j , z ) \leq 2H_{n}^{1/2}\lambda r.    
\end{equation}
Hence we can combine (\ref{eoWdbyjzj}) and (\ref{eoWdbzjz}) to obtain that
\begin{equation}\label{eoWdbyjz}
d^{g_{i, t_0-(\lambda r)^2}}_{W_1} (\nu_{y_j, t_0+(\lambda r)^2; t_0-(\lambda r)^2} , \nu_{z, t_0; t_0-(\lambda r)^2} ) \leq  3H_{n}^{1/2}\lambda r,
\end{equation}
hence we have
$$ B(z_j, t_0, (2H_{n})^{1/2}\lambda r)\subset P^*(y_j, t_0+(\lambda r)^2; 3H_{n}^{1/2}\lambda r, -2(\lambda r)^2). $$
Hence from Lemma \ref{hcv}, we have
\begin{equation}
\left| P^*(y_j, t_0+(\lambda r)^2; 3H_{n}^{1/2}\lambda r, -2(\lambda r)^2) \cap (M_i\times \left\{t_0\right\}) \right|_{g_{i, t_0}} \geq c(\lambda r)^{n}, 
\end{equation}
for some constant $c=c(n, \omega_0)>0$, but these subsets are mutually disjoint, hence we can combine Lemma \ref{P*nbhdofyix0}, Lemma \ref{P*nbhdofyirm} and Lemma \ref{icbax0} to obtain that, for all $i$ large enough, we have
$$ Q\cdot c(\lambda r)^{2n} \leq E_{\mathbf{p}}\lambda^{\mathbf{p}}r^{2n}, $$
where $E_{\mathbf{p}}$ is the constant from Lemma \ref{icbax0}. This completes the proof.
\end{proof}
Lemma \ref{eoQ} proves item (1) of Proposition \ref{eocnoMi}.

Finally, we come to prove item (3) of Proposition \ref{eocnoMi}. Assume there is a point $(y_0, t_0)\in \left\{ r_{\Rm}\leq \lambda r \right\}\cap P^*(x_0, t_0+r^2; Br, -2r^2)$, but 
\begin{equation}\label{eoWdbyjzy0}
d^{g_{i, t_0-(\lambda r)^2}}_{W_1} (\nu_{y_0, t_0; t_0-(\lambda r)^2} , \nu_{y_j, t_0+(\lambda r)^2; t_0-(\lambda r)^2} ) \geq B\lambda r,
\end{equation}
for all $j=1,\dots, Q$. As the arguments of (\ref{W1dbyjpi}) to (\ref{W1dbyj}) we have $R_{g_i}(y_0, t)\leq C(A,T)$ for all $t\in [-T, 0]$ and 
\begin{equation}\label{W1dby0}
\begin{split}
d^{g_{i, t_0-(\lambda r)^2}}_{W_1} (\nu_{y_0, t_0; t_0-(\lambda r)^2} , \nu_{y_0, t_0+(\lambda r)^2; t_0-(\lambda r)^2} ) \leq d^{g_{i, t_0}}_{W_1} (\nu_{y_0, t_0+(\lambda r)^2; t_0} , \delta_{y_0} ) \leq C(A, T)\lambda r.
\end{split}
\end{equation}
Let $(\hat{y}, \hat{t})\in P^*(y_0, t_0+(\lambda r)^2; 3H_{n}^{1/2}\lambda r, -2(\lambda r)^2)$, then from (\ref{W1dby0}), we have
\begin{equation}\label{W1dbhatyy0}
\begin{split}
d^{g_{i, t_0-(\lambda r)^2}}_{W_1} (\nu_{\hat{y}, \hat{t}; t_0-(\lambda r)^2} , \nu_{y_0, t_0; t_0-(\lambda r)^2} ) \leq C(A, T)\lambda r.
\end{split}
\end{equation}
Combining (\ref{eoWdbyjzy0}) and (\ref{W1dbhatyy0}), we have 
\begin{equation}\label{W1dbhathatyyj}
\begin{split}
d^{g_{i, t_0-(\lambda r)^2}}_{W_1} (\nu_{\hat{y}, \hat{t}; t_0-(\lambda r)^2} , \nu_{y_j, t_0+(\lambda r)^2; t_0-(\lambda r)^2} ) \geq (B-C(A, T))\lambda r\geq 10H_{n}^{1/2}\lambda r,
\end{split}
\end{equation}
provided that $B\geq 2C(A, T)$. Hence we have
$$ P^*(y_0, t_0+(\lambda r)^2; 3H_{n}^{1/2}\lambda r, -2(\lambda r)^2) \cap P^*(y_j, t_0+(\lambda r)^2; 3H_{n}^{1/2}\lambda r, -2(\lambda r)^2) = \emptyset $$
for all $j=1,\dots, Q$. But this contradicts the maximality of $Q$. 

This completes the proof of Proposition \ref{eocnoMi}.
\end{proof}
Applying Proposition \ref{eocnoMi} successively for sufficiently small $\lambda$ yields

\begin{proposition}\label{icbonflow}
For any $T\in (0, T_\infty)$, $A, B<\infty$, $0<\mathbf{p}<4$, $t_0\in(-T, 0)$, there exists constant $E_{\mathbf{p}}=E_{\mathbf{p}}(A, B, T)<\infty$, such that the following statement holds.
For any $t_0\in(-T, 0)$, there exists constant $\Bar{r}>0$ depends on $t_0, A, B, T, \mathbf{p}$, such that the following statement holds for all $i$ large enough.

Given any $i\geq i_0$, for any $(x_0, t_0)\in P^*(p_i, 0; A, -T)$, $0<r\leq\Bar{r}$, $0<s\leq 1$, we have
\begin{equation}
\left| \left\{ r_{\Rm}\leq sr \right\}\cap P^*(x_0, t_0+r^2; Br, -2r^2)\cap (M_i\times \left\{t_0\right\}) \right|_{g_{i, t_0}}   \leq E_{\mathbf{p}}s^{\mathbf{p}}r^{n}.  
\end{equation}
\end{proposition}
\begin{proof}
We choose $\eta\in (0,1)$ such that $-T-100\eta> T_\infty$, then we consider all $i$ large enough such that $-T-50\eta> -T_i$. We always require that $t_0-\Bar{r}^2\geq -T$, $t_0+\Bar{r}^2<0$, $0<r\leq\Bar{r}$.

Let $B_0=B_0(A+1, T)$ be the constant from Proposition \ref{eocnoMi}. We the replace $B$ by $\max\left\{ B, B_0 \right\}$.

Let $\mathbf{p}'=(\mathbf{p}+4)/2\in (\mathbf{p}, 4)$, then let $H=H(A+1, B, T, \mathbf{p}')$ and $\Bar{\lambda}=\Bar{\lambda}(A+1, B, T, \mathbf{p}')>0$ be the constants from Proposition \ref{eocnoMi}. Choose $0<\lambda=\lambda(A, B, T, \mathbf{p})\leq \Bar{\lambda}/2$ such that 
\begin{equation}\label{doflambda}
H\lambda^{\mathbf{p}'-\mathbf{p}}\leq 1.  
\end{equation}
Now, for any $k\in \mathbb{N}^+$, we apply Proposition \ref{eocnoMi} successively for $k$ times to find points $(y_1, t_0), \dots, (y_Q, t_0) \in P^*(x_0, t_0+r^2; 2Br, -2r^2)$ with $Q\leq (H\lambda^{\mathbf{p}'-n})^k$ such that
\[
\begin{split}
\left\{ r_{\Rm}\leq \lambda^k r \right\}\cap P^*(x_0, t_0+r^2; & Br, -2r^2)\cap (M_i\times \left\{t_0\right\})\\
&\subset \bigcup_{j=1}^{Q}P^*(y_j, t_0+(\lambda^k r)^2; B\lambda^k r, -2(\lambda^k r)^2),
\end{split}
\]
hence we can apply \cite[Theorem 9.8]{Bam20a} to obtain that
\[
\begin{split}
&\left| \left\{ r_{\Rm}\leq \lambda^k r \right\}\cap P^*(x_0, t_0+r^2; Br, -2r^2)\cap (M_i\times \left\{t_0\right\}) \right|_{g_{i, t_0}} \\
&\leq (H\lambda^{\mathbf{p}'-n})^k\cdot C(A, T)(B\lambda^k r)^{n}\leq E_{\mathbf{p}}(H\lambda^{\mathbf{p}'-\mathbf{p}})^k \lambda^{\mathbf{p}k} r^{n}\leq E_{\mathbf{p}} \lambda^{\mathbf{p}k} r^{n}.
\end{split}
\]
This completes the proof.
\end{proof}
Now we can pass the integral curvature bound in Proposition \ref{icbonflow} to the limit $\cX$.

\begin{proposition}\label{icbonlimitflow}
For any $T\in (0, T_\infty)$, $A, B<\infty$, $0<\mathbf{p}<4$, $t_0\in(-T, 0)$, there exist constants $0<\Bar{r}=\Bar{r}( t_0, A, B, T, \mathbf{p} )<1$, $E_{\mathbf{p}}=E_{\mathbf{p}}(A, B, T)<\infty$, such that the following statement holds.

For any $y_\infty\in P^*(p_\infty; A, -T)\cap\cX_{t_0}$, $0<r\leq\Bar{r}$, $0<s\leq 1$, we have
\begin{equation}\label{icbonlimitflow'2}
\left| \left\{ 0<\tilde r_{\Rm}\leq sr \right\}\cap P^*\left(y_\infty; Br, -\frac{3}{2}r^2\right)\cap \cR_{t_0-r^2} \right|_{g_{t_0-r^2}}   \leq E_{\mathbf{p}}s^{\mathbf{p}}r^{2n}.  
\end{equation}
\end{proposition}
\begin{proof}
We choose $\eta\in (0,1)$ such that $-T-100\eta> -T_\infty$. We then require that $\Bar{r}\leq \eta$.

Using \cite[Theorem 6.45]{Bam20b}, we can find points $(y_i, t_0)\in M_i\times \left\{t_0\right\}$ such that $(y_i, t_0)\xrightarrow[i\to\infty]{\CCC}y_\infty$.

Let $\delta>0$ be any small constant, then we define the region
$$
W_{\delta}:=\left\{ 2\delta\leq \tilde r_{\Rm}\leq s r \right\}\cap P^*\left(y_\infty; Br, -\frac{3}{2}r^2\right)\cap \cR_{t_0-r^2} .
$$
According to \cite[Lemma 15.16, (h)]{Bam20c}, we know $W_{\delta}\subset\cR_{t_0-r^2}$ is a relatively compact subset, hence $W_{\delta}\subset U_i$ for $i$ large enough. Hence $\psi_i(W_{\delta})\subset V_i$ are well-defined subsets. Let 
$$W_{i}:=\left\{ \delta\leq r_{\Rm}\leq 2s r \right\}\cap P^*\left(y_i, t_0; 2Br, -2r^2\right)\cap (M_i\times \left\{t_0-r^2\right\}).$$
We claim that $\psi_i(W_{\delta})\subset W_i$ for $i$ large enough. If this is not true, passing to a subsequence if necessary, we can find points $(\tilde y_i, t_0-r^2)\in \psi_i(W_{\delta})\setminus W_i$. Since $W_{\delta}\subset\cR_{t_0-r^2}$ is relatively compact, passing to a subsequence if necessary, we have $\psi_i^{-1}(\tilde y_i, t_0-r^2) \to \tilde y_\infty\in \overline{W_\delta}$. Hence by \cite[Theorem 9.31]{Bam20b}, we have $(\tilde y_i, t_0-r^2)\xrightarrow[i\to\infty]{\CCC}y_\infty$, hence $\lim_{i\to\infty} r_{\Rm}(\tilde y_i, t_0-r^2) = \tilde r_{\Rm}(\tilde y_\infty)\in [2\delta, sr]$, combining with \cite[Lemma 15.8]{Bam20c}, we have $(\tilde y_i, t_0-r^2)\in W_i$ for $i$ large enough, which is a contradiction.

Again by \cite[Lemma 15.8]{Bam20c}, we have $(y_i, t_0)\in P^*(p_i, 0; A+1, -T-\eta)$ for all $i$ large enough. Using Proposition \ref{viprop3}, we have $v_i(y_i, t_0)\leq C(A, T)$, hence 
$$
v_i(y_i, t)\leq C(A, T) ,
$$
for all $t\in [-T-\eta, 0]$. Hence we have
$$
R_{g_i}(y_i, t)\leq C(A, T) ,
$$
for all $t\in [-T-\eta, 0]$. Hence we can apply Lemma \ref{W1balongwl} to obtain that
\begin{equation}\label{W1dbyi}
\begin{split}
d^{g_{i, -T-\eta}}_{W_1} (\nu_{y_i, t_0; -T-\eta} , \nu_{y_i, t_0-r^2; -T-\eta} ) \leq d^{g_{i, t_0-r^2}}_{W_1} (\nu_{y_i, t_0; t_0-r^2} , \delta_{y_i} ) \leq C(A, T) r\leq 1,
\end{split}
\end{equation}
if $\Bar{r}$ is small enough. Hence we have $(y_i, t_0-r^2)\in P^*(p_i, 0; A+2, -T-\eta)$ for $i$ large enough. Hence we can apply Proposition \ref{icbonflow} to obtain that
\begin{equation}\label{icbofWdelta}
|\psi_i(W_{\delta})|_{g_{i, t_0-r^2}} \leq |W_i|_{g_{i, t_0-r^2}} \leq E_{\mathbf{p}}s^{\mathbf{p}}r^{n},
\end{equation}
if $r\leq \Bar{r}=\Bar{r}(A, B, T, \mathbf{p}, t_0)$, $i$ large enough, where $E_{\mathbf{p}}=E_{\mathbf{p}}(A, B, T)<\infty$. Again, since $W_{\delta}\subset\cR_{t_0-r^2}$ is relatively compact, we have $\psi_i^*g_{i, t_0-r^2}\to g_{t_0-r^2}$ in $C^{\infty}(W_{\delta})$-sense, hence we can pass $i\to\infty$ in (\ref{icbofWdelta}) to obtain
$$|W_{\delta}|_{g_{t_0-r^2}}\leq E_{\mathbf{p}}s^{\mathbf{p}}r^{n},$$
this holds for any $\delta>0$ with $E_{\mathbf{p}}$ being independent of $\delta$, hence we let $\delta\to 0$ to obtain (\ref{icbonlimitflow'2}). 

This completes the proof of Proposition \ref{icbonlimitflow}.
\end{proof}


\medskip
\subsection{Proof of Theorem \ref{mainforRF1}}

First, we need the following result.
\begin{proposition}\label{eoyinftyandtilyinfty}
For any $A<\infty$, $t_0\in(-T_\infty, 0)$, there exists constants $\Bar{r}>0$, $\lambda>0$, $C<\infty$, all depend on $t_0, A$, such that the following statement holds.
Suppose $x_\infty\in\cX_{t_0}$ satisfies that
$$
d^{\chi_{t_0}}_{W_1} (\nu_{p_\infty; t_0} , \delta_{x_\infty} ) \leq A .
$$
Then for any $0<r\leq\Bar{r}$, there exists a point $\tilde y_\infty\in \cR_{t_0+(\lambda r)^2}$, a point $ y_\infty\in \cR_{t_0}$, $\tilde y_\infty$ is the point survive from $y_\infty$, such that
\begin{enumerate}
\item $\tilde y_\infty\in P^*(p_\infty ; C, t_0-\Bar{r}^2)$;
\item $B(x_\infty,r)\subset P^*(\tilde y_\infty ; Cr, -(\lambda r)^2)$;
\item $d(x_\infty, y_\infty)\leq Cr$, $r_{\Rm}( y_\infty)\geq C^{-1}r$ and $|B( y_\infty, r)\cap \cR_{t_0}|_{g_{t_0}}\geq C^{-1}r^{n}$.
\end{enumerate}
\end{proposition}
\begin{proof}
Throughout the proof, unless otherwise stated, all the constants will depend at most on $t_0, A$. First, we choose $\eta\in (0,1)$ such that $t_0-100\eta> -T_\infty$. We then require that $\Bar{r}\leq \eta$.

Since $\cR_{t_0}$ is a dense subset of $\cX_{t_0}$, we can find a point $z_\infty\in B(x_\infty,r)\cap \cR_{t_0}$. Using \cite[Theorem 6.45]{Bam20b}, we can find points $(z_i, t_0)\in M_i\times \left\{t_0\right\}$ such that $(z_i, t_0)\xrightarrow[i\to\infty]{\CCC} z_\infty$. Since
$$
d^{\cX_{t_0}}_{W_1} ( \delta_{z_\infty} , \nu_{p_\infty; t_0} )  \leq d^{\cX_{t_0}}_{W_1} ( \delta_{z_\infty} , \delta_{x_\infty} ) + d^{\cX_{t_0}}_{W_1} ( \delta_{x_\infty} , \nu_{p_\infty; t_0} ) \leq A+1 ,
$$
from \cite[Lemma 15.8]{Bam20c}, we have $(z_i, t_0)\in P^*(p_i, 0; A+2, t_0-\eta)$ for $i$ large enough. According to Proposition \ref{viprop3}, we have $v_i(z_i, t_0)\leq C$, hence
\begin{equation}\label{viboundatzi}
v_i(z_i, t) \leq C,
\end{equation}
for all $t\in [t_0-\eta, 0]$. Hence we have
\begin{equation}\label{scanddealongzi}
R_{g_i}(z_i, t) \leq C, 
\end{equation}
for all $t\in [t_0-\eta, 0]$, for $i$ large enough. We have the following lemma.

\begin{lemma}\label{iceatzi}
For any $\lambda\in (0,1)$, let $\tilde t_0:=t_0+(\lambda r)^2$, then for $i$ large enough, for any $s\in (0,1)$, we have
\begin{equation}
\left| \left\{ r_{\Rm}\leq sr \right\}\cap B_{g_i}\left(z_i,\tilde t_0, r\right) \right|_{g_{i, \tilde t_0}}   \leq Csr^{n}.    
\end{equation}
\end{lemma}
\begin{proof}
From (\ref{scanddealongzi}) and Lemma \ref{W1balongwl}, we have
\begin{equation}\label{W1dbalongzi}
d^{g_{i, \tilde t_0}}_{W_1} (\nu_{z_i, \tilde t_0 + r^2; \tilde t_0} , \delta_{z_i} ) \leq C r.    
\end{equation}
and similarly
\begin{equation}\label{W1dbalongzi'2}
d^{g_{i, t_0}}_{W_1} (\nu_{z_i, \tilde t_0 ; t_0} , \delta_{z_i} ) \leq C(A, T)\lambda r\leq 1,    
\end{equation}
for $i$ large enough. Let $(x, \tilde t_0)\in B_{g_i}\left(z_i,\tilde t_0, r\right)$, then from (\ref{W1dbalongzi}) we have
\[
\begin{split}
d^{g_{i, \tilde t_0 - r^2}}_{W_1} &( \nu_{z_i, \tilde t_0 + r^2; \tilde t_0 - r^2} , \nu_{x, \tilde t_0 ; \tilde t_0 - r^2} ) \leq d^{g_{i, \tilde t_0}}_{W_1} ( \nu_{z_i, \tilde t_0 + r^2; \tilde t_0} , \delta_{x} )\\
& \leq d^{g_{i, \tilde t_0}}_{W_1} ( \nu_{z_i, \tilde t_0 + r^2; \tilde t_0} , \delta_{z_i} ) + d_{g_{i, \tilde t_0}} (z_i, x)\leq B r,
\end{split}
\]
for some constant $B<\infty$. Hence we have
$$ B_{g_i}\left(z_i,\tilde t_0, r\right)\subset P^*(z_i,\tilde t_0 + r^2; Br, -2r^2 ).$$
Also from (\ref{W1dbalongzi'2}) we have
\[
\begin{split}
&d^{g_{i, t_0-\eta}}_{W_1} ( \nu_{z_i, \tilde t_0; t_0-\eta} , \nu_{p_i, 0 ; t_0-\eta } ) \\
&\leq d^{g_{i, t_0-\eta}}_{W_1} ( \nu_{z_i, \tilde t_0; t_0-\eta} , \nu_{z_i, t_0 ; t_0-\eta} ) + d^{g_{i, t_0-\eta}}_{W_1} ( \nu_{z_i, t_0; t_0-\eta} , \nu_{p_i, 0 ; t_0-\eta} )\\
&\leq d^{g_{i, t_0}}_{W_1} (\nu_{z_i, \tilde t_0 ; t_0} , \delta_{z_i} ) + A +2\leq A+3,
\end{split}
\]
hence we can apply Proposition \ref{icbonflow} to obtain that
\[
\begin{split}
&\left| \left\{ r_{\Rm}\leq sr \right\}\cap B_{g_i}\left(z_i,\tilde t_0, r\right) \right|_{g_{i, \tilde t_0}} \\
&\leq \left| \left\{ r_{\Rm}\leq sr \right\}\cap P^*(z_i,\tilde t_0 + r^2; Br, -2r^2 )\cap (M_i\times \left\{\tilde t_0\right\}) \right|_{g_{i, \tilde t_0}} \leq Csr^{2n},
\end{split}
\]
for some constant $C<\infty$, provided that $r<\Bar{r}$. This completes the proof.
\end{proof}
From (\ref{viboundatzi}), we have $v_i\leq C$ on $B_{g_i}\left(z_i,\tilde t_0, r\right)$, hence we have
$$
R_{g_i}(x, \tilde t_0)\leq C, ~~for ~~all ~~x\in B_{g_i}\left(z_i,\tilde t_0, r\right), 
$$
for $i$ large enough. Hence from Lemma \ref{lvnc}, we have
$$\left|B_{g_i}\left(z_i,\tilde t_0, r\right) \right|_{g_{i, \tilde t_0}}\geq C^{-1}r^{n},$$
for some $C<\infty$. Combining this with Lemma \ref{iceatzi}, we can find a constant $\delta>0$, such that there exists a point $(y_i, \tilde t_0)\in B_{g_i}\left(z_i,\tilde t_0, r\right)$ such that
\begin{equation}\label{rrmlbatyi'1}
r_{\Rm}(y_i, \tilde t_0) \geq 2\delta r.    
\end{equation}
Using Lemma \ref{deofrrm}, there exists a dimensional constant $C_0(n)<\infty$ such that
$$ |r_{\Rm}^2(y_i, t) - r_{\Rm}^2(y_i, t_0)|\leq C_0(n)\lambda^2r^2,  $$
for all $t\in [t_0, \tilde t_0]$, hence if we choose $\lambda>0$ small enough, then from (\ref{rrmlbatyi'1}) we have
\begin{equation}\label{rrmlbatyi'2}
r_{\Rm}(y_i, t) \geq \delta r,   
\end{equation}
for all $t\in [t_0, \tilde t_0]$.

From (\ref{scanddealongzi}) and Lemma \ref{W1balongwl}, we have
\begin{equation}\label{W1dbalongzi2}
d^{g_{i, t_0-\eta}}_{W_1} ( \delta_{z_i} , \nu_{z_i, t_0 ; t_0-\eta} ) \leq C \eta^{1/2} \leq 1 .   
\end{equation}
But we also have $(z_i, t_0)\in P^*(p_i, 0; A+2, t_0-\eta)$ and $d^{g_{i, t_0-\eta}}_{W_1} ( \nu_{p_i, 0 ; t_0-\eta} , \delta_{p_i} ) \leq C $, hence by triangle inequality, we have
$$
d_{g_{i, t_0-\eta}} (z_i, p_i)\leq C ,
$$
hence by Proposition \ref{stdde}, if we choose $\eta$ and $\Bar{r}$ small enough, then we have
\begin{equation}\label{W1dbalongzi3}
d_{g_{i, t}} (z_i, p_i)\leq C   
\end{equation}
for all $t\in [t_0-\eta, \tilde t_0]$. Hence
$$
d_{g_{i, \tilde t_0}} (y_i, p_i)\leq C .
$$
Hence by Proposition \ref{stdde} again, if we choose $\Bar{r}$ small enough, then we have
\begin{equation}\label{dbbeyipi}
d_{g_{i, t}} (y_i, p_i)\leq C , 
\end{equation}
for all $t\in [t_0, \tilde t_0]$. Hence we can apply Proposition \ref{copuug'1} to find points $\tilde y_\infty\in \cR_{t_0+(\lambda r)^2}$ and $y_\infty\in \cR_{t_0}$ such that (passing to a subsequence)
$$(y_i, t_0+(\lambda r)^2)\xrightarrow[i\to\infty]{\CCC} \tilde y_\infty,$$
and
$$(y_i, t_0)\xrightarrow[i\to\infty]{\CCC} y_\infty,$$
with $\tilde y_\infty$ being the point survive from $y_\infty$. From (\ref{dbbeyipi}) we have $d^{g_{i, t}}_{W_1} ( \delta_{y_i} , \nu_{p_i, 0 ; t } ) \leq C $ for all $t\in [t_0, \tilde t_0]$. Hence item (1) is clear from \cite[Lemma 15.8]{Bam20c}.

Next, we prove item (3). Since $z_\infty\in \cR_{t_0}$, we can find a constant $0<r_0\leq \lambda r$ such that $\tilde r_{\Rm}(z_\infty)\geq r_0$. From (\ref{W1dbalongzi3}), if we choose $\Bar{r}$ small enough, then we can apply Proposition \ref{copuug'1} to find a dimensional constant $c_0(n)>0$, such that the worldline of $z_\infty$ in $\cR$ survives from $\cR_{t_0}$ to $\cR_{t_0+(c_0r_0)^2}$, and for any $t\in [t_0, t_0+(c_0r_0)^2]$, if we denote by $z_\infty(t)\in \cR_t$ the point survives from $z_\infty$, then we have 
$$(z_i, t)\xrightarrow[i\to\infty]{\CCC} z_\infty(t),$$
with $\tilde r_\Rm(z_\infty(t))\geq c_0r_0$. Similar to (\ref{W1dbalongzi'2}), we have 
$$d^{g_{i, t_0+(c_0r_0)^2}}_{W_1} (\nu_{z_i, t_0 + (\lambda r)^2 ; t_0+(c_0r_0)^2} , \delta_{z_i} ) \leq C r,$$
hence we have
\[
\begin{split}
&\qquad d^{g_{i, t_0+\frac{1}{2}(c_0 r_0)^2}}_{W_1} ( \nu_{y_i, t_0+(\lambda r)^2; t_0+\frac{1}{2}(c_0 r_0)^2} , \nu_{z_i, t_0+(c_0 r_0)^2; t_0+\frac{1}{2}(c_0 r_0)^2} ) \\
&\leq d^{g_{i, t_0+\frac{1}{2}(c_0 r_0)^2}}_{W_1} ( \nu_{y_i, t_0+(\lambda r)^2; t_0+\frac{1}{2}(c_0 r_0)^2} , \nu_{z_i, t_0+(\lambda r)^2; t_0+\frac{1}{2}(c_0 r_0)^2} )\\
&\qquad\qquad\qquad\qquad+ d^{g_{i, t_0+\frac{1}{2}(c_0 r_0)^2}}_{W_1} ( \nu_{z_i, t_0+(\lambda r)^2; t_0+\frac{1}{2}(c_0 r_0)^2} , \nu_{z_i, t_0+(c_0 r_0)^2; t_0+\frac{1}{2}(c_0 r_0)^2} )\\
&\leq d_{g_{i,t_0+(\lambda r)^2}}(y_i, z_i) + d^{g_{i, t_0+(c_0 r_0)^2}}_{W_1} ( \nu_{z_i, t_0+(\lambda r)^2; t_0+(c_0 r_0)^2} , \delta_{z_i} )\leq C r,
\end{split}
\]
hence from \cite[Lemma 15.8]{Bam20c}, we have 
\begin{equation}\label{W1debzinfyinf}
d^{\cX_{t_0}}_{W_1} ( \nu_{\tilde y_\infty; t_0} , \nu_{z_\infty(t_0+(c_0 r_0)^2); t_0} )\leq C r.    
\end{equation}

If we consider the spacetime curve $\gamma(\tau)=z_\infty(t_0+(c_0r_0)^2-\tau)$ with $\tau\in [0, (c_0r_0)^2]$, then $\tilde r_\Rm(\gamma(\tau))\geq c_0r_0$, hence we have $\cL(\gamma)\leq c_0r_0$, then from \cite[Lemma 22.2]{Bam20c}, we have 
\begin{equation}\label{W1debzinf}
d^{\cX_{t_0}}_{W_1} ( \delta_{z_\infty} , \nu_{z_\infty(t_0+(c_0 r_0)^2); t_0} )\leq C r_0\leq C r.    
\end{equation}
According to Proposition \ref{copuug'1}, we can choose $\lambda=\lambda(A,T)$ even smaller, such that the worldline of $y_\infty$ in $\cR$ survives from $\cR_{t_0}$ to $\tilde y_\infty\in\cR_{t_0+(\lambda r)^2}$, and for any $t\in [t_0, t_0+(\lambda r)^2]$, we have $(y_i, t)\xrightarrow[i\to\infty]{\CCC}  y_\infty(t)$. Hence $\tilde r_\Rm( y_\infty(t))\geq \delta r$ from (\ref{rrmlbatyi'2}) for all $t\in [t_0, t_0+(\lambda r)^2]$. Similarly to (\ref{W1debzinf}), we can prove
\begin{equation}\label{W1debyinf}
d^{\cX_{t_0}}_{W_1} ( \delta_{y_\infty} , \nu_{\tilde y_\infty; t_0} )\leq C r.    
\end{equation}
Now we can combine (\ref{W1debzinfyinf}) (\ref{W1debzinf}) and (\ref{W1debyinf}) with the triangle inequality to obtain 
$$d_{t_0}(z_\infty,  y_\infty)\leq C r.$$
Hence $d_{t_0}(x_\infty, y_\infty)\leq C r$. 

Also, there is a constant $c_0=c_0(n)>0$, such that $P(y_i, t_0; c_0\delta r)$ converge in the Cheeger-Gromov sense to $P^{\circ}( y_\infty; c_0\delta r, (c_0\delta r)^2, -(c_0\delta r)^2)\subset\cR$, which is unscathed. But $B_{g_{i,t_0}}(y_i, t_0, c_0\delta r)\subset P(y_i, t_0; c_0\delta r)$ and $|B_{g_{i,t_0}}(y_i, t_0, \frac{1}{2}c_0\delta r)|_{g_{i,t_0}}\geq C^{-1}r^{n}$, hence we have $|B( y_\infty, c_0\delta r)|_{g_{t_0}}\geq C^{-1}r^{n}$. This proves item (3).

Finally, we prove item (2). Let $x\in B(x_\infty,r)$, hence by item (3) we have
$$
d_{t_0}( x , y_\infty ) \leq d_{t_0}( x , x_\infty ) + d_{t_0}( x_\infty , y_\infty )\leq C r ,
$$
hence by (\ref{W1debyinf}), we can compute
\[
\begin{split}
d^{\cX_{t_0}}_{W_1} ( \delta_{x} , \nu_{\tilde y_\infty ; t_0 } ) \leq d^{\cX_{t_0}}_{W_1} ( \delta_{x} , \delta_{y_\infty} ) + d^{\cX_{t_0}}_{W_1} ( \delta_{y_\infty} , \nu_{\tilde y_\infty ; t_0 } ) \leq Cr ,
\end{split}
\]
this proves item (2).

This completes the proof of Proposition \ref{eoyinftyandtilyinfty}.
\end{proof}
Now we can prove Theorem \ref{mainforRF1}.
\begin{proof}[Proof of Theorem \ref{mainforRF1}]
We fix a time $t_0\in (-T_\infty, 0)$. Let $\eta\in ( 0 , 1 )$ be a small constant such that
$t_0-100\eta>-T_\infty$.




%
For item (1), we only need to verify condition (4) in Definition \ref{sp}. Let $K\subset \cX_{t_0}$ be any compact subset, then there is a constant $A=A(K)<\infty$ such that $K\subset P^*( p_{\infty} ; A ,t_0 )$. 

Hence for any $x_{\infty}\in K$, we can find constants $0<\Bar{r}=\Bar{r}(K, t_0)<\eta$, $0<\lambda=\lambda(K, t_0)<1$, $C=C(K, t_0)<\infty$, and for any $r\in (0, \Bar{r})$, we can find points $\tilde y_\infty\in \cR_{t_0+(\lambda r)^2}$, $y_\infty\in \cR_{t_0}$, such that the statements of Proposition \ref{eoyinftyandtilyinfty} hold. Hence we have
\begin{equation}\label{vnceatxinfty}
|B(x_\infty, 2Cr)\cap \cR_{t_0}|_{g_{t_0}}\geq |B( y_\infty, r)\cap \cR_{t_0}|_{g_{t_0}}\geq  C^{-1}r^{n}\geq  \tilde C^{-1}(2Cr)^{n},    
\end{equation}
for some constant $\tilde C=\tilde C(K, t_0)<\infty$, this proves the volume non-collapsing estimate in condition (4) in Definition \ref{sp}.

Next, for the volume non-inflating estimate, we let $ r\in (0, \Bar{r})$. We have $B(x_\infty,r)\subset P^*(\tilde y_\infty ; Cr, -(\lambda r)^2)$. For any $\delta>0$, we set $W_\delta:=\left\{ \tilde r_{\Rm}\geq 2\delta \right\}\cap P^*(\tilde y_\infty ; Cr, -\frac{3}{2}(\lambda r)^2) \cap \cR_{t_0}$. According to \cite[Theorem 6.45]{Bam20b}, we can find points $y_i\in M_i$ such that $(y_i, t_0+(\lambda r)^2)\xrightarrow[i\to\infty]{\CCC} \tilde y_\infty$. Then similar to the proof of Proposition \ref{icbonlimitflow}, we have $W_{\delta}\subset\cR_{t_0}$ is relatively compact, $W_{\delta}\subset U_i$ for $i$ large enough, and $\psi_i(W_{\delta})\subset V_i$ are well-defined, and if we let 
$$W_{i}:=\left\{ r_{\Rm}\geq \delta \right\}\cap P^*\left(y_i, t_0+(\lambda r)^2; 2C r , -10\lambda^2 \right)\cap (M_i\times \left\{t_0\right\}),$$
then we have $\psi_i(W_{\delta})\subset W_i$ for $i$ large enough. Since $(t_0+(\lambda r)^2) - 20 \lambda^2 r^2 >-T_{\infty} $ and $t_0\in [(t_0+(\lambda r)^2) - 10 \lambda^2 r^2 , (t_0+(\lambda r)^2)]$, according to \cite[Theorem 9.8]{Bam20a}, we have 
$$
|\psi_i(W_{\delta})|_{g_{i,t_0}}\leq|W_i|_{g_{i,t_0}}\leq \tilde C(K, t_0) r^{n} ,
$$
for all $i$ large enough, hence by the smooth convergence of $\psi^*g_{i, t_0}$ to $g_{t_0}$ on $W_{\delta}$, we have $|W_{\delta}|_{g_{t_0}}\leq\tilde C(K, t_0) r^{n}$, letting $\delta\to 0$, we obtain $|B(x_\infty,  r)\cap\cR_{t_0}|_{g_{t_0}}\leq\tilde C(K, t_0) r^{n}$. This proves the volume non-inflating estimate in condition (4) in Definition \ref{sp}, and proves item (1).

Finally, we prove item (2), we consider the case $K=\left\{x_\infty\right\}$. Since $\tilde y_\infty\in P^*(p_\infty ; C, t_0-\Bar{r}^2)$ with $\mathfrak{t}(\tilde y_\infty)=t_0+(\lambda r)^2\leq t_0/2$ if we choose $\Bar{r}$ small enough, we can apply Proposition \ref{icbonlimitflow} to obtain that, for any $0<s<1$, $0<\mathbf{p}<4$, 
$$\left| \left\{ 0<\tilde r_{\Rm}\leq sr \right\}\cap P^*\left(\tilde y_\infty; Cr, -\frac{3}{2}(\lambda r)^2\right)\cap \cR_{t_0} \right|_{g_{t_0}}   \leq E_{\mathbf{p}}s^{\mathbf{p}}r^{n},$$
for some $E_{\mathbf{p}}=E_{\mathbf{p}}(x_\infty, t_0)<\infty$, hence from $B(x_\infty, r)\subset P^*\left(\tilde y_\infty; Cr, -\frac{3}{2}(\lambda r)^2\right)$, we have
$$\left| \left\{ 0< \tilde r_{\Rm}\leq sr \right\}\cap B(x_\infty, r)\cap \cR_{t_0} \right|_{g_{t_0}}   \leq E_{\mathbf{p}}s^{\mathbf{p}}r^{n}.$$
Combining with item (1), we have $(\cX_{t_0}, d_{t_0},\cR_{t_0}, g_{t_0})$ is a singular space of dimension $n$, which has singularities of codimension $4$ in the sense of \cite[Definition 1.9]{Bam17}. This proves item (2).
\end{proof}
\begin{proof}[Proof of Corollary \ref{CorforRF1}]
Since $v_i$ is a $C_i$-barrier of $R_{g_i}$ and $Y_0$-based at $(p_i, 0)$, and $C_i\to 0$ as $i\to\infty$, we have $\lim_{i\to\infty} R_{g_i}(p_i, 0)\leq \lim_{i\to\infty} C_iY_0=0$. If we have $\lim_{i\to\infty} \inf_{M_i\times \left\{-T_i\right\} } R_{g_i}\geq 0$, then the same computation as (\ref{scbatxk2}) and \cite[Theorem 15.60]{Bam20c} show that $\cX$ is a static limit. This implies that $\cX$ is continuous in the Gromov-$W_1$ sense, hence by \cite{Bam20b}, for every $t\in (-T_\infty, 0)$ the $\bF$-convergence (\ref{FcofRF'1}) is time-wise. By Theorem \ref{FtoGHfromJST}, this completes the proof.
\end{proof}
%


\medskip
\subsection{Tangent spaces of the limiting space}

In this section, we have the following corollary for the tangent spaces of the limiting singular space.
\begin{corollary}\label{CorforRF2}
Suppose we have $C_i\leq Y_0$ for all $i$. Then for every $t\in (-T_\infty, 0)$ where the $\bF$-convergence (\ref{FcofRF'1}) is time-wise, the following statements hold.
\begin{enumerate}
\item For any sequence of scales $\sigma_k\to 0^+$, for any $q\in \cX_t$, by passing to a subsequence, we have
\begin{equation}\label{GHconofchi'1}
(\cX_t, \sigma_k^{-1}d_{t}, q) \xrightarrow[k\to\infty]{}  ( \tilde\cX_{-1}, d^{\tilde\cX}_{-1} , \tilde q ),
\end{equation}
in the pointed Gromov-Haudorff sense for some $ \tilde  q \in \cX_{-1}$, where $\tilde\cX_{-1}$ is the $-1$-time-slice of a static metric flow $\tilde\cX$, which is a limit arising as in Corollary \ref{CorforRF1}.
\item For any sequence of scales $\eta_k\to 0^+$, by passing to a subsequence, we have
\begin{equation}\label{GHconofchi'2}
(\tilde\cX_{-1}, \eta_k^{-1}d^{\tilde\cX}_{-1}, \tilde q ) \xrightarrow[k\to\infty]{}  ( \hat\cX_{-1}, d^{\hat\cX}_{-1} , \hat q ),
\end{equation}
in the pointed Gromov-Haudorff sense, where $\hat\cX_{-1}$ is the $-1$-time-slice of a static metric flow $\hat\cX$, which is a limit arising as in Corollary \ref{CorforRF1}. Moreover, $( \hat\cX_{-1}, d^{\hat\cX}_{-1} )$ is a metric cone.
\end{enumerate}
\end{corollary}
\begin{proof}
Fix a time $t_0\in (-T_\infty, 0)$ where the where the $\bF$-convergence (\ref{FcofRF'1}) is time-wise. We choose $\eta\in (0,1)$ such that $t_0-100\eta> -T_\infty$. Fix a point $q\in \cX_{t_0}$. Let $\sigma_k\to 0^+$ be any blow-up scales. We want to study the blow-up sequence $(\cX_{t_0}, \sigma_k^{-1}d_{t_0}, q)$. 

Since (\ref{FcofRF'1}) is time-wise at $t_0$, by the proof of \cite[Theorem 7.3]{JST23a}, the condition of \cite[Proposition 2.7]{Hal} hold for $(M_{i}, d_{g_{i,t_0}}, \nu_{p_i,0; t_0})$, with base-point $(p_i, t_0)$, which converge to $(\cX_{t_0}, d_{t_0}, \nu_{p_\infty; t_0})$ in the Gromov-$W_1$-Wasserstein sense. Hence, we can find points $q_i^*\in M_i$, such that $(q_i^*, t_0)$ strictly converge to $q$ within $\CCC$ and 
\begin{equation}\label{pGHofMigitocX}
(M_{i}, d_{g_{i,t_0}},  q_i^*) \xrightarrow[i\to\infty]{}  ( \cX_{t_0}, d^{\cX}_{t_0} ,  q ),
\end{equation}
in the pointed Gromov-Haudorff sense.


%
We need to do blow-up at the time $t_0+\sigma_k^2$. By Proposition \ref{eoyinftyandtilyinfty}, passing to a subsequence (such that all $\sigma_k>0$ small enough, depending on $t_0, q$), for each $k$, there exists a point $\tilde q_k \in P^*(p_\infty ; C, t_0-\eta)\cap\cX_{t_0+\sigma_k^2}$ such that
\begin{equation}\label{W1dbqtildeqk}
d^{\cX_{t_0}}_{W_1} ( \delta_{q} , \nu_{\tilde q_k; t_0} )  \leq  C\sigma_k .
\end{equation}
Since (\ref{FcofRF'1}) is time-wise at $t_0$, by \cite[Theorem 6.45]{Bam20b}, for each $k$, there exists $q^i_k \in M_i$, such that
\begin{equation}\label{W1conofqiktoqk}
(q^i_k, t_0+\sigma_k^2)\xrightarrow[i\to\infty]{\CCC} \tilde q_k ,
\end{equation}
where is time-wise at $t_0$.  

\begin{lemma}\label{debeqistarq}
There exists a constant $C=C(t_0, q)<\infty$, such that for each $k$, there exists $i_0(k)<\infty$, such that for all $i\geq i_0(k)$, we have
\begin{equation}\label{debeqistarq2}
d_{g_{i,t_0}} ( q_i^* , q^i_k ) \leq C\sigma_k  .
\end{equation}
\end{lemma}
\begin{proof}
By \cite[Lemma 15.8]{Bam20c}, for all $i\geq i_0(k)$ with $i_0(k)$ large enough, we have
$$
( q^i_k , t_0+\sigma_k^2 ) \in P^*(p_i, 0 ; 2C, t_0-2\eta) .
$$
Hence from $v_i(p_i, 0)\leq Y_0$ and Proposition \ref{viprop3}, we have $v_i( q^i_k , t_0+\sigma_k^2 )\leq C(t_0, q)$, hence
\begin{equation}\label{viboundatqik}
v_i( q^i_k , t )\leq C(t_0, q) ,
\end{equation}
for all $t\in [t_0-2\eta , 0]$, hence 
\begin{equation}\label{Rgiboundatqik}
R_{g_i}( q^i_k , t )\leq C(t_0, q) ,
\end{equation}
for all $t\in [t_0-2\eta , 0]$.

Now, since $(q_i^*, t_0)$ strictly converge to $q$ within $\CCC$, we have
\begin{equation}\label{W1bbqistarq}
\begin{split}
d^{Z_{t_0}}_{W_1} ( (\varphi^i_{t_0})_*(\delta_{q_i^*}) , (\varphi^{\infty}_{t_0})_*(\delta_{q}) ) \leq \sigma_k ,  
\end{split}
\end{equation}
for all $i\geq i_0(k)$ with $i_0(k)$ large enough. Next, by (\ref{W1dbqtildeqk}), we have
\begin{equation}\label{W1bbtilqkq}
\begin{split}
d^{Z_{t_0}}_{W_1} ( (\varphi^{\infty}_{t_0})_*(\delta_{q}) , (\varphi^{\infty}_{t_0})_*( \nu_{\tilde q_k; t_0} ) ) \leq d^{\cX_{t_0}}_{W_1} ( \delta_{q} , \nu_{\tilde q_k; t_0} ) \leq C\sigma_k .  
\end{split}
\end{equation}
Next, since (\ref{W1conofqiktoqk}) is time-wise at $t_0$, we have
\begin{equation}\label{W1bbtilqkqik}
\begin{split}
d^{Z_{t_0}}_{W_1} ( (\varphi^{\infty}_{t_0})_*( \nu_{\tilde q_k; t_0} ) , (\varphi^{i}_{t_0})_*( \nu_{q^i_k , t_0+\sigma_k^2 ; t_0} ) ) \leq \sigma_k ,  
\end{split}
\end{equation}
for all $i\geq i_0(k)$ with $i_0(k)$ large enough. Finally, from (\ref{Rgiboundatqik}), we have
\begin{equation}\label{W1bbtilqik}
\begin{split}
d^{Z_{t_0}}_{W_1} ( (\varphi^{i}_{t_0})_*( \nu_{q^i_k , t_0+\sigma_k^2 ; t_0} ) , (\varphi^{i}_{t_0})_*( \delta_{q^i_k} ) )  \leq C \sigma_k .  
\end{split}
\end{equation}
Now, by the triangle inequality and (\ref{W1bbqistarq}), (\ref{W1bbtilqkq}), (\ref{W1bbtilqkqik}), (\ref{W1bbtilqik}), we have
\begin{equation}\label{W1bbtilqikqstar}
\begin{split}
d^{Z}_{t_0} ( (\varphi^i_{t_0})(q_i^*) , (\varphi^{i}_{t_0})( q^i_k ) ) = d^{Z_{t_0}}_{W_1} ( (\varphi^i_{t_0})_*(\delta_{q_i^*}) , (\varphi^{i}_{t_0})_*( \delta_{q^i_k} ) )  \leq C \sigma_k ,  
\end{split}
\end{equation}
for all $i\geq i_0(k)$ with $i_0(k)$ large enough. Since $\varphi^i_{t_0}: (M_i, d_{g_{i, t_0}})\to (Z_{t_0}, d^Z_{t_0})$ is isometric embedding, we conclude
$$
d_{g_{i,t_0}} ( q_i^* , q^i_k ) \leq C\sigma_k ,
$$
for all $i\geq i_0(k)$ with $i_0(k)$ large enough. This completes the proof.
\end{proof}
From (\ref{pGHofMigitocX}), we can choose $i_0(k)$ large enough, such that for all $i\geq i_0(k)$, we have 
\begin{equation}\label{pGHofMigitocX2}
d_{ PGH }((M_{i}, d_{g_{i,t_0}},  q_i^*) ,  ( \cX_{t_0}, d^{\cX}_{t_0} ,  q ))\leq \sigma_k^3 .
\end{equation}
Now, for each $k\in\mathbb{N}$, we choose $i_0(k)$ large enough, such that Lemma \ref{debeqistarq} and (\ref{pGHofMigitocX2}) hold. Then we set
$$
\tilde M_k:=M_{i_0(k)},~~~ \tilde g_{k,t}:=\sigma_k^{-2}g_{ i_0(k) , \sigma_k^{2}t + ( t_0+\sigma_k^2 )  },~~~ x_k:=q^{i_0(k)}_{k}, ~~~ t\in [-\tilde T_k , 0] ,
$$
where $\tilde T_k=\sigma_k^{-2}( T_{i_0(k)} - \eta + t_0+\sigma_k^2 )\to\infty$ as $k\to\infty$. We also set
$$
\tilde q^*_k:=q^*_{i_0(k)} ,
$$
then from Lemma \ref{debeqistarq}, we have
\begin{equation}\label{debeqistarq3}
d_{\tilde g_{k,-1}} ( \tilde q^*_k , x_k ) \leq C  .
\end{equation}
Also, from (\ref{pGHofMigitocX2}), we have
\begin{equation}\label{pGHofMigitocX3}
d_{ PGH }( ( \tilde M_k , d_{\tilde g_{k,-1}} ,  \tilde q^*_k ) ,  ( \cX_{t_0}, \sigma_k^{-1}d^{\cX}_{t_0} ,  q ))\leq \sigma_k .
\end{equation}
Next, recall that $v_i$ is $C_i$-barrier of $R_{g_i}$, hence if we set $\tilde v_k(t)=v_{i_0(k)}(\sigma_k^{2}t + ( t_0+\sigma_k^2 ))$, $t\in [-\tilde T_k , 0]$, then from (\ref{viboundatqik}), $\tilde v_k$ is $\sigma_k^{2}Y_0$-barrier of $R_{\tilde g_{k}}$ and $C$-based at $(x_k, 0)$. Moreover, due to the choice of $\tilde T_k$, we have $R_{\tilde g_{k}} \geq -C(t_0, q)\sigma_k^{2}$. Hence, Corollary \ref{CorforRF1} applied here, that is, by passing to a subsequence, we have $\bF$-convergence on compact time-intervals
\begin{equation}\label{FcofRF'2}
(\tilde M_k, (\tilde g_{k,t})_{t\in [-\tilde T_k , 0]}, (\nu_{x_k,0; t})_{t\in [-\tilde T_k , 0]}) \xrightarrow[i\to\infty]{\bF,\tilde\CCC}  (\tilde\cX, (\nu_{x_\infty; t})_{t\in (-\infty , 0]}),
\end{equation}
within some correspondence $\tilde\CCC$, with $\tilde\cX$ being a static limit, satisfies the conclusions of \cite[Theorem 2.16]{Bam20c}. Moreover, since $\tilde\cX$ is static, it's continuous on $( -\infty , 0 )$ in the sense of \cite[Definition 4.25]{Bam20b}, hence by \cite[Theorem 7.6]{Bam20b}, the $\bF$-convergence (\ref{FcofRF'2}) is time-wise at every $t\in ( -\infty , 0 )$. Hence we can apply Theorem \ref{FtoGHfromJST} to conclude that,
\begin{equation}\label{pGHofMigitocX4}
d_{ PGH }( ( \tilde M_k , d_{\tilde g_{k,-1}} ,  \tilde q^*_k ) ,  ( \tilde \cX_{-1}, d^{\tilde \cX}_{-1} , \tilde  q ))\to 0 ,
\end{equation}
as $k\to\infty$ for some $ \tilde  q \in \cX_{-1}$. Note that here we have used (\ref{debeqistarq3}). Combining (\ref{pGHofMigitocX3}) and (\ref{pGHofMigitocX4}), we conclude that 
$$
d_{ PGH }( ( \cX_{t_0}, \sigma_k^{-1}d^{\cX}_{t_0} ,  q ) ,  ( \tilde \cX_{-1}, d^{\tilde \cX}_{-1} , \tilde  q ))\to 0 ,
$$
which proves item (1). Finally, due to \cite[Theorem 2.16]{Bam20c}, every tangent cone of $( \tilde \cX_{-1}, d^{\tilde \cX}_{-1} )$ at any point is a metric cone, repeating the above arguments prove item (2).

This completes the proof.
\end{proof}
As a consequence, we have the following corollary.
\begin{corollary}\label{cortang}
Under the set-up of Theorem \ref{mainforKRF1}, for every $t\in (-\infty, 0]$, the conclusions of Corollary \ref{CorforRF2} hold.
\end{corollary}


\medskip
\subsection{Proof of Theorem \ref{mainforKRF2}}
In this subsection, we come back to the set-up of Section \ref{setupandpre}. Under this K\"ahler-Ricci flow set-up, we actually have stronger results than Corollary \ref{cortang} due to our distance distortion estimates.

\begin{proposition}\label{tangentcmc}
Under the set-up of Theorem \ref{mainforKRF1}, for any $t_{0}\in(-\infty,0)$, and any $x_{0}\in\mathcal{X}_{t_{0}}$,
any tangent cone of $(\mathcal{X}_{t_{0}},d_{t_{0}})$ at $x_{0}$
is a metric cone.
\end{proposition}

\begin{proof}
Let $\mathcal{X}^{\lambda,t_{0}}$ denote the parabolic rescaling
of $\mathcal{X}$ by $\lambda$, based at time $t_{0}$; in particular,
we have $(\mathcal{X}_{t}^{\lambda,t_{0}},d_{t}^{\lambda,t_{0}})=(\mathcal{X}_{t_{0}+\lambda^{-2}t},\lambda d_{t_{0}+\lambda^{-2}t})$.
Then, for any sequence $\lambda_{i}\to\infty$, we can pass to a further
subsequence so that
\[
(\mathcal{X}^{\lambda_{i},t_{0}},(\nu_{t_{0}+\lambda^{-2}t})_{t\in(-\infty,0]})\xrightarrow[i\to\infty]{\mathbb{F}}(\mathcal{Y},(\mu_{t})_{t\in(-\infty,0]})
\]
where $(\mathcal{Y},(\mu_{t}))$ is a static cone, which is itself
a rescaled limit of the original flow $(M,(g_{t})_{t\in[0,1)})$.
Let $(C(Y),d_{Y},(\nu_{x;t}')_{x\in C(Y),t\in(-\infty,0]})$ be the
model of the metric cone, with $y_{\ast}$ the vertex. Because the $\mathbb{F}$-convergence 
$$(M_i, (g_{i,t})_{t\in [-T_i , 0]}, (\nu_{p_i,0; t})_{t\in [-T_i , 0]}) \xrightarrow[i\to\infty]{\bF,\CCC}  (\cX, (\nu_{p_\infty; t})_{t\in (-\infty , 0]})$$
is time-wise at every time $t\in (-\infty,0)$, Claim \ref{nonvanofvol} implies that $(\mathcal{X}_t^{\lambda_i,t_0},d_t^{\lambda_i,t_0},z_t)$ satisfy the hypotheses of \cite[Proposition 2.7]{Hal}, so after passing to a further subsequence, we have
\[
(\mathcal{X}_{t}^{\lambda_{i},t_{0}},d_{t}^{\lambda_{i},t_{0}},z_{t})\xrightarrow{i\to\infty}(C(Y),d_{Y},y_{t})
\]
in the pointed Gromov-Hausdorff sense, for any fixed $t\in(-\infty,0)$,
where $z_{t}\in\mathcal{X}_{t_{0}+\lambda_{i}^{-2}t}$ is an $H_{2n}$-center
of $x_{0}$, and $y_{t}\in C(Y)$ is a $2H_{2n}$-center of $y_{\ast}$.
The distance distortion estimates and the proof of Claim 3.15 then
imply that 
\[
d_{PGH}\left((\mathcal{X}_{t}^{\lambda_{i},t_{0}},d_{t}^{\lambda_{i},t_{0}},z_{t}),(\mathcal{X}_{t_{0}},\lambda_{i}d_{t_{0}},x_{0})\right)\leq\epsilon(t),
\]
where $\lim_{t\nearrow0}\epsilon(t)=0$. By choosing the original
sequence $\lambda_{i}$ so that $(\mathcal{X}_{t_{0}},\lambda_{i}d_{t_{0}},x_{0})$
converges in the pointed Gromov-Hausdorff sense to a given tangent
cone $(\widehat{X},\widehat{d},\widehat{x})$ of $(\mathcal{X}_{t_{0}},d_{t_{0}})$
at $x_{0}$, we obtain
\begin{align*}
d_{PGH}&\left((C(Y),d_{Y},y_{t}),(\widehat{X},\widehat{d},\widehat{x})\right) & \\ \leq & \limsup_{i\to\infty}\left(d_{PGH}\left((\mathcal{X}_{t}^{\lambda_{i},t_{0}},d_{t}^{\lambda_{i},t_{0}},z_{t}),(\mathcal{X}_{t_{0}},\lambda_{i}d_{t_{0}},x_{0})\right)+d_{PGH}\left((\mathcal{X}_{t_{0}},\lambda_{i}d_{t_{0}},x_{0}),(\widehat{X},\widehat{d},\widehat{x})\right)\right)\\
\leq & \epsilon(t).
\end{align*}
Moreover, because $y_{t}$ is an $2H_{2n}$-center of $y_{\ast}$
and $\mathcal{Y}$ is a static cone, it follows that $d_{Y}(y_{t},y_{\ast})\leq C|t|$,
so that $(C(Y),d_{Y},y_{t})\to(C(Y),d_{Y},y_{\ast})$ in the pointed
Gromov-Hausdorff sense as $t\nearrow0$. In particular, taking $t\nearrow0$
above tells us that $(C(Y),d_{Y},y_{\ast})$ is pointed isometric
to $(\widehat{X},\widehat{d},\widehat{x})$. 
\end{proof}

We can finish the Proof of Theorem \ref{mainforKRF2}.

\begin{proof}[Proof of in Theorem \ref{mainforKRF2}]
Let $(M_i, (g_{i,t})_{t\in [-T_i , 0]})$ be the sequence of Ricci flows defined in Section \ref{setupandpre}. According to Lemma \ref{viprop}, we conclude that $v_i$ is a $C$-barrier of $R_{g_{i}}$ and is $2B_0$-based at $ (p_i, 0) $, where $v_i$ is the Ricci potential defined in (\ref{cordovsi}). Hecne Theorem \ref{mainforKRF2} follows immediately from Theorem \ref{mainforRF1} and Proposition \ref{tangentcmc}.
\end{proof}


\bigskip
\bigskip

\noindent{\bf Acknowledgements} The authors would like to thank Richard Bamler, Yalong Shi and Zhenlei Zhang for many inspiring discussions. The second named author thanks Xiaochun Rong, Zhenlei Zhang and Kewei Zhang for hospitality and providing an excellent environment during his visits to Capital Normal University and Beijing Normal University where part of this work was carried out.  

\bigskip
\bigskip


\end{document}